\documentclass[11pt]{article}
\usepackage[margin=0.75in]{geometry}
\usepackage[T1]{fontenc}
\usepackage{amssymb,amsmath,mathtools,amsthm,thmtools,amsfonts}
\usepackage{dsfont} 
\usepackage{comment}
\usepackage{thm-restate}

\usepackage{url} 
\usepackage{hyperref}
\hypersetup{
colorlinks,
linkcolor={red}, 
citecolor={black},
urlcolor={blue!60!black},
pdftitle={Erd\H{o}s meet Nash-Williams},
pdfauthor={Michelle Delcourt, Cicely Henderson, Thomas Lesgourgues, Luke Postle}}

\usepackage[utf8]{inputenc}
\usepackage{color}

\usepackage{enumitem}
\setlist[enumerate,1]{label={(\roman*)}}
\setlist[enumerate,2]{label={(\alph*)}}
\setlist[enumerate,3]{label={(\arabic*)}}

\setlist[itemize]{nolistsep,noitemsep, topsep=.1in}
\setlist[enumerate]{nolistsep,noitemsep, topsep=.1in}

\newcommand{\mc}[1]{\mathcal{#1}}%

\newcommand{\cB}{\ensuremath{\mathcal B}}
\newcommand{\cF}{\ensuremath{\mathcal F}}
\newcommand{\cH}{\ensuremath{\mathcal H}}
\newcommand{\cQ}{\ensuremath{\mathcal Q}}
\newcommand{\cS}{\ensuremath{\mathcal S}}

\newcommand{\Expect}[1]{\ensuremath{\mathbb E\left[#1\right]}}
\newcommand{\Prob}[1]{\ensuremath{\mathbb P\left[#1\right]}}

\DeclarePairedDelimiter{\parens}{(}{)}
\DeclarePairedDelimiter{\set}{\{}{\}}
\DeclarePairedDelimiter{\brackets}{[}{]}
\DeclarePairedDelimiter{\floor}{\lfloor}{\rfloor}

\newcommand{\cBoff}{\ensuremath{\cB_{{\rm off}}}}
\newcommand{\cBon}{\ensuremath{\cB_{{\rm on}}}}

\usepackage[noabbrev,capitalise]{cleveref}

\theoremstyle{plain}
\newtheorem{thm}{Theorem}[section]
\newtheorem{lem}[thm]{Lemma}
\newtheorem{proposition}[thm]{Proposition}

\newtheorem{cor}[thm]{Corollary}

\newtheorem{conj}[thm]{Conjecture}

\newenvironment{proofclaim}[1][\proofname]
{\proof[#1]}
{\endproof}

\declaretheorem[
  style=plain,
  name=Claim,
  within=theorem,
]{claim}

\newenvironment{lateproof}[1]
 {%
  \renewcommand{\theclaim}{\ref{#1}.\arabic{claim}}%
  \begin{proof}[Proof of~\cref{#1}]%
 }
 {\end{proof}}

\usepackage{etoolbox}
 \AtEndEnvironment{proof}{\setcounter{claim}{0}}

\theoremstyle{definition}
\newtheorem{definition}[thm]{Definition}

\newtheorem{remark}[thm]{Remark}

\crefname{equation}{}{}
\crefname{lem}{Lemma}{Lemmas}
\crefname{claim}{Claim}{Claims}
\crefname{thm}{Theorem}{Theorems}
\crefname{enumi}{}{}

\usepackage{thm-restate}

\usepackage[dvipsnames, table]{xcolor}

\newcommand{\Erdos}{Erd\H{o}s}
\newcommand{\Kuhn}{K\"{u}hn}

\usepackage{ifthen}

\usepackage{tikz}
\usepackage{calc}
\usetikzlibrary{calc,patterns,decorations.pathreplacing,backgrounds,shapes}

\tikzstyle{vertex}=[circle, draw, fill=black, inner sep=0pt, minimum size=4pt] 
\newcommand{\vertex}{\node[vertex]} 

\newcommand{\TriangleBooster}[1]
{
\begin{tikzpicture}[x=#1 cm, y=#1 cm]

    \draw [draw = red!40,fill=red!40] (10,0) -- (5,5) -- (5,3) -- cycle; 
    \draw [draw = red!40,fill=red!40] (10,0) -- (5,-3) -- (5,-5) -- cycle;
    
    \draw [draw=black, fill=red!40] (0,0) -- (5,3) -- (5,1) -- cycle;
    \draw [draw=black, fill=red!40] (10,0) -- (5,1) -- (5,-1) -- cycle;
    \draw [draw=black, fill=red!40] (0,0) -- (5,-1) -- (5,-3) -- cycle;

    \draw [draw=black, fill=yellow!40] (0,0) -- (5,5) -- (5,3) -- cycle;
    \draw [draw=black, fill=yellow!40] (10,0) -- (5,3) -- (5,1) -- cycle;
    \draw [draw=black, fill=yellow!40] (0,0) -- (5,1) -- (5,-1) -- cycle;
    \draw [draw=black, fill=yellow!40] (10,0) -- (5,-1) -- (5,-3) -- cycle;
    \draw [draw=black, fill=yellow!40] (0,0) -- (5,-3) -- (5,-5) -- cycle;

    \vertex (u) at (0,0) [label=180:$u$]{};
    \vertex (v) at (10,0) [label=0: $v$]{};
    
    \foreach \i in {1, 2, 3, 4, 5, 6}{
        \pgfmathsetmacro\y{7-2*\i}
        \vertex (b\i) at (5,\y) [label=0:$b_{\i}$]{};
    }

    \begin{scope}[on background layer]
        \path (b1) edge[line width= 1pt, dashed, fill =red!40, out=180, in=180, looseness=2] (b6);

    \end{scope} 
        \path (b1) edge[line width= 1pt, dashed] (v);
        \path (b6) edge[line width= 1pt, dashed] (v);

    \draw [draw = black, fill = red!40] (0,-6) rectangle (1,-6.5);
    \node at (1,-6.25) [label=0:$\cBon$]{};

    \draw [draw = black, fill = yellow!40] (3,-6) rectangle (4,-6.5);
    \node at (4,-6.25) [label=0:$\cBoff$]{};

    \draw [draw = black] (6,-6.25) -- (7,-6.25);
    \node at (7,-6.25) [label=0:$B$]{};

    \draw [line width= 1pt, dashed, draw = black] (9,-6.25) -- (10,-6.25);
    \node at (10,-6.25) [label=0:$R$]{};

\end{tikzpicture}}

\title{\Erdos{} meets Nash-Williams}

\author{
Michelle Delcourt
\thanks{Department of Mathematics, Toronto Metropolitan University (formerly named Ryerson University),
Toronto, Ontario M5B 2K3, Canada {\tt mdelcourt@torontomu.ca}. Research supported by NSERC under Discovery Grant No. 2019-04269 and a Sloan Research Fellowship.} \and 
Cicely (Cece) Henderson
\thanks{Combinatorics and Optimization Department,
University of Waterloo, Waterloo, Ontario N2L 3G1, Canada 
{\tt \{c3hender,tlesgourgues,lpostle\}@uwaterloo.ca}. Partially supported by NSERC
under Discovery Grant No. 2019-04304.} \and
Thomas Lesgourgues\footnotemark[2]
\and
Luke Postle \footnotemark[2] }
\date{\today}

\begin{document}

\maketitle

\begin{abstract}
In 1847, Kirkman proved that there exists a Steiner triple system on $n$ vertices (equivalently a triangle decomposition of the edges of $K_n$) whenever $n$ satisfies the necessary divisibility conditions (namely $n\equiv 1,3 \mod 6$). In 1970, Nash-Williams conjectured that every graph $G$ on $n$ vertices with minimum degree at least $3n/4$ (for $n$ large enough and satisfying the necessary divisibility conditions) has a triangle decomposition. In 1973, \Erdos{} conjectured that for each integer $g$, there exists a Steiner triple system on $n$ vertices with girth at least $g$ (provided that $n\equiv 1,3 \mod 6$ is large enough compared to the fixed $g$). In 2021, Glock, K\"uhn, and Osthus conjectured the common generalization of these two conjectures, dubbing it the ``\Erdos{} meets Nash-Williams' Conjecture''.

In this paper, we reduce the combined conjecture to the fractional relaxation of the Nash-Williams' Conjecture. Combined with the best known fractional bound of Delcourt and Postle, this proves the combined conjecture above when $G$ has minimum degree at least $0.82733n$. We note that our result generalizes the seminal work of Barber, K\"uhn, Lo, and Osthus on Nash-Williams' Conjecture and the resolution of \Erdos{}' Conjecture by Kwan, Sah, Sawhney, and Simkin. Both previous proofs of those results used the method of iterative absorption. Our proof instead proceeds via the newly developed method of refined absorption (and hence provides new independent proofs of both results).
\end{abstract}


\section{Introduction}

The study of combinatorial designs has a rich history spanning nearly two centuries. One of the most classical theorems in all of design theory is a result of Kirkman~\cite{K47} which classifies for which $n$ there exists a set of triples of an $n$-set $X$ such that every pair in $X$ is in exactly one triple; such a set is called a \emph{Steiner Triple System}. From the graph theoretic perspective, this is equivalent to a decomposition of the edges of the complete graph $K_n$ into edge-disjoint triangles. Kirkman proved that the necessary divisibility conditions for this, namely that $3~|~\binom{n}{2}$ and $2~|~(n-1)$ which equates to $n\equiv 1,3 \mod 6$, are also sufficient. 

Arguably the most studied object in design theory is a natural generalization of this object as follows. Given integers $n > q > r \geq 1$, an $(n,q,r)$\textit{-Steiner System} is a set $S$ of $q$-subsets of an $n$-set $X$ such that each $r$-subset of $X$ is contained in exactly one element of $S$. From the hypergraph theoretic perspective, this is equivalent to a decomposition of the edges of $K_n^r$ (the complete $r$-uniform hypergraph on $n$ vertices) into edge-disjoint copies of $K_q^r$ (referred to as cliques). Once again, there are necessary divisibility conditions for the existence of an $(n, q, r)$-Steiner system: namely, for each $0 \leq i \leq r-1,$ we require $\binom{q-i}{r-i}~|~\binom{n-i}{r-i}$. Whether these divisibility conditions suffice for all large enough $n$ became a notorious folklore conjecture from the 1800s called the Existence of Combinatorial Designs Conjecture; here we require $n$ to be large enough because, for example, $n$ must satisfy Fisher's inequality. This conjecture, which for brevity we call the Existence Conjecture, was the central conjecture in design theory and was only resolved in the affirmative by Keevash~\cite{K14} nearly two centuries later. 

\begin{conj}[Existence Conjecture - proved by Keevash~\cite{K14} in 2014]
Let $q > r \ge 2$ be integers.  If $n$ is sufficiently large and $\binom{q-i}{r-i}~|~\binom{n-i}{r-i}$ for all $0 \leq i \leq r-1,$  then there exists a $(n, q, r)$-Steiner system.   
\end{conj}

The Existence Conjecture has a storied history (see Keevash~\cite{K14} or~\cite{K18survey} for example). Of particular note, in the 1970s Wilson \cite{W72-EC1, W72-EC2, W75-EC3} revolutionized design theory when he proved the `graph' case of the Existence Conjecture, namely for all values of $q$ when $r = 2$. In 1963, \Erdos{} and Hanani~\cite{EH63} conjectured an {\em approximate} version of the Existence Conjecture, positing that it is possible to find a packing of edge-disjoint copies of $K_q^r$ covering $(1-o(1))$ proportion of the edges of $K_n^r$. This problem was solved in 1985 by R\"{o}dl~\cite{R85} using his seminal ``nibble'' method.

Yet, it was not until 2014, over a century after the conjecture first appeared in the literature, that Keevash~\cite{K14} proved the Existence Conjecture, a major breakthrough in design theory. While Keevash used a combination of algebraic and combinatorial methods, additional proofs using alternate methods appeared in subsequent years. In 2016, Glock, K\"{u}hn, Lo, and Osthus~\cite{GKLO16} provided a fully combinatorial proof of the conjecture using the method of iterative absorption. In 2024, Delcourt and Postle~\cite{DPI} developed the methodology of {\em refined absorption}, giving a one-step, combinatorial proof of the Existence Conjecture. We note that using the method of refined absorption is key to proving the results in this article. Finally, later in 2024, using the framework from refined absorption, Keevash~\cite{K24} presented a fourth (shorter) proof of the Existence Conjecture.

Despite these pivotal breakthroughs, there are many important questions in design theory that remain unsolved - in particular various generalizations of the Existence Conjecture, some of which are open even for the case of Steiner triple systems. This paper concerns the combination of two such fundamental questions in design theory: the Nash-Williams' Conjecture of 1970 \cite{nash1970unsolved} (the central question of what might be termed \emph{extremal design theory}) and \Erdos{}' Existence Conjecture for High Girth Steiner Triple Systems of 1973 \cite{E73} (an important question of what might be termed \emph{structural design theory}). 

Nash-Williams' Conjecture concerns triangle decompositions of graphs with large minimum degree. A \emph{triangle decomposition} of a graph $G$ is a decomposition of the edges of $G$ into triangles. The necessary divisibility conditions are as follows: a graph $G$ is \emph{$K_3$-divisible} if $3~|~e(G)$ and for each $v\in G$, we have $2~|~d_G(v)$, where $d_G(v)$ denotes the degree of $v$. Nash-Williams conjectured that every large enough $K_3$-divisible graph on $n$ vertices with minimum degree at least $3n/4$ has a triangle decomposition. The conjecture is nearly resolved due to a combination of the work of Barber, K\"{u}hn, Lo, and Osthus~\cite{BKLO16} from 2016 (who reduced a version of the conjecture to its fractional relaxation) as well as Delcourt and Postle~\cite{DP2021progress} from 2021 (who proved the best known fractional upper bound of $\frac{7+\sqrt{21}}{14}\approx0.82733$). We defer a more detailed history of this conjecture until later in Section~\ref{subsec:NW} of the introduction.

\Erdos{}' High Girth Steiner Triple System Conjecture concerns finding Steiner triple systems that are locally sparse (equivalently they avoid dense local configurations). Namely, one defines the \textit{girth} of a $K_3$-decomposition as the smallest integer $g$ such that there are $g$ triangles in the decomposition spanning at most $g+2$ vertices. \Erdos{} conjectured that for each integer $g$ and for all large enough $n\equiv 1,3 \pmod 6$, there exists an $(n,3,2)$-Steiner system of girth at least $g \geq 3$. \Erdos{}' Conjecture was completely resolved by Kwan, Sah, Sawhney, and Simkin \cite{KSSS2024STS} (published in 2024). We defer a more detailed history of this conjecture until later in Section~\ref{ss:Erdos} of the introduction.

The common generalization of these two conjectures was dubbed the 
``\Erdos{} meets Nash-Williams' Conjecture'' by Glock, \Kuhn, and Osthus~\cite{GKO20Survey} in 2021. Our main result is the following:
\begin{thm}\label{thm:Main_ErdosNashWilliams}
For every integer $g \geq 3$ and real $\varepsilon > 0$, any sufficiently large  $K_3$-divisible graph $G$ on $n$ vertices with minimum degree
\[\delta(G) \geq \Big(\max\Big\{\delta^*_{K_{3}}, \frac{3}{4}\Big\}+\varepsilon\Big)\cdot n\]
admits a $K_3$-decomposition with girth at least $g$.
\end{thm}
Here $\delta^*_{K_{3}}$ denotes the \textit{fractional $K_3$-decomposition threshold}, that is, the infimum of all real numbers $c$ such that every graph $G$ with minimum degree at least $c\cdot v(G)$ has a \emph{fractional $K_3$-decomposition} (an assignment of non-negative weights to the triangles of $G$ such that for each edge, the sum of the weights of triangles containing that edge is exactly $1$). \cref{thm:Main_ErdosNashWilliams} reduces the ``\Erdos{} meets Nash-Williams' Conjecture'' (or rather the version with minimum degree $cn$ for any constant $c > 3/4$) to the fractional relaxation of Nash-Williams' Conjecture, which is also the current state of Nash-Williams' Conjecture. Combining \cref{thm:Main_ErdosNashWilliams} with the best known bound for this fractional decomposition threshold leads to the following corollary. 

\begin{cor}[Corollary to~\cref{thm:Main_ErdosNashWilliams}]\label{cor:main}
    For every integer $g \geq 3$ and real $\varepsilon > 0$, every sufficiently large $K_3$-divisible graph $G$ on $n$ vertices with minimum degree 
    \[\delta(G) \ge \left( \frac{7+\sqrt{21}}{14} + \varepsilon \right) \cdot n\]
    admits a $K_3$-decomposition with girth at least $g$. 
\end{cor}

Both of these conjectures and their generalizations to general $q$ and $r$ have an extensive history. Thus, in the remaining sections of the introduction, we describe the history of each of the two conjectures. Then we discuss the history of the combined conjectured. Finally, we finish with a discussion of the novelties in our proof.

\subsection{History of Nash-Williams' Conjecture}\label{subsec:NW}

It is natural to wonder about $K_3$-decompositions of graphs other than the complete graph. In contrast to Kirkman's theorem, $K_3$-divisibility is not sufficient to guarantee a $K_3$-decomposition: for example,~$C_6$ is $K_3$-divisible but not $K_3$-decomposable, and more generally, any $K_3$-divisible triangle-free graph does not even contain one triangle. 

As the Existence Conjecture concerns the decompositions of complete graphs, it is natural in the spirit of that conjecture to consider decompositions of nearly complete graphs. That said, even finding one triangle or clique or vertex-disjoint sets of triangles or cliques has a storied history. Indeed, the threshold for the number of edges to guarantee even one triangle is a famous classical result widely attributed to Mantel~\cite{mantel} from 1907 who showed that every graph on $n$ vertices with at least $\lfloor \frac{n^2}{4} \rfloor +1$ edges contains a triangle (which is tight). In his seminal result of 1942, Tur{\'a}n~\cite{turan1941extremal} generalized this by determining the exact number of edges needed to guarantee the existence of a $K_q$ in a graph - roughly around $\big(1-\frac{1}{q-1}\big)n$. This became known as Tur\'an's theorem, the numbers as Tur\'an numbers, and launched an entire subfield of extremal combinatorics known as Tur\'an theory.

Still, instead of finding one triangle, one might ask to find a collection of vertex-disjoint triangles spanning all vertices (called a \emph{triangle factor}). For that problem, it is natural to impose a minimum degree condition on the dense graph since an isolated vertex would not be in any triangle. The question of what minimum degree forces a triangle factor was famously resolved by Corr{\'a}di and Hajnal~\cite{CH63} in 1963 who showed that any graph $G$ on $n$ vertices with $3 ~|~n$ (a necessary divisibility condition) and $\delta(G) \ge \frac{2}{3}\cdot n$ contains a triangle factor.  As for the minimum degree necessary to force the existence of a \emph{$K_q$-factor} (a collection of vertex-disjoint copies of $K_q$ spanning all vertices), the Hajnal-Szemer\'edi Theorem~\cite{HS70} from 1970 states that any graph $G$ on $n$ vertices with $q ~|~n$ and and $\delta(G) \ge \frac{q-1}{q}\cdot n$ contains a $K_q$-factor.

Thus, it was only natural to wonder what minimum degree would force a triangle decomposition (provided the necessary divisibility conditions are satisfied). In 1973, Nash-Williams conjectured the following. 

\begin{conj}[Nash-Williams~\cite{nash1970unsolved}]
    Every sufficiently large $K_3$-divisible graph $G$ on $n$ vertices with $\delta(G)\geq \frac34n$ admits a $K_3$-decomposition.
\end{conj}

We note that this conjecture of Nash-Williams also has a connection to another famous result: Dirac~\cite{dirac1952some} showed in 1952 that if $G$ is a graph on $n\ge 3$ vertices and $\delta(G) \ge n/2$, then $G$ is Hamiltonian (i.e.~it has a cycle spanning all vertices). Dirac's result implies the much easier result that if $G$ satisfies the necessary divisibility conditions for the existence of a perfect matching (namely that $n$ is even) and $\delta(G) \ge n/2$, then $G$ has a perfect matching. Since a perfect matching is equivalent to an $(n, 2, 1)$-Steiner system, Dirac's Theorem provides a minimum degree threshold for when the obvious necessary divisibility condition for the existence of an $(n, 2, 1)$-Steiner system becomes sufficient. Nash-Williams' Conjecture can then be seen as one approach to generalize this to $(n, 3, 2)$-Steiner systems. Indeed, one could view these results of Dirac~\cite{dirac1952some}, Corr\'adi-Hajnal~\cite{CH63}, Hajnal-Szemer\'edi~\cite{HS70} and their ilk, but also results such as the famous result of R\"odl-Ruci\'nski-Szemer\'edi~\cite{RRS06} on the minimum degree threshold for hypergraph matchings, as the fundamental questions of $1$-uniform extremal design theory (covering all the vertices); from that perspective, one could also view all of Tur\'an theory as $0$-uniform extremal design theory (since covering the empty set with cliques is equivalent to finding one clique). 

In this way, {\bf Nash-Williams' Conjecture is the central question in extremal design theory} as borne out by its generalizations (to $K_q$ or general uniformities) and reductions to it (such as our main result). On the other hand, Nash-Williams' Conjecture has relevance to classical design theory. Notably, Nash-William's Conjecture implies a maximum degree bound for the related \textit{completion problem for designs:} assuming Nash-Williams' Conjecture is true, any $K_3$-packing of $K_n$ with maximum degree at most $n/4$ can be extended to a full $K_3$-decomposition. In other words, the Nash-Williams' Conjecture relates to the threshold for which a partial $(n,3,2)$-Steiner system can always be completed to a full system. We note that Nash-Williams' Conjecture is stronger than this completion problem since the complement of the graph in Nash-Williams' Conjecture may not admit a $K_3$-decomposition (e.g. it could be $C_6$ or more generally triangle-free).

We should also note that if true, Nash-Williams' Conjecture is tight. In an addendum to Nash-Williams' article~\cite{nash1970unsolved}, Graham remarked that the fraction $3/4$ would be best possible, as seen by the following construction: Let $G$ be a blowup of $C_4$, replacing every vertex by a complete graph on $6m+3$ vertices for some $m\geq 1$ and every edge by a complete bipartite graph, and call an edge between two parts {\em crossing}. The graph $G$ is a $K_3$-divisible graph on $n=24m+12$ vertices and with minimum degree $\frac{3n}{4}-1$. Every triangle in $G$ contains at most $2$ crossing edges. However, as $4\binom{6m+3}{2}<\frac12 4(6m+3)^2$, there are not enough non-crossing edges available to decompose all crossing edges into triangles. 

We note that a necessary condition for the existence of a $K_3$-decomposition is the existence of a fractional $K_3$-decomposition (defined earlier). The same counting argument shows that the above construction does not even admit a fractional $K_3$-decomposition, and hence for the fractional threshold we have that $\delta_{K_3}^*\ge 3/4$. Improvement to the upper bound on $\delta_{K_3}^*$ also has a long history. In 2014, Garaschuk~\cite{garaschuk2014linear} showed $\delta_{K_3}^*\le 0.956$; in 2016, Dross~\cite{dross2016fractional} improved this $0.9$; in 2020, Dukes and Horsley~\cite{dukes_minimum_2020} improved this further to $0.852$. Currently the best known upper bound is due to Delcourt and Postle~\cite{DP2021progress} from 2021 who proved that $\delta_{K_3}^*\leq \frac{7+\sqrt{21}}{14}\approx0.82733$. 

In a seminal work from 2001, Haxell and R\"{o}dl~\cite{HR01} combined Szemer\'{e}di’s Regularity Lemma and R\"odl's nibble method with other ingenious arguments to show that the maximum number of edges of a graph $G$ that can be fractionally covered with copies of a graph $F$ is approximately equal to the number that can be covered by edge-disjoint copies of $F$. This implies that if $G$ admits a fractional $F$-decomposition, then $G$ admits an approximate $F$-decomposition (see the work of Yuster~\cite{Yuster05} from 2005 for a shorter proof and generalization to decompositions into families of graphs).

Returning to the status of Nash-Williams' Conjecture, an enormous breakthrough was the work of Barber, K\"{u}hn, Lo, and Osthus~\cite{BKLO16} who used iterative absorption to show how to complete an approximate $K_3$-decomposition into a full one, hence tying the existence of $K_3$-decompositions to the existence of fractional $K_3$-decompositions as follows.

\begin{thm}[Barber, K\"{u}hn, Lo, and Osthus~\cite{BKLO16}]\label{thm:NW_fractional}
    For every real $\varepsilon>0$, every sufficiently large $K_3$-divisible graph~$G$ on~$n$ vertices with minimum degree at least
    $(\max\{\delta^*_{K_3},\frac{3}{4}\}+\varepsilon)n$
    admits a $K_3$-decomposition.
\end{thm}

As noted before, $\delta_{K_3}^* \ge 3/4$ and so this value is redundant in the theorem above but we include it for completeness to denote the value required from the absorber part of the proof (we use the same formulation in our main results for clarity). Combined with the result of Delcourt and Postle~\cite{DP2021progress}, this shows that every $K_3$-divisible graph $G$ with $\delta(G)\ge 0.82733n$ (and $n$ large enough) admits a $K_3$-decomposition. Given Theorem~\ref{thm:NW_fractional}, improving the fractional bound is now the barrier to improving the bound in Nash-Williams' Conjecture. 

One might wonder about generalizing Nash-Williams' Conjecture to all graph cliques (such a result would generalize Wilson's theorem). To that end, we have to recall the correct notion of divisibility. For a graph $F$, a graph $G$ is \emph{$F$-divisible} if $e(F)~|~e(G)$, and the greatest common divisor of the degrees of the vertices in~$F$, denoted by~$\gcd(F)$, divides $d_G(v)$ for every vertex $v$ of $G$. A folklore generalization of Nash-Williams' Conjecture 
asserts that for each integer $q\ge 3$, every sufficiently large $K_q$-divisible graph $G$ on $n$ vertices with $\delta(G)\geq \big(1-\frac{1}{q+1}\big)n$ admits a $K_q$-decomposition. In 2005, Yuster~\cite{Yuster2005asymptotically} provided examples showing that this threshold would be best possible for every $q\geq 3$. 

Once again, study of the fractional relaxation is paramount. For a graph $F$, a {\em fractional $F$-decomposition} of a graph $G$ is an assignment of non-negative weights to the copies of $F$ of $G$ such that for each edge $e$ of $G$, the sum of the weights of copies of $F$ containing $e$ is exactly $1$. The {\em fractional $F$-decomposition threshold $\delta^*_F$} is defined as the infimum of all $\delta\in [0,1]$  such that there exists an integer $n_0$  such that for all integers $n\ge n_0$ and graphs~$G$ on~$n$ vertices with minimum degree $\delta(G)\ge \delta\cdot n$ has a fractional $F$-decomposition. A simple blow-up argument shows that the $n_0$ restriction is unnecessary when $F$ is a clique (hence the earlier definition for triangles). 

Yuster's constructions actually provide obstructions in the fractional setting as well, proving that for every $q\geq 3$ we have $\delta_{K_q}^*\geq 1-\frac{1}{q+1}$. As for an upper bound on $\delta_{K_q}^*$, there was a series of works by Yuster~\cite{Yuster2005asymptotically}, Dukes~\cite{dukes_minimum_2020}, and Barber, K\"{u}hn, Lo, Montgomery, and Osthus~\cite{BKLMO17}; the current best known bound is by
Montgomery~\cite{montgomery_fractional_2019}  from 2019 who showed that for every $q\geq 4$ we have $\delta_{K_q}^*\leq 1-\frac{1}{100q}$ (notably the first and only upper bound proof to yield a bound whose denominator is linear in $q$). 

As to finding $K_q$-decompositions, Glock, K\"{u}hn, Lo, Montgomery, and Osthus~\cite{GKLMO19} used iterative absorption to prove that every sufficiently large $K_q$-divisible graph~$G$ on~$n$ vertices with minimum degree at least $(\max\{\delta_{K_q}^*,1-\frac{1}{q+1}\}+\varepsilon)n$ admits a $K_q$-decomposition. This improved the earlier work of Barber, K\"{u}hn, Lo, and Osthus~\cite{BKLO16} (who showed the same but with $1-\frac{1}{3q-3}$). 

As for general graphs $F$, Wilson~\cite{W75} in 1975 extended his proof to show that for any graph~$F$, every sufficiently large $F$-divisible complete graph admits an $F$-decomposition. In 2012, Yuster~\cite{Yuster12} showed for any graph $F$ that $\delta_F^* \le \delta_{K_{\chi(F)}}$, with $\chi(F)$ the chromatic number of $F$. In 2019 using Yuster's work and iterative absorption, Glock, K\"{u}hn, Lo, Montgomery, and Osthus~\cite{GKLMO19} generalized Wilson's theorem and their $K_q$ decomposition result mentioned above by showing that for any graph $F$ where we let $q:= \chi(F)$, every sufficiently large $F$-divisible graph~$G$ on~$n$ vertices with minimum degree at least $(\max\{\delta^*_{K_q},1-\frac{1}{q+1}\}+\varepsilon)n$ admits a $K_q$-decomposition. Note this is analogous to how the \Erdos{}-Stone–Simonovits theorem (first explicitly proven and linked to Tur\'an numbers by Erd\H{o}s and Simonovits~\cite{ES66} in 1966, but following easily from the earlier work of Erd\H{o}s and Stone \cite{ES46} from 1946) generalizes Tur\'an's theorem from $K_q$ to general $F$; this demonstrates the importance of determining $\delta_{K_q}^*$. We refer the reader to the survey by Glock, K\"{u}hn, and Osthus~\cite{GKO20Survey} for further information on what is known for general $F$.

\subsection{History of Erd\H{o}s' Conjecture}\label{ss:Erdos}

Another stream of research in design theory concerns finding designs that avoid certain substructures (and so might be termed \emph{structural design theory}); for example, much research has gone in to finding and studying Latin squares that avoid 2 by 2 subsquares (called \emph{intercalates}) or more generally any Latin subsquare; avoiding the latter is the subject of a notorious conjecture of Hilton (e.g., see~\cite{DK76} or~\cite{AW23} for the history) whose proof was only recently concluded by Allsop and Wanless~\cite{AW23}. As for Steiner systems, it was of great interest to prove the existence of Steiner triple systems that are \emph{Pasch-free}, that is, they avoid the \emph{Pasch configuration}, the unique $K_3$-packing of $4$ triangles on $6$ vertices. In 1977, Brouwer~\cite{brouwer1977steiner} proved that Pasch-free Steiner triple systems exist for all $n\equiv 3\mod 6$;
while only in 2000 the conjecture was resolved by Grannell, Griggs, and Whitehead~\cite{GGW00AntiPasch} who showed the existence for every $n\equiv 1\mod 6$, $n\neq 7,13$, (after important progress by Ling, Colbourn, Grannell, and Griggs~\cite{LCGG00construction} who showed that exceptions must lie in the residue classes $13, 31, 67 \mod 72$). It is natural then to consider more generally what substructures can be avoided in Steiner systems.

To that end, we say a {\em $(j,i)$-configuration} in a set $\mc{P}$ of sets on a ground set $X$ is a set of $i$ elements of $\mc{P}$ spanning at most $j$ elements of $X$. Thus a $(j,i)$-configuration in a $K_3$-packing is a set of $i$ triangles spanning at most $j$ vertices. Observe that one triangle is a $(4, 1)$-configuration and any two triangles sharing a vertex is a $(5,2)$-configuration. Indeed, every $(n,3,2)$-Steiner system contains an $(i+3,i)$-configuration for every $1 \le i \le n-3$; this observation prompted \Erdos{} to study $(i+2,i)$-configurations in the 1970s. The \textit{girth} of a triangle packing is the smallest integer $i\ge 2$ such that the packing contains an $(i+2, i)$-configuration. Note that a family of triangles having the property of containing no $(4,2)$-configuration is equivalent to being a triangle packing; similarly having no $(6,4)$-configuration is equivalent to being Pasch-free. 

In 1973, Brown, Erd\H{o}s, and S\'{o}s~\cite{BES73} proved that for every $i\geq 2$, there exist constants $C_i>c_i> 0$ such
that every triangle-packing on $n$ vertices of size at least $C_i\cdot n^2$ has an $(i+2,i)$-configuration, while for every integer~$n$ there exists a triangle-packing of size $c_i\cdot n^2$ with no $(i+2,i)$-configuration. We refer the reader to the recent work of Delcourt and Postle~\cite{DP22BES} from 2022 for further information on this problem. 

Stemming from this work, \Erdos~\cite{E73} in 1973 conjectured the existence of $K_3$-decompositions of $K_n$ with arbitrarily large girth (provided of course that $n$ is sufficiently large compared to the girth). In 1993, Lefmann, Phelps, and R{\"o}dl \cite{LPR93} showed that for every $g \ge 2$, there exists a constant $c_g$ depending on $g$ such that for every $n \ge 3$, there exists a partial Steiner triple system $S$ with $|S| \ge c_g \cdot n^2$ and girth at least $g$ (with $c_g \rightarrow 0$ as $g\rightarrow \infty$). More recently, in 2019 Glock, K{\"u}hn, Lo, and Osthus~\cite{GKLO20} and independently Bohman and Warnke~\cite{BW19} settled the approximate version of \Erdos{}' Conjecture. In 2022, Kwan, Sah, Sawhney, and Simkin \cite{KSSS2024STS} impressively proved the conjecture in full using the method of iterative absorption.

\begin{conj}[\Erdos{}' Conjecture - proved by Kwan, Sah, Sawhney, and Simkin~\cite{KSSS2024STS}]\label{thm:KSSS}
    For every integer $g\geq 3$, every sufficiently large $K_3$-divisible complete graph admits a $K_3$-decomposition with girth at least $g$. 
\end{conj}

We note that in a follow-up paper, Kwan, Sah, Sawhney, and Simkin~\cite{KSSS2023substructures} also proved the existence of high girth $K_3$-decompositions of $K_{n,n,n}$, the complete tripartite graph with parts of size $n$, answering positively a question by Linial~\cite{linial2018challenges} on the existence of Latin squares with arbitrarily high girth.

It is natural to consider generalizing the notion of girth to $(n,q,r)$-Steiner systems.  In \cite{GKLO20}, Glock, K{\"u}hn, Lo, and Osthus noted that for every fixed $i \ge 2$, every $(n, q, r)$-Steiner system contains a $((q -r)i+r+1, i)$-configuration. Motivated by this fact, one defines the \emph{girth} of an $(n, q, r)$-Steiner system as the smallest integer $g \ge 2$ for which it contains a $((q - r)g + r, g)$-configuration.  Glock, K{\"u}hn, Lo, and Osthus~\cite{GKLO20}, as well as Keevash and Long~\cite{KL20}, conjectured the existence of $K_q^r$-decompositions of $K_n^r$ for arbitrarily large girth (for $r = 2$ and $q \ge 3$, this was previously asked by F{\"u}redi and Ruszink{\'o} \cite{FR13} in 2013).

In 2022, the approximate version of this High Girth Existence Conjecture was settled by Delcourt and Postle \cite{DP22} and, independently, Glock, Joos, Kim, K{\"u}hn, and Lichev \cite{GJKKL24}; that is, they proved the existence of approximate $(n, q, r)$-Steiner systems of high girth with almost full size. In fact, both papers developed a general methodology that finds almost perfect matchings in hypergraphs that avoid a set of forbidden submatchings provided certain degree and codegree conditions are met. In particular, the general results then imply the approximate version of the High Girth Existence Conjecture as a corollary. Recently, in 2024 Delcourt and Postle \cite{DPII} proved the High Girth Existence Conjecture in full using their newly developed refined absorption methodology.

\subsection{History of ``Erd\H{o}s meets Nash-Williams''}

Given these two fundamental conjectures, namely Nash-Williams' Conjecture and Erd\H{o}s' Conjecture, as well as the progress on them in recent years, it is natural to wonder about the minimum degree threshold to guarantee the existence of a high girth $K_3$-decomposition. Indeed in 2020, Glock, \Kuhn, and Osthus~\cite[Conjecture 7.7]{GKO20Survey} proposed the following combination of these two conjectures.

\begin{conj}[``\Erdos{} meets Nash-Williams''~\cite{GKO20Survey}]\label{conj:E-NW}
    For every integer $g$, any sufficiently large $K_3$-divisible graph $G$ on $n$ vertices with minimum degree $\delta(G)\geq 3n/4$ admits a $K_3$-decomposition with girth at least $g$.
\end{conj}

As for progress on the above conjecture, Glock, Joos, Kim, K{\"u}hn, and Lichev \cite{GJKKL24} showed an approximate high girth decomposition exists when $\delta(G) \ge (1-\varepsilon)n$ for $\varepsilon = \frac{1}{12^4} = \frac{1}{20,736}$ (more generally they showed a similar version for general $q$ and $r$). Essentially, no progress had been made on even this approximate version until very recently, when M. K{\"{u}}hn~\cite{kuhnPhD2025} independently (in his Ph.D. thesis) showed that $\delta(G)\ge (\delta_{K_3}^*+\varepsilon)n$ yields a $(1-\varepsilon)$-approximate high girth version. 

Still, no progress had been made on proving the existence of a high girth $K_3$-decomposition whenever $\delta(G)\ge (1-\varepsilon)n$ for any $\varepsilon > 0$. Our main result,~\cref{thm:Main_ErdosNashWilliams}, pushes past this barrier and indeed even ties the existence of high girth $K_3$-decompositions to the ``standard'' fractional threshold $\delta^*_{K_3}$ (recall that the current best known bound is $\delta^*_{K_3} \le \frac{7+\sqrt{21}}{14}$, which thus provides \cref{cor:main}).

\subsection{On a Generalization to High Girth \texorpdfstring{$(n,q,2)$-Steiner systems}{(n,q,2)-Steiner systems}}\label{sec:Generalq}

We now turn to the common generalization of \Erdos{}' Conjecture and the generalization of Nash-Williams' Conjecture to $K_q$-decompositions. Namely, it is natural to conjecture the following.

\begin{conj}\label{conj:ENWLargeCliquesNoEpsilon}
    For all integers $g\geq 3$ and $q\geq 3$, every sufficiently large $K_q$-divisible graph~$G$ on~$n$ vertices with minimum degree at least $\big(1-\frac{1}{q+1}\big)n$ admits a $K_q$-decomposition with girth at least $g$.
\end{conj}

Our results in this article can be generalized from triangles to $K_q$ to prove the following.

\begin{thm}\label{thm:GenralisationMain2q}
    For all integers $g\geq 3$ and $q\geq 3$ and real $\varepsilon>0$, every sufficiently large $K_q$-divisible graph~$G$ on~$n$ vertices with minimum degree at least
    $(\max\{\delta^*_{K_q},1-\frac{1}{2q-2}\}+\varepsilon)n$
    admits a $K_q$-decomposition with girth at least $g$. 
\end{thm}

We explain in~\cref{sec:Finishing}, after the proof of~\cref{thm:Main_ErdosNashWilliams}, how~\cref{thm:GenralisationMain2q} follows from a simple generalization of our results. Recall that the best known upper bound for $q \geq 4$ on $\delta^*_{K_q}$ is $1-\frac{1}{100q}$ by Montgomery~\cite{montgomery_fractional_2019} and hence the fractional bound is the bottleneck in the above theorem. While most of the machinery of our proof would work for the better bound of Conjecture~\ref{conj:ENWLargeCliquesNoEpsilon} (at least with an additional $\varepsilon$ in the minimum degree), there are two areas which would require new ideas. One is utilizing the better absorber constructions from Glock, K\"{u}hn, Lo, Montgomery, and Osthus~\cite{GKLMO19} to embed the absorbers in a graph with the lower minimum degree condition; the other much harder issue is proving that girth boosters can also be embedded for such a lower minimum degree condition. We discuss these issues further in the conclusion. 

\subsection{On the Novelty of Our Proof}

It must be noted that our proof of~\cref{thm:Main_ErdosNashWilliams} is not a mere extension of the proof of~\cref{thm:KSSS} by Kwan, Sah, Sawhney, and Simkin~\cite{KSSS2024STS} to the minimum-degree world, or vice versa, of the proof of~\cref{thm:NW_fractional} by Barber, K\"{u}hn, Lo, and Osthus~\cite{BKLO16} to the high girth setting. These proofs are based on iterative absorption approaches, and both contain parts that do not work in the ``\Erdos{} meets Nash-Williams'' setting. 
For example, the proof of~\cref{thm:NW_fractional} relies on a result by Haxell and R\"{o}dl to transform a fractional decomposition into an approximate one; however, the resulting decomposition is not guaranteed to have large girth. On the other hand, modifying the proof of~\cref{thm:KSSS} would require redoing the whole argument to work in the high minimum degree setting while still needing to resolve issues we tackle here, such as regularity boosting and embedding girth boosters.

Rather, our proof utilizes the new method of refined absorption (developed in~\cite{DPI}) to overcome these issues. We also utilize the forbidden submatching method (developed in~\cite{DP22}) along with the `girth booster' machinery developed in~\cite{DPII}. Nevertheless, this paper is not a mere combination of these previous techniques. 

Indeed, this is the first paper showing how to adapt the proof framework of refined absorption to yield results in the setting of Nash-Williams' Conjecture; thus, we provide a new proof of Theorem~\ref{thm:NW_fractional} based on refined absorption (see proof overview) which is of independent interest. 

Still more work is required for the combination of the refined-absorption-based high girth and high minimum degree proofs. In particular, before applying nibble, for the regularity boosting step of our proof (or rather the forbidden submatchings generalization), we need to find a very regular set of cliques. We accomplish this via an `Inheritance Lemma' similar to that found in Lang~\cite{lang2023tiling} that bootstraps the fractional version to yield a low-weight fractional decomposition (as was also done independently by K\"uhn~\cite{K24}). However, while K\"uhn only needed constant irregularity, we need our irregularity much smaller (i.e.~polynomially small) to finish the decomposition. We accomplish this by a novel idea of `seeding' the fractional decomposition so that every clique has some amount of weight that we might then invoke the a general form of the Boosting Lemma of Glock, K\"{u}hn, Lo, and Osthus~\cite{GKLO16} to achieve such small irregularity. 

We also have to adapt the girth booster machinery from~\cite{DPII} to work in the high minimum degree setting. Thankfully for $q=3$, the girth boosters have rooted degeneracy at most $4$ and so they will embed inside $G$. Nevertheless, we have to explain how to adapt said machinery, which requires some care as the machinery is highly technical.

We discuss these intricacies further (and provide relevant definitions) in the next subsection which provides a proof overview. We first review the refined absorption proof from~\cite{DPI}, then how to adapt it to Nash-Williams' Conjecture; we then recall (at a very high level) how the refined absorption proof was adapted for high girth in~\cite{DPII} before proceeding to discuss our proof and the intricacies alluded to above.

\section{Proof Overview}

\subsection{A Refined Absorption Primer}

The new proof of the Existence Conjecture by Delcourt and Postle in~\cite{DPI} introduced the novel method of refined absorption whose definitions we now recall in the setting of graphs specifically. 

\begin{definition}[Absorber] Let $L$ be a $K_q$-divisible graph.  A graph $A$ is a \emph{$K_q$-absorber} for $L$ if $V(L)$ is independent in $A$ and both $A$ and $A\cup L$ admit $K_q$-decompositions.    
\end{definition}

\begin{definition}[Omni-Absorber]\label{def:OmniAbsorbers}
A graph $A$ is a \emph{$K_q$-omni-absorber} for a graph $X$ with \emph{decomposition family} $\cF_A$ and \emph{decomposition function} $\cQ_A$  if $X$ and $A$ are edge-disjoint and for every $K_q$-divisible subgraph $L$ of $X$, $\cQ_A(L)\subseteq \cF_A$ is a $K_q$-decomposition of $A\cup L$.  We say that an omni-absorber $A$ has \emph{collective girth} $g$ if $\cQ_A(L)$ has girth at least $g$ for every $K_q$-divisible subgraph $L$ of $X$.
\end{definition}

The key to the method of refined absorption is a black-box theorem about the existence of efficient omni-absorbers that also have a `refinedness' property as follows. A $K_q$-omni-absorber $A$ is \emph{$C$-refined} if every edge of $X\cup A$ is in at most $C$ elements of $\cF_A$. Here is the black-box theorem from~\cite{DPI} restricted to the graph case.

\begin{thm}[Refined Efficient Omni-Absorber~\cite{DPI}]\label{thm:RefinedEfficientOmniAbsorber}
For every integer $q\geq 3$, there exists an integer $C\ge 1$ such that the following holds: If $X\subseteq  K_n$ and $\Delta(X) \le \frac{n}{C}$, then there exists a $C$-refined $K_q$-omni-absorber $A$ for $X$ with $\Delta(A) \le C\cdot \max\left\{\Delta(X),~\sqrt{n}\cdot \log n\right\}$.    
\end{thm}

We now recall the outline of the new proof of the Existence Conjecture~\cite{DPI} restricted to the graph case:

\begin{enumerate}\itemsep.05in
    \item[(1)] `Reserve' a random subset $X$ of $E(K_n)$ uniformly with some small probability $p$.
    \item[(2)] Construct an omni-absorber $A$ of $X$ via Theorem~\ref{thm:RefinedEfficientOmniAbsorber}.
    \item[(3)] ``Regularity boost'' $K_n\setminus (A\cup X)$ to find a very regular set of cliques.
    \item[(4)] Apply ``nibble with reserves'' theorem to find a $K_q$-packing of $K_n\setminus A$ covering $K_n\setminus (A\cup X)$ and then extend this to a $K_q$-decomposition of $K_n$ by definition of omni-absorbers.  
\end{enumerate}

For (1), the authors of~\cite{DPI} proved via Chernoff bounds that there is a choice of $X$ with $\Delta(X)$ small and where for every edge $e\in G\setminus X$, $e$ is in many copies of $K_q$ in $X\cup\{e\}$. Thus from this and Theorem~\ref{thm:RefinedEfficientOmniAbsorber}, they find that $\Delta(X\cup A)$ is small. For (3), they then used a special case of the Boosting Lemma of Glock, K\"{u}hn, Lo, and Osthus~\cite{GKLO16} to show that a graph $G$ with minimum degree $\delta(G)\ge (1-\varepsilon)n$ for some small enough $\varepsilon$ has a very regular set of cliques - meaning every edge is in roughly the same amount of cliques, say $\left(\frac{1}{2}\pm n^{-1/3}\right)\cdot\binom{n-2}{q-2}$, where the irregularity is polynomially small (as required by ``nibble with reserves'').

\subsection{A Refined Absorption Proof of Nash-Williams}

We now discuss how to use refined absorption to prove Theorem~\ref{thm:NW_fractional}. The adaptations actually work to prove the more general version where $G$ is a $K_q$-divisible graph with $\delta(G)\ge \left(\max\{\delta_{K_q}^*, 1-\frac{1}{2q-2}+\varepsilon\}\right)n$ and we desire a $K_q$-decomposition of $G$. One modifies the above proof outline as follows. For (1), we `reserve' a random subset $X$ of $E(G)$. Namely in Section~\ref{s:NW}, we prove the following reserves lemma via Chernoff bounds. Notice that lemma below actually works for the better bound of $1-\frac{1}{q+1}+\varepsilon$ in the minimum degree.

\begin{lem}[Reserves]\label{lem:RandomXHighMinDeg}
For every integer $q \ge 3$ and real $\varepsilon\in (0,1)$, there exists a real $\gamma \in (0,1)$ such that the following holds for all large enough $n$. Let $G$ be a spanning subgraph of $K_n$ with $\delta(G)\ge (1-\frac{1}{q+1}+\varepsilon)n$ and $p\in [n^{-\gamma},1)$. There exists $X\subseteq G$ with $\Delta(X)\le 2pn$ such that every $e\in K_n\setminus X$ is in at least $\frac{1}{(q+1)^q}\cdot p^{\binom{q}{2}-1}\cdot \binom{n}{q-2}$ copies of $K_q$ in $X\cup \{e\}$.
\end{lem}

For (2), we need to find an omni-absorber $A$ of $X$ \emph{inside $G\setminus X$}. To that end, we find an omni-absorber $A_0$ of $X$ in $K_n$ via Theorem~\ref{thm:RefinedEfficientOmniAbsorber}. Then we \emph{re-embed} $A$ \emph{into $G\setminus X$} using the methodology of fake-edges and private absorbers developed in~\cite{DPI}. Thus we prove the following in Section~\ref{s:NW}. 

\begin{thm}[Nash-Williams Refined Efficient Omni-Absorber]\label{thm:NWRefinedEfficientOmniAbsorber}
For every integer $q\ge 3$ and real $\varepsilon\in (0,1)$, there exists an integer $C\ge 1$ such that the following holds:
\vskip.1in
\noindent If $X\subseteq  G\subseteq K_n$ with $\delta(G)\ge (1-\frac{1}{2q-2} + \varepsilon)n$ and $\Delta(X) \le \frac{n}{C}$, then there exists a $C$-refined $K_q$-omni-absorber $A$ for $X$ in $G$ with $\Delta(A) \le C\cdot \max\left\{\Delta(X),~\sqrt{n}\cdot \log n\right\}$.    
\end{thm}

The key to proving the above theorem is first that ``fake-edges '' (formally defined in Section~\ref{subsec:omni}) have \emph{rooted degeneracy} at most $q-1$ (meaning an ordering of non-root vertices with at most $q-1$ earlier neighbors - see~\cref{def:RootedDegeneracy} for a formal definition); the second key is that there exist $K_q$-absorbers of rooted degeneracy at most $2q-2$ (as shown in~\cite{DKPIII}). Note that using such an absorber construction is the bottleneck here (the rooted degeneracy bound of $2q-2$ is tight for $K_q$-absorbers as we discuss later) but it may be possible to utilize the more sophisticated absorber constructions of~\cite{GKLMO19} to improve the bound above. 

For (3) and (4), a different approach than was originally used in the refined absorption proof of the Existence Conjecture is needed. Indeed, we cannot directly apply the Boosting Lemma from Glock, K\"uhn, Lo, and Osthus~\cite{GKLO16} since that only works if the minimum degree is much closer to $n$, namely there is some $\varepsilon$ where it works for minimum degree $(1-\varepsilon)n$ (but unfortunately this $\varepsilon$ is much smaller than $1/4$). 

The cleanest approach (and by far the easiest) is to use the ``Inheritance Lemma for Minimum Degree'', a concept that has echoes in many earlier works in the literature but has started to gain prominence in works on Hamilton cycles, tilings and now (with this current work) decompositions. We will use a version of the lemma by Lang~\cite{lang2023tiling}, see also the work of Ferber and Kwan~\cite{FB22} on matchings in random hypergraphs for a non-rooted version of the Inheritance Lemma. We note that Lang's lemma itself has a remarkably simple proof from standard concentration inequalities that bypasses any use of Szemer\'edi's Regularity Lemma.  Lang's version of the Inheritance Lemma says that for any $m$-subset $M$ of an $r$-uniform hypergraph $G$, almost all $s$-subsets $S$ of $G$ that contain $M$ have approximately the same minimum $d$-degree (that is of $G[S]$) as $G$. We require only a very special case of this lemma (see Lemma~\ref{lem:LangInheritance} below) where we restrict to the case of graphs (as opposed to $r$-uniform hypergraphs) and vertex degree (instead of $d$-degree for some $d\in [r-1]_0$).

From the Inheritance Lemma, we can prove a regularity boosting lemma that works in the setting of Nash-Williams' Conjecture as follows. 

\begin{lem}[Nash-Williams Boosting Lemma - Constant Irregularity Version]\label{lem:NWRegBoostConstant}
For every integer $q\geq 3$ and real $\varepsilon \in (0,1)$, there exists $c\in(0,1)$ such that the following holds for every $\gamma\in (0,1)$ (provided $n$ is large enough). If $J\subseteq K_n$ with minimum degree at least $(\max\{\delta_{K_q}^*,1-\frac{1}{q+1}\}+\varepsilon)n$, then there exists a family $\mc{H}$ of copies of $K_q$ in $J$ such that every $e\in J$ is in $\parens*{c\pm \gamma}\cdot \binom{n-2}{q-2}$ copies of $K_q$ in $\mc{H}$.
\end{lem}

The idea is that any $s$-subset $S$ of our graph $J$ that satisfies $\delta(J[S]) \ge (\delta_{K_q}^*+\frac{\varepsilon}{2})m$ will admit a fractional $K_q$-decomposition by definition of the fractional decomposition threshold if $s$ is large enough. Thus if we average over these decompositions for all such $S$ and appropriately scale, we will find an ``almost'' fractional $K_q$-decomposition of $J$ (almost meaning each edge has roughly weight 1) but where the weight of any copy of $K_q$ is at most $\frac{C}{n^{q-2}}$ for some constant $C$ only depending on $q$ and $\varepsilon$ (since any fixed copy of $K_q$ is not in too many $S$). Then Lemma~\ref{lem:NWRegBoostConstant} follows by randomly sampling from this fractional weighting and applying Chernoff bounds. We note that K\"uhn~\cite{K24} essentially (and independently) used the same idea in his proof. Note importantly here that this only gives irregularity (the $\gamma$) for any small but fixed constant $\gamma$ as opposed to a version where $\gamma$ can be made polynomially small in $n$ which we will need for our main result as we discuss below. 

For step (4) then, one invokes the theorem of Alon and Yuster~\cite[Theorem~3.1]{AY05} to show that nibble produces a leftover $L'$ with maximum degree at most $\gamma'n$ where $\gamma'$ goes to zero as $\gamma$ does. Hence if we choose the reserves (and omni-absorber) with probability $p$ much larger than this, we can use the L\'ovasz Local Lemma to find a $K_q$-packing of $L'\cup X$ covering all of $L'$ and finish via the definition of omni-absorber as before.

\begin{remark}
A reader intimately familiar with the iterative absorption proofs or Haxell-R\"odl may wonder about alternative approaches for (3) and (4) here. For these readers, we include this remark. First, we note that the Inheritance Lemma approach is not the only avenue to prove Lemma~\ref{lem:NWRegBoostConstant}. Indeed, one can extract from Yuster's short proof of Haxell-R\"odl the existence of a very regular set of cliques covering most of the edges (equivalently a low weight fractional $K_q$-packing covering most of the edges). The issue is that this only covers most of the edges with no guarantee that the graph of uncovered edges has low maximum degree. To remedy this, one can delete all cliques containing vertices of large uncovered degree and then use a low-weight fractional version of Hajnal-Szemer\'edi on the neighborhood of each such vertex in turn. 

Another alternative for (3) and (4) is to sidestep a regularity boosting lemma by using the covering version of the fractional approach outlined above in this remark, similar to the method in Glock, K\"{u}hn, Lo, Montgomery, and Osthus~\cite{GKLMO19}. Namely, one uses the approximate $K_q$-decomposition of Haxell-R\"odl and then modifies it to have low maximum degree leftover so as to apply the Lov\'asz Loval Lemma into the reserves. The low maximum degree leftover is accomplished by covering the neighborhoods of high uncovered degree vertices each in turn using a robust version of Hajnal-Szemer\'edi and martingale concentration. 

We note that both approaches mentioned above are much longer and more intricate than the approach using the Inheritance Lemma and neither easily generalizes to hypergraphs (e.g.~that would require use of the Hypergraph Regularity Lemma and the hypergraph generalization of Haxell-R\"odl by R{\"o}dl, Schacht, Siggers, and Tokushige~\cite{RSST07}) while the Inheritance Lemma approach easily generalizes to hypergraphs.  
\end{remark}

\subsection{A Refined Absorption Proof for High Girth Existence}

As to a refined absorption proof of \Erdos{}' Conjecture, we now discuss at a high level how the proof of the High Girth Existence Conjecture from~\cite{DPII} modifies the basic refined absorption framework as follows. For parts (1) and (3) of the proof outline, the proof remains the same. 

Whereas for (2), one instead constructs a \emph{high girth} omni-absorber $A$ of $X$ that \emph{does not shrink `too many' configurations}. We refer the reader to~\cref{def:OmniAbsorberProj} and its subsequent remark for an explanation of this concept. To accomplish this, one uses~\cref{thm:RefinedEfficientOmniAbsorber} to find an omni-absorber $A_0$ of $X$ in $K_n$. One then uses ``girth boosters" to boost the girth of each clique in $\cF_{A_0}$. A \emph{$K_q$-booster} (see~\cref{def:rootedbooster}) is just a minimal non-trivial $K_q$-absorber for a copy of $K_q$; that is, a $K_q$-booster is equivalent to having two disjoint $K_q$-decompositions. This allows us the freedom to switch which decomposition is used. A ``girth booster" then is just an informal term for a booster where both decompositions have high girth (we refer the reader to~\cref{def:RootedGirth} for the more formal definition of \emph{rooted girth} of a booster). 

Using these boosters for each clique in the decomposition family $\cF_{A_0}$ of $A_0$ naturally defines a new omni-absorber $A$ of $X$ (namely if a decomposition would use a clique in $\cF_{A_0}$ we instead use the alternate decomposition of its booster; if it is not used, we  decompose the booster without the clique via the first decomposition). By choosing each booster randomly, one can also argue with high probability that $A$ collectively has high girth and that $A$ does not shrink too many configurations. 

For (4), we used the ``Forbidden Submatchings with Reserves'' Theorem, the main theorem from~\cite{DP22}. All of this machinery is rather technical and so precise definitions will be delayed until Section~\ref{s:HighGirth}. We note that the machinery is not really simpler for the triangle case as opposed to the general $K_q^r$ setting (except for girth boosters whose constructions are thankfully simpler for triangles and even for graphs).  

\subsection{A Refined Absorption Proof for Erd\H{o}s meets Nash-Williams}

Here then is our outline for the proof of Theorem~\ref{thm:Main_ErdosNashWilliams}:

\begin{enumerate}\itemsep.05in
    \item[(1)] `Reserve' a random subset $X$ of $E(G)$.
    \item[(2)] Construct a high girth omni-absorber $A$ of $X$ \emph{inside $G$} that does not shrink `too many' configurations. 
    \item[(3)] ``Regularity Boost'' via the Inheritance Lemma, fixing low-weight ``almost" fractional decompositions into low-weight fractional decompositions ones inside the subgraphs produced by the Inheritance Lemma, ``seeding'' the decomposition so every clique has some non-negligible weight, and then using the General Boosting Lemma (the latter two are only necessary to avoid the \emph{dangerous} copies of $K_q$ arising from low girth configurations shrunk by the high girth omni-absorber).
    \item[(4)] Apply ``Forbidden Submatchings with Reserves'' theorem to find a \emph{high girth} $K_q$ packing of $G\setminus A$ covering $G\setminus (A\cup X)$ and then extend this to a \emph{high girth} $K_q$ decomposition of $G$ by definition of \emph{high girth} omni-absorber.
\end{enumerate}

Our approach to show part (1) of the proof outline here is the same as the one in the discussion of Nash-Williams' Conjecture above (namely we utilize Lemma~\ref{lem:NWRegBoost}). For (2), we start with an omni-absorber $A_0$ as given by Theorem~\ref{thm:NWRefinedEfficientOmniAbsorber} and subsequently use girth boosters to make it high girth - with some care needed to ensure we only use boosters that embed \emph{inside} $G$ (this works since the girth boosters also have rooted degeneracy at most $2q-2$). Note that the bottleneck on improving the minimum degree requirement comes from the rooted degeneracy of these girth boosters. For $q=3$, this yields the desired results; whereas for general $q$, a more sophisticated argument would be needed in order to work around the rooted degeneracy (which is tight for girth boosters just as it was for absorbers) if one wants to obtain a minimum degree of $\big(1-\frac{1}{q+1}\big)$ as in Conjecture~\ref{conj:ENWLargeCliquesNoEpsilon}. 

As to how we accomplish using the girth boosters, we extract from the proof given in~\cite{DPII}, a stronger theorem than what is stated in~\cite{DPII}; namely, Theorem~\ref{thm:HighGirthQuantumOmniBooster} (also see Theorem~\ref{cor:ExistenceOmniAbsorber} for its use), which says that the girth booster machinery works as intended even when each clique demands that their booster comes from only some constant proportion of the possible girth boosters; we quarantine this extraction to the appendix. Given the rooted degeneracy of our girth booster, it is then straightforward to show that not only does one such girth booster embed in $G$ but indeed a constant proportion do (see Proposition~\ref{prop:DegeneracySphere}). Since the definitions needed for the formal statements of these theorems are rather technical, we delay them until Section~\ref{s:HighGirth} where we then also prove a high girth omni-absorber analogue of Theorem~\ref{thm:NWRefinedEfficientOmniAbsorber} (namely Theorem~\ref{thm:HighGirthAbsorber}) assuming our extraction theorem. 

For part (4) of the outline, we finish similarly to the proof of the High Girth Existence Conjecture by using the ``Forbidden Submatchings with Reserves'' Theorem (see Section~\ref{s:HighGirth} for its statement and relevant definitions). We note that one could probably also use the conflict-free hypergraph matchings machinery from Glock, Joos, Kim, K{\"u}hn, and Lichev \cite{GJKKL24}, and more precisely its ``reserve'' version by Joos, Mubayi, and Smith~\cite{JMS24}; however, this would require a very substantial rewriting of the results from~\cite{DPII} that we intent to use as black boxes as much as possible.

\subsection{Proof Overview for Regularity Boosting for \Erdos{} meets Nash-Williams}

The main issue then is in step (3) - namely, regularity boosting. In particular the use of ``Forbidden Submatchings with Reserves'' requires an irregularity in our set of cliques that is polynomially small in $n$. Thus an important issue for this paper is how to resolve this discrepancy in the irregularity; namely, we will instead prove the following stronger regularity boosting lemma in Section~\ref{s:RegBoosting}. 

\begin{lem}[Nash-Williams Boosting Lemma - Polynomially Small Irregularity Version]\label{lem:NWRegBoost}
For every integer $q\geq 3$ and real $\varepsilon \in (0,1)$, there exists $c\in(0,1)$ such that the following holds for all $n$ large enough. If $J\subseteq K_n$ with minimum degree at least $(\max\{\delta_{K_q}^*,1-\frac{1}{q+1}\}+\varepsilon)n$, then there exists a family $\mc{H}$ of copies of $K_q$ in $J$ such that every $e\in J$ is in $\parens*{c\pm n^{-1/3}}\cdot \binom{n-2}{q-2}$ copies of $K_q$ in $\mc{H}$.
\end{lem} 

~\cref{lem:NWRegBoost} is a corollary of the following lemma (via sampling and Chernoff bounds). We say a fractional $F$-decomposition of a graph $G$ on $n$ vertices is \emph{$C$-low-weight} if every copy of $F$ in $G$ has weight at most $\frac{C}{n^{v(F)-2}}$. 

\begin{thm}\label{lem:LowWeightDecomposition}
    For every graph $F$ and real $\varepsilon\in (0,1)$, there exists an integer $C\geq 1$ such that the following holds for every integer $n$ large enough.    
    
    If $G$ is a $n$-vertex graph with $\delta(G)\geq (\delta^*_F+\varepsilon)n$, then there exists a $C$-low-weight fractional $F$-decomposition of $G$.
\end{thm}

We prove~\cref{lem:LowWeightDecomposition} by ``fixing'' the ``almost'' $C$-low-weight fractional decomposition of $G$ generated by the Inheritance Lemma. For this, we were inspired by a clever idea of Montgomery from his work on fractional decompositions~\cite{M19b}. Namely, if a graph $G$ has a fractional $F$-decomposition and for every $e\in G$, we have that $G-e$ has a fractional $F$-decomposition, then for any `target' weights $\phi(e)\in [1-\frac{1}{e(G)},1]$, there exists a fractional $F$-packing hitting these targets. This is accomplished by appropriately averaging the different fractional decompositions (see~\cref{lem:SeededFixed} for an application of this proof idea).

To prove~\cref{lem:LowWeightDecomposition} then, we use the target weight version inside the subgraphs generated by the Inheritance Lemma. Namely for each edge $e$ we calculate how many high minimum degree $s$-subsets $S$ that $e$ is in and adjust its target weight $\phi(e)$ for all of those subsets accordingly (that is such that the resulting weight on all edges is equal). This works since the value of ``almost'' from the Inheritance Lemma is $1-e^{-\sqrt{s}}$ while $e(G[S])\le s^2$ and $s$ is large enough so that $e^{\sqrt{s}}\ge s^2$. 

Unfortunately,~\cref{lem:NWRegBoost} is not sufficient for our purposes as one rather annoying technicality remains. Namely, we need that this family $\mc{H}$ does not include any \emph{dangerous} clique (a clique that would form a low girth configuration when combined with cliques from some decomposition from the high girth omni-absorber $A$). For this, it suffices to prove the version of~\cref{lem:NWRegBoost} where $\mc{H}$ must lie inside some given family of cliques with the property that every edge is in at least $(1-\alpha)\binom{n-2}{q-2}$ copies of $K_q$ for some small enough $\alpha$ (since the non-dangerous cliques will have this property) as follows.

\begin{lem}[Nash-Williams Restricted Boosting Lemma]\label{lem:NWRegBoostTreasury}
For every integer $q\geq 3$ and real $\varepsilon \in (0,1)$, there exist $\alpha,c\in(0,1)$ such that the following holds for all $n$ large enough. Let $J\subseteq K_n$ with minimum degree at least $(\max\{\delta_{K_q}^*,1-\frac{1}{q+1}\}+\varepsilon)n$ and $\cF_{orb}$ be a family of copies of $K_q$ in $J$, such that every edge $e\in J$ is in at most $\alpha\binom{n-2}{q-2}$ elements of $\cF_{orb}$. Then there exists a family $\mc{J}$ of copies of $K_q$ in $J$ with $\mc{J}\cap\cF_{orb}=\emptyset$ and such that every $e\in J$ is in $\parens*{c\pm n^{-1/3}}\cdot \binom{n-2}{q-2}$ copies of $K_q$ in $\mc{J}$.
\end{lem}

To prove this, we could just attempt to delete the cliques in the $C$-low-weight fractional $K_q$-decomposition of $G$ given by~\cref{lem:LowWeightDecomposition}. However, this idea (combined with sampling and Chernoff) would yield an irregularity proportional to $\alpha$; granted $\alpha$ is a small constant, as small compared to $\varepsilon$ and $C$ as we desire, but still a constant. So how to fix this new irregularity? We could attempt to use the general version of the Boosting Lemma of Glock, K\"{u}hn, Lo, and Osthus~\cite{GKLO16} (as found in~\cite{DKPIII}). This would work if every edge of $J$ was in many $K_{q+2}$'s all of whose $K_q$ subgraphs were in $\mc{H}$. While the minimum degree condition on $J$ precisely guarantees that every edge of $H$ is in many $K_{q+2}$'s, there is no guarantee from~\cref{lem:NWRegBoost} that their $K_q$ subgraphs are in $\mc{H}$. \emph{But what if we could guarantee this?} What if we did a version of~\cref{lem:LowWeightDecomposition} that guarantees this?

This leads to our last innovation - `seeding' the fractional decomposition. Namely, we will have the desired $K_{q+2}$ property if we can guarantee that every copy of $K_q$ in $G$ receives some large enough weight (because then when we sample from the fractional decomposition, the proportion of $K_{q+2}$ all of whose $K_q$-subgraphs are sampled is still very high). To that end, we make the following definition.

\begin{definition}[Balanced Fractional Decomposition]\label{def:BalancedDecomposition}
    For an $n$-vertex graph $G$ and real numbers $\sigma,C > 0$, we say that a fractional $F$-decomposition $\phi$ of $G$ is \emph{$(\sigma,C)$-balanced} if for each copy $H$ of $F$ in $G$ we have $\phi(H)\cdot  \binom{n-2}{v(F)-2}\in [\sigma,C]$.
\end{definition}

Here then is the crucial fractional decomposition theorem (which we think is nice in its own right).

\begin{thm}\label{lem:BalancedDecomposition}
    For every graph $F$ and real $\varepsilon\in (0,1)$, there exist a real $\sigma\in (0,1)$ and an integer $C\geq 1$ such that the following holds for every integer $n$ large enough.    
    
    If $G$ is a $n$-vertex graph with $\delta(G)\geq (\delta^*_F+\varepsilon)n$, then there exists a $(\sigma,C)$-balanced fractional $F$-decomposition of $G$.
\end{thm}

Note that~\cref{lem:BalancedDecomposition} is strictly stronger than~\cref{lem:LowWeightDecomposition} and so we only prove this stronger version (since we do not use the former as a black-box but rather adapt its would-be proof). We do this by proving a lemma (\cref{lem:Seeded}) which shows that in the high minimum degree setting, there exists fractional $F$-decomposition where each copy of $F$ has weight at least $\sigma/ \binom{n-2}{v(F)-2}$ for some $\sigma\in (0,1)$ only depending on $F$ and $\varepsilon$ (what we call `seeded' fractional decompositions). The key idea is again a probabilistic trick; we partition the copies of $F$ into edge-disjoint sets of maximum degree at most $\varepsilon n/2$. For each set of copies $\mc{S}$, we build a decomposition of $G$ using $\mc{S}$ and a fractional $F$-decomposition of the remaining graph, $G-\bigcup \mc{S}$, as guaranteed by the definition of $\delta_F^*$. We then average these decompositions to generate a seeded fractional decomposition as desired. Finally,~\cref{lem:LowWeightDecomposition} follows by combining seeded fractional decompositions with our fixing idea (to yield~\cref{lem:SeededFixed}) and then the Inheritance Lemma. See Section~\ref{s:RegBoosting} for these proofs.

 
\subsection{Outline of Rest of Paper}

In Section~\ref{s:NW}, we prove our Nash-Williams Refined Efficient Omni-Absorber Theorem,~\cref{thm:NWRefinedEfficientOmniAbsorber}, via embedding fake edge and absorbers of low rooted degeneracy. We actually first prove a General Embedding Lemma,~\cref{lem:EmbedGeneral}, to simplify this process which might prove to be of further use (we include its proof, even of its hypergraph generalization in the appendix). Finally, we end the section with a proof of the Reserves Lemma,~\cref{lem:RandomXHighMinDeg}, via Chernoff bounds. 

In Section~\ref{s:RegBoosting}, we prove our Nash-Williams Restricted Boosting Lemma,~\cref{lem:NWRegBoostTreasury}. As outlined in the proof overview, this is accomplished by first proving a seeded fractional decomposition lemma,~\cref{lem:Seeded}. This is then combined with the fixing idea to prove a target weight version of the seeded lemma,~\cref{lem:SeededFixed}. We then invoke this lemma combined with the Inheritance Lemma to prove our balanced fractional decomposition theorem,~\cref{lem:BalancedDecomposition}. Finally we use sampling, Chernoff bounds, and the General Boosting Lemma from~\cite{DKPIII} (reproduced as Lemma~\ref{lem:GeneralBoostLemma} below) to prove~\cref{lem:NWRegBoostTreasury}. 

In Section~\ref{s:HighGirth}, we prove an \Erdos{}-Nash-Williams Refined Efficient Omni-Absorber Theorem,~\cref{thm:HighGirthAbsorber}. We do this by starting with an omni-absorber given by our Nash-Williams Refined Efficient Omni-Absorber Theorem,~\cref{thm:NWRefinedEfficientOmniAbsorber}, and then using girth boosters and a theorem we extract from the proof of the High Girth Existence Conjecture (\cref{thm:HighGirthQuantumOmniBooster}), whose extraction is discussed in the appendix. This is mostly straightforward given the existence of high rooted girth $K_q$-boosters of rooted degeneracy $2q-2$ (as was shown in~\cite{DPII} but we include a proof for completeness). Nevertheless, to understand all the terminology in the theorem statements as well as the statement of Forbidden Submatchings with Reserves, requires many definitions and hence many pages. 

In Section~\ref{sec:Finishing}, we derive~\cref{thm:Main_ErdosNashWilliams} from the above theorems; this is mostly straightforward checking that theorems work together as promised in the proof overview but we include a proof for completeness. We also discuss the simple modifications to prove the more general main result,~\cref{thm:GenralisationMain2q}.

Finally, in Section~\ref{sec:Concluding}, we discuss questions about improving the bound for general $q$, proving a higher uniformity version, and other open questions. 

\section{Nash-Williams Omni-Absorber}\label{s:NW}

In this section, we both prove our Nash-Williams Omni-Absorber Theorem (\cref{thm:NWRefinedEfficientOmniAbsorber}) as well as at the end of this section, we prove the Reserves Lemma (\cref{lem:RandomXHighMinDeg}). 

As described in the proof overview, we have by Theorem~\ref{thm:RefinedEfficientOmniAbsorber} that there exists an omni-absorber $A\subseteq K_n\setminus X$ for our reserves $X$. Unfortunately, $A$ may use edges from $K_n\setminus G$ (our main result is not necessarily for high girth). To workaround this first issue, we follow the strategy of~\cite{DPI} by replacing the edges in $A$ with a gadget called a `fake edge' that has the same divisibility properties as an edge. These gadgets can be shown to exist in large numbers inside $G\setminus X$ at our minimum degree level; from this, we can show that we can embed all of them edge-disjointly inside $G\setminus X$ with low maximum degree. Of course, the cliques in the decomposition family will now be `fake cliques' and so we must embed an absorber for each of these to restore it to yielding a $K_q$-decomposition. This is only possible since the original omni-absorber is $C$-refined (so there are not too many private absorbers to embed). Once again, these absorbers can be shown to exist in large numbers inside $G$ at our minimum degree level; from this, we can show that we can embed all of them edge-disjointly inside $G$ with low maximum degree. Later we will also need to embed girth boosters so as to boost the girth of our omni-absorber. Since we have three embedding steps to accomplish, it is easier to craft a general embedding lemma (which can also be re-used later). Unfortunately, the embedding lemma in~\cite{DPI} only works with minimum degree $(1-\varepsilon)n$ and so we will need to prove a stronger form that can work in the setting of Nash-Williams' Conjecture.  

Thus we begin this section with definitions and results relevant to these embedding steps and then state the graph case of our General Embedding Lemma. 

\subsection{The Embedding Lemma}

First recall that a graph $W$ is a \emph{supergraph} of graph $H$ if $V(H)\subseteq V(W)$ and $E(H)\subseteq E(W)$. We define a supergraph system to describe the relationship between the gadgets we wish to embed and the subgraphs we wish to embed them on. To facilitate the embedding process, we introduce the notion of $C$-boundedness to ensure that these gadgets are not ``too big.''

\begin{definition}[$C$-bounded Supergraph System]
Let $\mc{H}$ be a family of subgraphs of a graph $J$. A \emph{supergraph system} $\mc{W}$ for $\mc{H}$ is a family $(W_H : H\in \mc{H})$ where for each $H\in \mc{H}$, $W_H$ is a supergraph of $H$ with $V(W_H)\cap V(J)\neq\emptyset$ and for all $H'\ne H\in \mc{H}$, we have $V(W_H)\cap V(W_{H'}) \setminus V(J)= \emptyset$. We let $\bigcup \mc{W}$ denote $\bigcup_{H\in \mc{H}} W_H$ for brevity. For a real $C \geq 1$, we say that $\mc{W}$ is \emph{$C$-bounded} if $\max\{e(W_H),~v(W_H)\}\le C$ for all $H\in \mc{H}$.
\end{definition}

\noindent
We now formalize the notion of embedding a supergraph system into a host graph $G$ as follows.

\begin{definition}[Embedding a Supergraph System]
Let $J$ be a graph and let $\mc{H}$ be a family of subgraphs of $J$. Let $\mc{W}$ be a supergraph system for $\mc{H}$ and let $G$ be a supergraph of $J$. An \emph{embedding} of $\mc{W}$ \emph{into} $G$ is a map $\phi : V(\bigcup \mc{W}) \hookrightarrow V(G)$ preserving edges such that $\phi(v)=v$ for all $v\in V(J)$. We let $\phi(\mc{W})$ denote $\bigcup_{e\in \bigcup W} \phi(e)$ (i.e.~the subgraph of $G$ corresponding to $\bigcup \mc{W}$).
\end{definition}

Equipped with these notions, we are now prepared to formally state an embedding lemma. We note that although the general embedding lemma applies to multi-hypergraphs, we restrict our attention to the graph subcase in Lemma~\ref{lem:EmbedGeneral}.  This is convenient for us, not only because it is all that is needed for our current arguments, but also because it allows us to sweep under the rug a technical condition that is trivially satisfied on graph provided $J$ has no isolated vertices. We direct the reader to~\cref{sec:ProofEmbedding} for the statement and proof of the hypergraph generalization of this general embedding lemma (we omit a direct proof of Lemma~\ref{lem:EmbedGeneral} because limiting our attention to the graph case does not significantly simplify the argument). Similarly to the notion of refined absorbers, for a real $C \geq 1$, we say that a family $\mc{H}$ of subgraphs of a graph $J$ is {\em $C$-refined} if every edge $e\in J$ is contained in at most $C$ elements of $\cH$.

\begin{lem}[General Embedding for Graphs]\label{lem:EmbedGeneral}
For every reals $C\geq 1$ and $\gamma\in (0,1]$, there exists $C'\ge 1$ such that the following hold. Let $J\subseteq G\subseteq K_n$ with $\Delta(J)\le \frac{n}{C'}$ with no isolated vertices. Suppose that $\mc{H}$ is a $C$-refined family of subgraphs of $J$ and $\mc{W}$ is a $C$-bounded supergraph system of $\mc{H}$. If for each $H\in \mc{H}$ there exist at least $\gamma\cdot n^{|V(W_H)\setminus V(H)|}$ embeddings of $W_H$ into $G\setminus (E(J)\setminus E(H))$,
then there exists an embedding $\phi$ of $\mc{W}$ into $G$ such that $\Delta(\phi(\mc{W})) \le C'\cdot \Delta(J)$.
\end{lem}

The proof of the hypergraph generalization~\cref{lem:EmbedGeneral} in the appendix is a generalization of the ``avoid the bad" proof of the embedding lemma in~\cite{DPI}. We note that Delcourt, Kelly, and Postle~\cite[Lemma~4.5]{DKPIV} used a variant of the Lov\'asz Local Lemma and a ``slotting" argument to prove an embedding lemma for rooted cliques in a graph $G$ with high minimum degree (at least $(1-\varepsilon)n$ for a sufficiently small $\varepsilon>0$). Their lemma further yields a ``spread'' distribution over all such embeddings. We note that a similar strategy using the Lov\'asz Local Lemma and a ``slotting" argument could also be used to prove~\cref{lem:EmbedGeneral}. Similarly, we do not pursue a `spread' version of the above lemma here (as it is not needed for our purposes).

\subsection{Nash-Williams Omni-Absorbers}\label{subsec:omni}

To embed a refined omni-absorber $A$ into the graph $G$, we start by replacing edges of $A$ by so-called ``fake-edges'' formally defined below, which have small rooted degeneracy, and hence we can embedded them in $G$. Then, for every $F\cong K_q$ in the decomposition family of $A$, we take a low degeneracy $K_q$-absorber for the graph $F'$, built from $F$ by replacing the edges of $A$ by the embedding of their respective fake-edges. To that end, we recall the following gadgets from~\cite{DPI} restricted to the graph case.

\begin{definition}\label{def:BasicGadget}
Let $q\geq 3$ be an integer. Let $S$ be a set of $2$ vertices.
\begin{itemize}\itemsep.05in
    \item An \emph{anti-edge on $S$}, denoted ${\rm AntiEdge}_q(S)$, is a set of new vertices $x_1,\ldots, x_{q-2}$ together with edges $\binom{S\cup \{x_i:~i\in[q-2]\} }{2} \setminus \{S\}$.
    \item A \emph{fake-edge on $S$}, denoted ${\rm FakeEdge}_q(S)$, is a set of new vertices $x_1,\ldots, x_{q-2}$ together with anti-edges $\{ {\rm AntiEdge}_q(T): T\in \binom{S\cup \{x_i:~i\in[q-2]\} }{2} \setminus \{S\} \}$.
\end{itemize}    
\end{definition}

We note that every fake-edge has at most $q^{3}$ vertices and at most $q^{4}$ edges. We note that fake edges have the same divisibility properties as actual edges, that is, if $F={\rm FakeEdge}_q(S)$, then $e(F) \equiv 1 \mod \binom{q}{2}$, $d_F(v)\equiv 1 \mod (q-1)$ for every $v\in S$ and $d_F(v)\equiv 0 \mod (q-1)$ for every $v\in V(F)\setminus S$. Therefore, replacing an edge $e$ of a graph $J$, where $V(e)=S$, by a fake edge on $S$ maintains the divisibility properties of $J$. We note for the reader that a collection of (well-chosen) fake edges can be viewed as a \textit{refiner}, a key concept in the proof of~\cref{thm:RefinedEfficientOmniAbsorber} from \cite{DPI}, which we do not require here. The capacity to embed small graphs into a host graph with high enough minimum degree is then related to the following notion.

\begin{definition}[Rooted degeneracy]\label{def:RootedDegeneracy}
For a graph $W$ and $U\subseteq V(W)$, the \emph{degeneracy of $W$ rooted at $U$} is the smallest non-negative integer $d$ such that there exists an ordering $v_1,\ldots, v_{v(W)-|U|}$ of the vertices $V(W)\setminus U$ such that for all $i\in [v(W)-|U|]$, we have $|N_W(v_i)\cap (U\cup \{v_j: 1\le j < i\})| \le d$.  
\end{definition}

We note that a fake-edge on $S$ has degeneracy rooted at $S$ at most $q-1$. The following easy lemma shows that high enough minimum degree guarantees a constant proportion of embeddings of small rooted degeneracy subgraphs.

\begin{lem}\label{lemma:DegeneracyEmbedding}
    For every real $\varepsilon\in(0,1)$ and integer $d\geq 2$, there exists $\gamma\in(0,1)$ such that the following holds for every fixed graph $W$ and for every $n$ large enough. Let $G$ be a graph on $n$ vertices with minimum degree at least $(1-\frac{1}{d}+\varepsilon)n$. Let $H$ be a subgraph of $G$ such that $W$ is a supergraph of $H$ with degeneracy rooted at $V(H)$ at most $d$. 
    There are at least $(\gamma\cdot n)^{|V(W)\setminus V(H)|}$ embeddings of $W$ in $G$.
\end{lem}

\begin{proof}
    Let $h=|V(H)|$ and $w=|V(W)|$. Fix an ordering $v_1,\ldots, v_{w-h}$ of the vertices $V(W)\setminus V(H)$ such that for all $i\in [w-h]$, we have $|N_W(v_i)\cap (V(H)\cup \{v_j: 1\le j < i\})| \leq d$. We embed the vertices $v_1,\ldots, v_{w-h}$ one at a time in $G$. For every $i\in [w-h]$, there are at least $(d\varepsilon n-v(W_H))$ choices for the vertex $v_i$. As $n$ is large enough, there are at least $(d\varepsilon n/2)^{w-h}$ embeddings of $W$ in $G$.
\end{proof}

We also need absorbers of rooted degeneracy at most $2q-2$, as were shown to exist by Delcourt, Kelly, and Postle in~\cite{DKPIII}. 

\begin{thm}[\cite{DKPIII}]\label{thm:AbsorbersSmallDegeneracy}
    For every integers $q\geq 3$ and $C_L\geq 1$, there exists a positive integer $C_A$ such that for every $K_q$-divisible graph $L$ on at most $C_L$ vertices, there exists a $K_q$-absorber $A$ of $L$ on at most $C_A$ vertices with rooted degeneracy rooted at $V(L)$ at most $2q-2$.
\end{thm}

Armed with the tools of low degeneracy fake-edges and absorbers, we are now ready to prove that we can embed refined absorbers into a graph with high enough minimum degree.

\begin{proof}[Proof of Theorem~\ref{thm:NWRefinedEfficientOmniAbsorber}]
    Fix $q\geq 3$ and $\varepsilon\in(0,1)$. Let $C_A\geq q^4$ be such that~\cref{thm:AbsorbersSmallDegeneracy} holds for $C_L=q^4$, that is, there exists an absorber with small degeneracy on at most $C_A$ vertices for every $K_q$-divisible graph on at most $q^4$ vertices. Let $C_0\geq C_A$ be an integer such that~\cref{thm:RefinedEfficientOmniAbsorber} (the existence of $C_0$-refined $K_q$-omni-absorbers in $K_n$) holds. 
    Let $\gamma\in(0,1)$ be such that~\cref{lemma:DegeneracyEmbedding} holds for $\varepsilon/4$ and $d=2q-2$, and let $\gamma'>0$ be small enough compared to $\gamma$ and $q$. Let $C'\geq 1$ such that the General Embedding Lemma,~\cref{lem:EmbedGeneral}, holds for $C_0^2$ and $\gamma'$.
    Finally, let $C:= 4\cdot (C')^2\cdot C_0 /\varepsilon$. 

    For $n$ large enough, let $X\subseteq  G\subseteq K_n$ with $\delta(G)\ge (1-\frac{1}{2q-2} + \varepsilon)n$ and $\Delta(X) \leq \frac{n}{C}$. Let $\Delta:=\max\left\{\Delta(X),~\sqrt{n}\cdot \log n\right\}$. Let $G_0 := G\setminus X$ and note that $G_0$ has minimum degree at least $(1-\frac{1}{2q-2} + \frac{\varepsilon}{2})n$. Recall that we aim to prove that there exists a $C$-refined $K_q$-omni-absorber $A$ for $X$ in $G$ with $\Delta(A) \le C\cdot \max\left\{\Delta(X),~\sqrt{n}\cdot \log n\right\}$. We do so by taking a refined omni-absorber in $K_n$, replacing each one of its edges by a fake edge in $G$, and therefore each copy of $K_q$ in its decomposition family by a graph $F$, and then finding private absorber in $G$ for each such graph $F$.

    By~\cref{thm:RefinedEfficientOmniAbsorber} there exists a $C_0$-refined $K_q$-omni-absorber $A_0$ for $X$ with $\Delta(A_0) \leq C_0\cdot \Delta\leq C_0n/C$, with decomposition family $\mc{F}_0$ and decomposition function $\mc{Q}_0$.
    Let $J:=A_0$ and $\mc{H}:=\set*{\{e\}\colon e\in J}$. We define a supergraph system $\mc{W}:=(W_e: \{e\}\in \mc{H})$ where $W_e:= {\rm FakeEdge}_q(e)$. We start by embedding $\mc{W}$ into $G_0$. Observe that, by definition of $C$, we have $C_0/C\geq \varepsilon /4$. Therefore, for every edge $e\in J$, and for $n$ large enough, $G_0\setminus(J\setminus \{e\})$ has minimum degree at least
    \[\Big(1-\frac{1}{2q-2}+\frac{\varepsilon}{2}\Big)n-\Delta(A_0)\geq \Big(1-\frac{1}{2q-2}+\frac{\varepsilon}{4}\Big)n,\]
    while $W_e$ has degeneracy rooted at $V(e)$ at most $q-1$. Therefore, by~\cref{lemma:DegeneracyEmbedding}, there are at least $(\gamma\cdot n)^{|V(W_e)\setminus V(e)|}\geq \gamma'\cdot n^{|V(W_e)\setminus V(e)|}$ embeddings of $W_e$ into $G_0\setminus (J\setminus \{e\})$. 
    Note that $\mc{W}$ is $q^4$-bounded and that~$\mc{H}$ is $1$-refined. By~\cref{lem:EmbedGeneral}, as $C_0^2\geq q^4$, there exists an embedding $\phi$ of $\mc{W}$ into $G_0$ such that $\Delta(\phi(\mc{W})) \leq C'\cdot \Delta(J)$.\bigskip

    Let $J':=X\cup \phi(\mc{W})$. Observe that with $C\geq 2(C')^2\cdot C_0$, we have 
    $$\Delta(J')\leq \Delta(X)+C'\cdot C_0\cdot\Delta\leq 2\cdot C'\cdot C_0\cdot \frac{n}{C}\leq \frac{n}{C'}$$ 
    with no isolated vertices. Let
    \[\mc{H}':=\big\{ (F\cap X)\cup \bigcup_{e\in F\setminus X} \phi(W_e)\colon F\in \mc{F}_0\big\}.\]
    Observe that $\mc{H}'$ is a $C_0$-refined family of subgraphs of $J'$. For every $F\in\mc{H}'$, let $A_F$ be a $K_q$-absorber for $F$ on at most $C_A\leq C_0$ vertices and with degeneracy rooted at $V(F)$ at most $2q-2$ as guaranteed by~\cref{thm:AbsorbersSmallDegeneracy}, and such that $A_F$ and $A_{F'}$ are edge-disjoints for distinct $F,F'\in\mc{H}'$. We define a supergraph system $\mc{W}':=(W_F: F\in \mc{H}')$ where $W_F=F\cup A_F$. Note that every element of $\mc{W}'$ has at most $q^4+\binom{C_0}{2}$ edges, hence $\mc{W}'$ is $C_0^2$-bounded.

    Observe that for every $F\in \mc{H}'$, $W_F$ has degeneracy at most $2q-2$, while $G_0\setminus (J'\setminus F)$ has minimum degree at least $(1-\frac{1}{2q-2}+\frac{\varepsilon}{4})$. Therefore, by~\cref{lemma:DegeneracyEmbedding}, for every $F\in \mc{H}'$, there are at least $\gamma'\cdot n^{|V(W_F)\setminus V(F)|}$ embeddings of $W_F$ into $G_0\setminus (J'\setminus F)$. By~\cref{lem:EmbedGeneral}, there exists an embedding $\phi'$ of $\mc{W}'$ into $G_0$ such that $\Delta(\phi'(\mc{W}')) \le C'\cdot \Delta(J')$.

    For each $F \in \mc{F}_0$, let $F'\in\mc{H}'$ such that $F':= (F \cap X) \cup \bigcup_{e \in F \setminus X} \phi(W_e)$. Let $A=\phi(\mc{W})\cup\phi'(\mc{W}')$.
    We first define a decomposition family $\mc{F}_A$ for $A$ as follow. We include in $\mc{F}_A$ both the $K_q$-decomposition of $\phi'(A_{F'})$ and of $F' \cup \phi'(A_{F'})$. Observe that these decompositions exist because $A_{F'}$ is a $K_q$-absorber of $F'$. 
    
    We then define a decomposition function $\cQ_A$ for $A$ as follows. Given a $K_q$-divisible $L \subseteq X$, for every $F\in \mc{F}_0$, we let $\cQ_{A}(L)$ contain the $K_q$-decomposition of $F' \cup \phi'(A_{F'})$ if $F \in \cQ_{A_0}(L)$ and of $\phi'(A_{F'})$ if $F \in \mc{F}_0\setminus \cQ_{A_0}(L)$.
    
    By construction, we have $\Delta(A)\leq \Delta(\phi(\mc{W}))+\Delta(\phi'(\mc{W}')) \le 3C'\cdot C_0 \cdot \Delta\leq C\cdot \Delta$, while every edge of $A\cup X$ is in at most $C_0+2$ elements in $\mc{F}_A$. We obtain that $A$ is a $C$-refined $K_q$-omni-absorber for $X$ in $G$ with decomposition family $\mc{F}_A$ and decomposition function $\cQ_A$, as desired. 
\end{proof}

\subsection{The Reserve Lemma}\label{sec:ReserveLemma}

In this subsection, we prove our reserve lemma via standard applications of the Chernoff bound.

\begin{lateproof}{lem:RandomXHighMinDeg}
Let $\gamma > 0$ be small enough, in particular such that $\gamma(q-1)<1$. Let $n$ be a large enough integer. Let $X$ be the spanning subgraph of $G$ obtained by choosing each edge of $G$ independently with probability $p\in [n^{-\gamma},1)$.

For every vertex $v\in V(G)$, we have \(\Expect{d_X(v)}=p\cdot d_G(v)\). By the Chernoff bound~\cite{AS16}, we find that \(\Prob{d_X(v)\geq 2p\cdot d_G(v)} \leq e^{-p\cdot d_G(v) / 3}\). By the Union bound, since $p\geq n^{-\gamma}$, we obtain
\[\Prob{ \Delta(X)>2pn} \leq \Prob{\exists v\in V(G)\colon d_X(v)\geq 2p\cdot d_G(v)} <ne^{-p n(1+\varepsilon-1/(q+1)) / 3} = o(1).\]

Recall that $G$ is a spanning subgraph of $K_n$. Fix a set $T\subseteq V(G)$ with $1\le |T|\le q-1$, and let $t:=|T|$.
We define $N_G(T)$, respectively $N_X(T)$, to be the common neighborhood of all vertices of $T$ in $G$, respectively in $X$, that is, 
\[N_G(T):= \left\{ v\in V(G)\setminus T: uv \in E(G) \text{ for every } u\in T\right\},\]
and
\[N_X(T):= \left\{ v\in V(G)\setminus T: uv \in E(X) \text{ for every } u\in T\right\}.\]

Since $\delta(G)\ge (1-\frac{1}{q+1}+\varepsilon)n$ and $t\leq q-1$, it follows that \(|N_G(T)| \geq 2n/(q+1)\). For every $v\in N_G(T)$, we have that $\Prob{v\in N_X(T)} = p^t$. Thus by Linearity of Expectation, we obtain
\[\Expect{|N_X(T)|} = \sum_{v\in N_G(T)} \Prob{v\in N_G(T)} = p^t\cdot |N_G(T)|.\]

Note that $|N_X(T)|$ is the sum of independent Bernoulli $\{0,1\}$-random variables. Let $B_T$ be the event that $|N_X(T)| < p^t\cdot |N_G(T)|/2$.  By the Chernoff bound, we find that
\[\Prob{B_T} \le \exp\left[-p^t\frac{|N_G(T)|}{8}\right]\leq \exp\left[-n^{t\gamma}\frac{n}{4(q+1)}\right]\]
hence 
\[\Prob{ \bigcup_{T} B_T} \leq  2n^{q-1}\cdot\exp\left[-\frac{n^{1-(q-1)\gamma}}{4(q+1)}\right]=o(1),\]
because $t\leq q-1$ and $\gamma<1/(q-1)$. Therefore with positive probability, $\Delta(X)\leq 2pn$ and none of the $B_T$ happen. Fix such an outcome. For $e\in K_n\setminus X$, since none of the events $B_T$ happen, the number of copies of $K_q$ in $X\cup \{e\}$ containing $e$ is at least 

\begin{align*}
\prod_{t=2}^{q-1} \left(\frac{p^t}{2}\cdot \frac{2n}{q+1}\cdot\frac{1}{t+1}\right)
\geq \frac{1}{(q+1)^{q-2}} \cdot p^{\binom{q}{2}-1} \cdot \frac{n^{q-2}}{q!}
&\geq \frac{1}{(q+1)^{q}}\cdot p^{\binom{q}{2}-1}\cdot \binom{n}{q-2},
\end{align*}
where we used the standard bound that $\frac{n^k}{k!} \geq \binom{n}{k}$.
\end{lateproof}

\section{Regularity Boosting}\label{s:RegBoosting}

In this section, we prove~\cref{lem:NWRegBoostTreasury}. First we need some additional notation and definitions as follows. For graphs $F,G$, we denote the set of copies of $F$ in $G$ by $\binom{G}{F}$. Given an $F$-weighting $\psi$ of a graph $G$ and an edge $e$ of $G$, we let $\partial\psi(e)$ denote the total weight of $\psi$ over the edge $e$, that is $\partial\psi(e):= \sum_{F: e\subseteq F} \psi(F)$. Similarly, for a set of edges $R$ of $G$, we let $\partial\psi(R):= \sum_{F: V(R)\subseteq F} \psi(Q)$.

Throughout this section, if $\cH$ is a family of copies of $K_q$ in a graph $G$, we use notation as if $\cH$ is a hypergraph with $V(\cH) = V(G)$ and $E(\cH) = \cH \subseteq \binom{V(G)}{q}$. In particular for a set of vertices $R$, we denote by $\cH(R)$ the set of edges of $\cH$ containing $R$. We say that a hypergraph $G=(A,B)$ is \emph{bipartite with parts $A$ and $B$} if $V(G)=A\cup B$ and every edge of $G$ contains exactly one vertex from $A$; we further say that a \emph{matching} $M$, a set of vertex disjoint edges, of $G$ is \emph{$A$-perfect} if every vertex of $A$ is in an edge of the matching $M$.

\subsection{Seeded Fractional Decompositions}

We start by building a {\em seeded} fractional decomposition, that is, one where the weight on each graph is bounded below.

\begin{lem}[Seeded Fractional Decomposition]\label{lem:Seeded}
    For every graph $F$ and real $\varepsilon$, there exist $\sigma\in(0,1)$ and an integer $n_0\geq 1$, such that following holds for every $n\geq n_0$. Every graph $G$ on $n$ vertices with minimum degree at least $(\delta^*_F+\varepsilon)n$ admits a fractional $F$-decomposition such that  every copy of $F$ has weight at least $\sigma / \binom{n-2}{v(F)-2}$.
\end{lem}

We prove~\cref{lem:Seeded} by finding a perfect matching in a well-chosen auxiliary bipartite hypergraph. There are many results establishing sufficient conditions on a bipartite hypergraph $\mc{G} = (A,B)$ to guarantee the existence of an $A$-perfect matching, generalizing the classical Hall’s theorem to hypergraphs. For our purposes, it suffices to use the following lemma (which is a corollary of Proposition 5.3 of Alon~\cite{alon1994probabilistic}) whose proof follows easily from the Lov\'{a}sz Local Lemma; we note the best result of this form follows as a corollary from the work of Haxell~\cite{H95} (specifically Theorem 3 in~\cite{H95} which shows that degree $(2k-3)D$ would suffice below) which is essentially optimal. 

\begin{lem}\label{thm:AlonMatching}
        Let $\mc{G}=(A,B)$ be a $k$-uniform bipartite hypergraph. Assume that for some integer $D$, every vertex $v\in A$ has degree at least $2ekD$, and every vertex $b\in B$ has degree at most $D$. Then $\mc{G}$ contains an $A$-perfect matching. 
\end{lem}

We are now prepared to prove~\cref{lem:Seeded}. 

\begin{proof}[Proof of~\cref{lem:Seeded}]
    Let $q:=v(F)$, $C:=8e\cdot q^{q+2}\cdot\binom{n-2}{q-2}\cdot\varepsilon^{-1}$, and let $v_1,\ldots,v_{q}$ be an ordering of $V(F)$. Let $M := \floor*{\varepsilon n/(2q)}$. We build an auxiliary bipartite hypergraph $\mc{G}=(A,B)$ where 
    \[A:=\binom{G}{F},\qquad B:=\big\{(e,c): e\in E(G), c\in[C]\big\}\cup\big\{(v,c,m)\colon v\in V(G), c\in [C], m\in[M] \big\},\]
    and for every $H\in \binom{G}{F}$, every $c\in[C]$, and every $m_1,\ldots,m_{q}\in[M]$, $E(\mc{G})$ contains the edge
    \[ \{H\}\cup \bigcup_{e\in H}(e,c)\cup \bigcup_{v_j\in V(H) } (v_j,c,m_j).\]
    
    For every edge $e$ and every color $c$, there is a unique vertex $(e,c)$ in $B$; therefore an $A$-perfect matching of $\mc{G}$ represents a $C$-coloring of the copies of $F$ in $G$ such that for every color $c\in [C]$, all $c$-colored copies of $F$ are pairwise edge-disjoint. 
    Furthermore, for every vertex $v$ and every color $c$, there are exactly $M$ vertices in $B$ of the form $(v,c,m)$ for some integer $m$; therefore an $A$-perfect matching of $\mc{G}$ is a $C$-coloring of the copies of $F$ in $G$ such that for every $c\in[C]$, every vertex appears in at most $M$ $c$-colored copies of $F$. 

    Let $k:=1+e(F)+q\leq q^2$ be the uniformity of $\mc{G}$. Every edge $e\in G$ is in at most $q^q\cdot\binom{n-2}{q-2}$ copies of $F$, while every vertex $v\in V(G)$ is in at most $q^q\cdot\binom{n-1}{q-1}$ copies of $F$. Then, every vertex $b\in B$ has degree at most
    \[\max\set*{q^{q}\cdot\binom{n-2}{q-2}\cdot M^{q},~q^{q}\cdot\binom{n-1}{q-1}\cdot M^{q-1}}
    \leq
    \binom{n-2}{q-2}\cdot q^{q-1}\cdot M^{q-1}\cdot n=:D\]
    Meanwhile, every vertex $a\in A$ has degree
    \[C\cdot M^{q} 
    \geq \frac{8eq^{q+2}}{\varepsilon}\binom{n-2}{q-2} \cdot M^{q-1} \cdot \frac{\varepsilon n}{4q}
    \geq 2eq^2 \cdot \binom{n-2}{q-2} \cdot q^{q-1} \cdot M^{q-1} \cdot n \geq 2ekD.\]
     Therefore, by~\cref{thm:AlonMatching}, $\mc{G}$ contains an $A$-perfect matching. This matching encodes a $C$-coloring of the copies of $F$ in $G$. For each $c\in [C]$, let $G_c$ be the subgraph of $G$ induced by the edges contained in a $c$-colored copy of $F$. Observe that for every color $c\in [C]$,
    \begin{enumerate}
        \item $G_c$ is the union of edge-disjoint copies of $F$,
        \item every vertex appears in at most $M$ copies of $F$ colored $c$, hence $\Delta(G_c)\leq \varepsilon n/2$.
    \end{enumerate}
    
    For each color $c\in [C]$, we have $\delta(G\setminus G_c)\geq (\delta^*_F+\varepsilon/2)n$, therefore there exists a fractional $F$-decomposition $\phi_c$ of $G\setminus G_c$. We extend $\phi_c$ by adding a weight of $1$ to every copy of $F$ in $G_c$, forming a fractional $F$-decomposition of $G$ by $(i)$. We then define
    \[\phi(H) := \frac{1}{C}\sum_{c\in[C]}\phi_c(H),\]
    which is the desired fractional $F$-decomposition of $G$, where every copy of $F$ has weight at least $\frac{1}{C}=\frac{\sigma}{\binom{n-2}{q-2}}$ with $\sigma = \frac{\varepsilon}{8e\cdot q^{q+2}}$.
\end{proof}

We now extend~\cref{lem:Seeded} to a {\em fixed} version, allowing us to set a target for the weight on every edge instead of a weight of $1$.
\begin{lem}[Seeded-Fixed Fractional Decomposition]\label{lem:SeededFixed}
    For every graph $F$ and real $\varepsilon$, there exist $\sigma\in(0,1)$ and an integer $n_0\geq 1$, such that following holds for every $n\geq n_0$: Let $G$ be a graph on $n$ vertices with minimum degree at least $(\delta^*_F+\varepsilon)n$, and let $\varphi:E(G)\to[1-\frac{1}{e(G)},1]$. Then there exists a fractional $F$-packing $\Phi$ of $G$ such that every copy of $F$ has weight at least $\sigma / \binom{n-2}{v(F)-2}$, and $\partial\Phi(e)=\varphi(e)$ for every edge $e\in G$.
\end{lem}

\begin{proof}
Note that as $n$ is large enough, we have that $\varepsilon n\geq2$ and hence $\delta(G-e)\geq(\delta^*_F+\varepsilon/2)n$ for every $e\in G$. Let $\sigma'$ be as the $\sigma$ in~\cref{lem:Seeded} for $F$ and $\varepsilon/2$. By~\cref{lem:Seeded} as $n$ is large enough, there exists a ``seeded'' fractional $F$-decomposition $\Phi_0$ of $G$, where every copy of $F$ in $G$ has weight at least $\sigma' / \binom{n-2}{v(F)-2}$. Similarly, for every $e\in G$ by~\cref{lem:Seeded}, there exists a fractional $F$-decomposition  $\Phi_{e}$ of $G-e$ where every copy of $F$ in $G-e$ has weight at least $\sigma' / \binom{n-2}{v(F)-2}$.  Let $\sigma:=\frac{\sigma'}{2}$.

For each $e\in G$, let $\lambda_e := e(G)\Big[\varphi(e) - \big(1-\frac{1}{e(G)}\big)\Big]$. Note that since $\varphi(e) \in [1-\frac{1}{e(G)},1]$, we have that $\lambda_e\in [0,1]$. We define 
$$\Phi_e' := \lambda_e \cdot \Phi_0 + (1-\lambda_e) \cdot \Phi_e.$$ 
Note that $\Phi'_e$ is fractional $F$-packing of $G$ such that $\partial \Phi'_e(e) = \lambda_e$ and $\partial \Phi'_e(f) = 1$ for each $f\in G-e$. 
Finally we set
$$\Phi := \frac{1}{e(G)} \cdot \sum_{e\in G} \Phi'_e.$$
Observe that $\Phi$ is a fractional $F$-packing of $G$ (since it is the average of fractional $F$-packings of $G$). Furthermore for each $f\in G$, we have that
\begin{align*}
\partial \Phi(f) = \frac{1}{e(G)} \cdot \sum_{e\in G} \partial \Phi'_e(f) = \frac{1}{e(G)} \cdot \left( e(G)-1 + \lambda_f\right) = 1 - \frac{1}{e(G)} + \frac{1}{e(G)} \cdot \lambda_f = \varphi(f),
\end{align*}
as desired. Finally, we verify that the fractional packing assigns the desired minimum weight to each copy of $F$ in $G$. For any such copy $H$ we have
\begin{align*}
\Phi(H)
&= 
\frac{1}{e(G)} \cdot \sum_{e\in G} \lambda_e \cdot \Phi_0(H) + (1-\lambda_e) \cdot \Phi_e(H)\\
&\geq 
\frac{1}{e(G)} \cdot \sum_{e\in G} \lambda_e \cdot \frac{\sigma'}{\binom{n-2}{v(F)-2}} + (1-\lambda_e) \cdot \frac{\sigma'}{\binom{n-2}{v(F)-2}}\cdot\mathds{1}_{\{e\notin H\}}\\
&\geq \frac{e(G)-e(F)}{e(G)}\cdot\frac{2\sigma}{\binom{n-2}{v(F)-2}}    \\
&\geq \frac{\sigma}{\binom{n-2}{v(F)-2}},  
\end{align*}
as desired where we used that $H$ is seeded in $\Phi_e$ if and only if $e\not\in H$, that $\sigma'=2\sigma$, and that $e(F)\le \frac{e(G)}{2}$ as $n$ is large enough.
\end{proof}

\subsection{Balanced Fractional Decompositions}

We now use the fractional packings of~\cref{lem:SeededFixed}, to build a balanced fractional decomposition. Here is the special case of Lang's Inheritance Lemma~\cite{lang2023tiling} required to prove~\cref{lem:BalancedDecomposition}.

\begin{lem}[Inheritance Lemma for minimum degree~\cite{lang2023tiling}]\label{lem:LangInheritance}
    For every $\varepsilon\in(0,1)$ and integer $m\geq 1$, there exists $s_0$ such that for all $s\geq s_0$, the following holds for every $n$ large enough. Let $\delta\geq 0$, and let $G$ be an $n$-vertex graph with $\delta(G)\geq (\delta+\varepsilon)(n-1)$. Then for every m-set $M\subseteq V(G)$, there are at least $(1-e^{-\sqrt{s}})\binom{n-m}{s-m}$ distinct $s$-sets $S\subseteq V(G)$ such that $M\subseteq S$ and $\delta(G[S])\geq (\delta+\varepsilon/2)(s-1)$.
\end{lem}

We are now prepared to prove~\cref{lem:BalancedDecomposition}.

\begin{proof}[Proof of~\cref{lem:BalancedDecomposition}]
    Observe that we may assume that $\delta^*_F < 1$ as otherwise the statement is vacuously true. 
    Let $s_1\geq 1$ such that~\cref{lem:LangInheritance} holds for $\varepsilon$ and $m=2$; let $s_2\geq 1$ such that~\cref{lem:LangInheritance} holds for $\varepsilon$ and $m=v(F)$; and let $\sigma\in(0,1)$ and $s_3\geq 1$ such that~\cref{lem:SeededFixed} holds for $F$ and graphs $G$ on at least $s_3$ vertices. Let $s$ be large enough, in particular such that $s\geq\max\{s_1,s_2,s_3\}$. Let $C:=2\cdot\binom{s-2}{v(F)-2}$. 
    
    Let $n$ be large enough, and let $G$ be a $n$-vertex graph with $\delta(G)\geq (\delta^*_F+\varepsilon)n$. Let $\cS$ be the family of all $s$-sets in $V(G)$ with minimum degree at least $(\delta^*_F+\varepsilon/2)(s-1)$. For every edge $e\in G$, we define $\varphi(e)$ to be
    \[\varphi(e):=(1-e^{-\sqrt{s}})\cdot \frac{\binom{n-2}{s-2}}{|\{S\in\mc{S}:e\in S\}|}.\]

    By~\cref{lem:LangInheritance}, every edge $e\in G$ is in at least $(1-e^{-\sqrt{s}})\cdot\binom{n-2}{s-2}$ sets in $\mc{S}$. Therefore, by building, for every $S\in\mc{S}$, an $F$-decomposition $\Phi_S$ of $G[S]$ such that $\partial\Phi_S(e)=\varphi(e)$, and taking the sum of all these decompositions, we obtain an assignment such that the weight on each edge is constant; we then only need a scaling factor to obtain the desired $F$-decomposition.
    
    Observe that, by~\cref{lem:LangInheritance}, we have $\varphi(e)\in[1-e^{-\sqrt{s}},1]$ and hence as $s$ is large enough, we have $\varphi(e)\geq 1-\frac{1}{s^2}$.
    Thus by~\cref{lem:SeededFixed}, for each $S\in\cS$, there exists a seeded and fixed fractional $F$-packing $\Phi_S$ of $G[S]$, that is where every copy of $F$ has weight at least $\sigma/\binom{s-2}{v(F)-2}$ and such that $\partial\Phi_S(e)=\varphi(e)$ for every edge $e\in G[S]$. We let $\Phi_S(H)=0$ if $V(H)\setminus S\ne\emptyset $ so as to extend the function to all copies of $F$ in $G$. We define the following fractional $F$-packing of $G$. For every $H\in\binom{G}{F}$, define
    \[\phi(H):=\frac{1}{(1-e^{-\sqrt{s}})\cdot\binom{n-2}{s-2}}\cdot\sum_{S\in\cS}\Phi_S(H).\]
    By~\cref{lem:LangInheritance}, every $H\in \binom{G}{F}$ is in at least $(1-e^{-\sqrt{s}})\binom{n-v(F)}{s-v(F)}$ sets in $\cS$, therefore
    \[\phi(H)\geq 
    \frac{1}{\binom{n-2}{s-2}}\cdot \binom{n-v(F)}{s-v(F)}\cdot \frac{\sigma}{\binom{s-2}{v(F)-2}}
    = \frac{\sigma}{\binom{n-2}{v(F)-2}}.
    \]
    Similarly, every $H\in \binom{G}{F}$ is in at most $\binom{n-v(F)}{s-v(F)}$ sets $S\in\cS$, hence
    \[\phi(H)\leq \frac{1}{(1-e^{-\sqrt{s}})\cdot\binom{n-2}{s-2}}\cdot \binom{n-v(F)}{s-v(F)}
    = \frac{\binom{s-2}{v(F)-2}}{(1-e^{-\sqrt{s}})\cdot\binom{n-2}{v(F)-2}}
    \leq \frac{C}{\binom{n-2}{v(F)-2}},    \]
    where we used $(1-e^{-\sqrt{s}})\geq 1/2$ for $s\geq 1$. Finally, by definition of $\varphi$, every edge $e\in G$ is in exactly $\frac{1-e^{-\sqrt{s}}}{\varphi(e)}\cdot\binom{n-2}{s-2}$ sets in $\cS$, therefore we obtain 
    \[\partial\phi(e)=\frac{1}{(1-e^{-\sqrt{s}})\cdot\binom{n-2}{s-2}}\cdot\sum_{S\in\cS}\partial\phi_S(e)=\frac{1}{(1-e^{-\sqrt{s}})\cdot\binom{n-2}{s-2}}\cdot\brackets*{\frac{1-e^{-\sqrt{s}}}{\varphi(e)}\cdot\binom{n-2}{s-2}\cdot \varphi(e)} = 1.
    \]
\noindent
    Hence $\phi$ is the desired $(\sigma,C)$-balanced fractional $F$-decomposition of $G$.
\end{proof}

\subsection{From Decompositions To Regular Cliques}

To conclude the proof of the Nash-Williams Boosting Lemma, we require the following simplified version of the general boosting lemma from~\cite{DKPIII}. The \emph{support} of an $F$-weighting $\psi$ is the set of subgraphs $F$ of $G$ such that $\psi(F)>0$. 
\begin{lem}[General Boosting Lemma]\label{lem:GeneralBoostLemma}
For every integer $q \geq 3$, there exists $\Gamma > 0$ such that the following holds: Let $G$ be a graph, let $\cH$ be a family of copies of $K_q$ in $G$ and $\cQ$ be a family of copies of $K_{q+2}$ in $G$ such that for all $Q\in \cQ$, we have $\binom{Q}{q} \subseteq \cH$. If there exists a real $d > 0$ such that 
\[\frac{\Big| |\cH(e)|-d \Big|}{|\cQ(e)|} \cdot \left(\max_{R\in \binom{V(G)}{q}} |\cQ(R)|\right) \leq \Gamma\]
for all $e \in E(G)$, then there exists a fractional $K_q$-decomposition of $G$ whose support is in $\cH$ and whose weights are in $\left(1\pm \frac{1}{2}\right) \cdot \frac{1}{d}$.
\end{lem}

We are now prepared to prove~\cref{lem:NWRegBoostTreasury}.

\begin{lateproof}{lem:NWRegBoostTreasury}
    For $q\geq 3$ and $\varepsilon\in(0,1)$, let $\Gamma$ be as guaranteed by~\cref{lem:GeneralBoostLemma}, and let $\sigma>0$ and $C\geq 1$ such that~\cref{lem:BalancedDecomposition} holds for $F=K_q$. Let $\alpha\in(0,1)$ be a small enough constant. Let $n$ be a large enough integer, and let $J$ and $\cF_{orb}$ be as described in the statement of~\cref{lem:NWRegBoostTreasury}. 

    Fix an edge $e\in J$. Observe that, for every $n$ large enough, $e$ is in at least 
    \[\Big(1-\frac{2}{q+1}+\varepsilon\Big)n\cdot \Big(1-\frac{3}{q+1}+\varepsilon\Big)\frac{n}{4}\geq \Big(1-\frac{5}{q+1}\Big)\frac{n^2}{4}\]
    copies of $K_4$ in $J$. Fix such a copy $K$, and let $H$ be a copy of $K_{q+2}$ containing $K$. If $H$ intersects $\cF_{orb}$, that is, if it contains a copy of $K_q$ in $\cF_{orb}$, then that copy of $K_q$ must be in $\cF_{orb}(f)$ for an edge $f\in K$. Recall that $|\cF_{orb}(f)|\leq \alpha\binom{n-2}{q-2}$ for every edge $f\in G$. Therefore the total number of copies of $K_{q+2}$ containing $e$, and not intersecting $\cF_{orb}$, is at least
    \[\Big(1-\frac{5}{q+1}\Big)\frac{n^2}{4}\cdot
    \parens*{
        \brackets*{\prod_{t=4}^{q+1}\Big(1-\frac{t}{q+1}+\varepsilon\Big)\frac{n}{t+1}}
        -6\alpha\binom{n-2}{q-2}}.\]
    By taking $\alpha$ small enough compared to $q,\varepsilon$, we obtain that there exists a constant $\xi=\xi(q,\varepsilon)$ such that there are at least $\xi\cdot\binom{n}{q}$ copies of $K_{q+2}$ containing $e$ and not intersecting $\cF_{orb}$.

    Let $\phi$ be the $(\sigma,C)$-balanced fractional $K_q$-decomposition of $J$ guaranteed by~\cref{lem:BalancedDecomposition}. For every $H\in\binom{J}{K_q}$, set $\phi(H)=0$ if $H\in \cF_{orb}$. For every edge $e\in G$ we have
    \[\partial\phi(e)\geq 1-\frac{C}{\binom{n-2}{q-2}}\cdot \alpha \binom{n-2}{q-2}= 1-\alpha\cdot C.\]
    
    Let $d:=\frac{1}{C}\binom{n-2}{q-2}$. We define $\cH$ to be a random subcollection of copies of $K_q$ in $J$, by including every $H\in\binom{J}{K_q}$ with probability $\phi(H)\cdot d$, all independently. Observe that $\phi$ is $(\sigma,C)$-balanced, therefore $\phi(H)\cdot d$ is within $[0,1]$ for every $H$. Then, for every $e\in J$, $\Expect{|\cH(e)|}=\partial\phi(e)\cdot d$ with $\partial\phi(e)\in[1-\alpha\cdot C,1]$.  By a standard application of 
    the Chernoff bound, we see that
    \[\Prob{\Big||\cH(e)|-\partial\phi(e)\cdot d\Big|\geq \alpha\cdot C\cdot \partial\phi(e)\cdot d} \leq 2e^{-(\alpha\cdot C)^2 \partial\phi(e) d/3}, \]
    and we obtain from the Union Bound that with probability at least $2/3$, for every edge $e\in J$, we have
    \begin{equation}\label{eq:BoostingH}
        \Big||\cH(e)|-  d \Big|\leq 2\alpha\cdot C\cdot d.
    \end{equation}

    Let $\cQ$ be the family of all copies of $K_{q+2}$ in $J$ such that for every $Q\in\cQ$ we have $\binom{Q}{K_q}\subseteq\cH$. Recall that for every edge $e\in J$, there are at least $\xi\cdot\binom{n}{q}$ copies of $K_{q+2}$ containing $e$ and not intersecting $\cF_{orb}$, and observe that every copy of $K_q$ not in $\cF_{orb}$ is selected at random with probability $\phi(H)\cdot d \geq \sigma /C$. Therefore we obtain 
    \[\Expect{\cQ(e)}\geq \Big(\frac{\sigma}{C}\Big)^{\binom{q+2}{2}}\cdot\xi\cdot\binom{n}{q},\text{ and }\qquad\Expect{\cQ(e)} = \Theta(n^{q}).\]

    Note that $|\cQ(e)|=\sum_{A}I_A$ where the sum is taken over all copies $A$ of $K_{q+2}$ containing $e$ in $J$, and $I_A := \prod_{R}\mathds{1}_{\{R\in\mc{H}(e)\}}$, where the product is taken over all copies 
    $R$ of $K_{q}$. By Janson's inequality~\cite{janson1990poisson}, 
    \[\Prob{|\cQ(e)|\leq \frac12 \Expect{\cQ(e)}} \leq \exp\brackets*{\frac{-\Expect{\cQ(e)}^2}{2(\Expect{\cQ(e)}+\Delta_e)}}, \]
    where 
    \[\Delta_e := \sum_{\{A,B\}\in S_e}\Expect{I_AI_B} \qquad \textrm{ and }
    \qquad
    S_e:=\Big\{\{A,B\}\colon  A\neq B\in \binom{J}{K_{q+2}},\ e\in A\cap B,\ \binom{A}{K_q}\cap \binom{B}{K_q}\neq\emptyset\Big\}.
    \]
    Any two copies $A,B$ of $K_{q+2}$ containing $e$ intersect in at least one $K_q$ if and only if they share $i\in\{q-2,q-1\}$ other vertices. Therefore, using $\Expect{I_AI_B}\leq 1$, we obtain
    \[\Delta_e\leq |S_e| \leq \sum_{i\in\{q-2,q-1\}}\frac12\binom{n-2}{i}\binom{n-2-i}{q-i}\binom{n-2-q}{q-i}\leq n^{q+2},\]
    and therefore,
    \[\Prob{|\cQ(e)|\leq \frac12 \Expect{\cQ(e)}} \leq \exp\brackets*{-\Theta(n^{q-2})}.\]
    By the Union Bound, as $q\geq 3$ is fixed while $n$ is large enough, with probability at least $2/3$, for every edge $e\in J$ we have
    \begin{equation}\label{eq:BoostingQ}
        |\cQ(e)|\geq \frac12 \Expect{\cQ(e)}\geq\frac{\xi}{2}\cdot\Big(\frac{\sigma}{C}\Big)^{\binom{q+2}{2}}\cdot\binom{n}{q}.
    \end{equation}

    Fix a realization of $\cH$ such that~\cref{eq:BoostingH,eq:BoostingQ} hold. Recall that $d:=\frac{1}{C}\binom{n-2}{q-2}$.  We trivially have \(\max_{R\in \binom{V(G)}{q}} |\cQ(R)|\leq \binom{n}{2}\); then, 
    \[\frac{\Big| |\cH(e)|- d \Big|}{|\cQ(e)|} \cdot \max_{R\in \binom{V(G)}{q}} |\cQ(R)|
    \leq 
    \frac{4\cdot\alpha}{\xi}
    \cdot \Big(\frac{C}{\sigma}\Big)^{\binom{q+2}{2}}
    \cdot \binom{n-2}{q-2} \cdot\binom{n}{2}/\binom{n}{q}
    = \alpha\cdot \frac{4}{\xi}\cdot \Big(\frac{C}{\sigma}\Big)^{\binom{q+2}{2}} \cdot \binom{q}{2}
    \leq\Gamma,\]
    where in the last inequality we used the fact that $\Gamma$, $\sigma$, $C$ and $\xi$ depend only on $q$ and $\varepsilon$ while $\alpha$ is arbitrarily small compared to $q$ and $\varepsilon$. Therefore~\cref{lem:GeneralBoostLemma} yields a fractional $K_q$-decomposition $\psi$ of $J$, whose support is in $\cH$ and whose weights are in $\left(1\pm \frac{1}{2}\right) \cdot \frac{1}{d}$.

    We define $\mc{J}\subseteq \cH$ to be a random subcollection of copies of $K_q$, by including every $H\in \cH$ with probability $p_H = \frac23 \cdot \psi(H)\cdot d$, all independently. Observe that all weights of $\psi$ are in $\left(1\pm \frac{1}{2}\right) \cdot \frac{1}{d}$, therefore $p_H\in[0,1]$ for every $H\in\cH$. Then, for every $e\in G$, $\Expect{|\mc{J}(e)|}=\frac23 d$, and by the Chernoff bound,

    \[\Prob{\Big||\mc{J}(e)|-\frac23 d\Big|\geq n^{-1/3}\binom{n-2}{q-2}}\leq 2\cdot \exp\parens*{-\frac{2C}{9}n^{-2/3}\binom{n-2}{q-2}}.\]
    By the Union Bound over the at most $n^2$ edges of $G$, as $n$ is large enough, there exists a family $\mc{J}$ of copies of $K_q$ in $J$ with $\mc{J}\subseteq\cH$, hence $\mc{J}\cap\cF_{orb}=\emptyset$, and such that every $e\in J$ is in $\parens*{\frac{2}{3C}\pm n^{-1/3}}\cdot \binom{n-2}{q-2}$ copies of $K_q$ in $\mc{J}$, as desired.
\end{lateproof}

\section{High Girth Absorption}\label{s:HighGirth}

We now build \emph{high girth} omni-absorbers. To accomplish this, we use~\cref{thm:RefinedEfficientOmniAbsorber} to find an omni-absorber, and then embed additional girth boosters to arbitrarily increase the girth of every triangle-packing in our construction. We prove the following theorem, with required formal definitions in the next subsection.

\begin{thm}[Erd\H{o}s-Nash-Williams Omni-Absorber Theorem]\label{thm:HighGirthAbsorber}
For every integer $g\ge 3$ and reals $\varepsilon\in(0,1)$ and $\alpha\in(0,1/6)$, there exist integers $a,n_0\ge 1$ such that the following holds for all $n\ge n_0$. Let $G\subseteq K_n$ be a graph such that $\delta(G)\geq(\frac34+\varepsilon)n$, $X$ be a spanning subgraph of $G$ with $\Delta(X) \leq \frac{n}{\log^{ag}n}$ and such that ${\rm Treasury}^g(K_n,K_n,K_3,X)$ is $(n,\alpha n,\frac{1}{2g},\alpha)$-regular. Let $\Delta:= \max\left\{ \Delta(X),~\sqrt{n}\cdot \log n\right\}$. Then there exists a $K_3$-omni-absorber $A$ for $X$ in $G$ with $\Delta(A)\le a\Delta\cdot \log^{a} \Delta$ and collective girth at least $g$, together with a subtreasury $T=(G_1,G_2,H)$ of ${\rm Proj}_g(A, K_n, X)$ that is $(n,6\alpha n,\frac{1}{4g},4\alpha)$-regular.
\end{thm}

An attentive reader might notice that, while the absorber $A$ is guaranteed to be in $G$, the input and output treasuries of~\cref{thm:HighGirthAbsorber} are in $K_n$ and not necessarily $G$.  This provides stronger results, by forbidding all low girth configurations, with no requirement for them to be in $G$. This statement is simpler to deduce from the work of Delcourt and Postle~\cite{DPII} rather than switching to the graph $G$, but also necessary to distinguish the cliques that are forbidden due to girth restrictions from the ones forbidden because they are not in~$G$. 

We also require the following definitions and concepts related to hypergraphs. For an integer $r\ge 1$, a hypergraph $H$ is said to be \emph{$r$-bounded} if every edge of $H$ has size at most $r$ and  \emph{$r$-uniform} if every edge has size exactly $r$; for brevity, we  refer to an \emph{$r$-uniform hypergraph} as an \emph{$r$-graph}. For a hypergraph $G$ and subset $S\subseteq V(G)$, we let $G(S)$ denote the set $\{e\in G: S\subseteq e\}$. Note this differs from the standard notation $G(S)$ used to denote the related concept of a `link hypergraph' where the set $S$ is removed from every edge of $G(S)$. For an integer $r\ge 1$, we let $G^{(r)}$ denote the $r$-uniform subhypergraph of $G$ consisting of all edges of $G$ of size $r$. If $G$ is uniform, then for a vertex $v \in V(G)$, we let $d_G(v)$ denote the number of edges of $G$ containing $v$. If $G$ is not uniform, then for an integer $i\ge 1$, we thus write $d_{G^{(i)}}(v)$ for the \emph{$i$-degree of $v$ in $G$}, which is defined to be the number of edges of $G$ of size $i$ containing $v$. 
If $G$ is uniform, then \emph{maximum $i$-codegree of $G$}, denoted $\Delta_i(G)$ is the maximum of $|G(U)|$ for $U\subseteq V(G)$ with $|U|=i$. If $G$ is $r$-uniform, then the \emph{maximum degree of $G$}, denoted $\Delta(G)$ is $\Delta_{r-1}(G)$; similarly we write $\delta(G)$ for the \emph{minimum degree of $G$} which is $\delta_{r-1}(G)$. If $G$ is not uniform, then the \emph{maximum $(s,t)$-codegree of $G$}, denoted $\Delta_{t}\left(G^{(s)}\right)$ is the maximum of $G^{(s)}(U)$ for $U\in \binom{V(G)}{t}$. 

\subsection{Treasuries}

We rephrase our decomposition problems into matchings in auxiliary hypergraphs, using the Forbidden Submatchings with Reserves methodology developed by Delcourt and Postle~\cite{DP22}. To that extend, we use the following concept of {\em treasury}, that is, a triple of hypergraphs, encoding the cliques that can be used for a decomposition, and the configurations we may need to avoid (e.g. low girth configurations in the present article).

Recall that a hypergraph $G=(A,B)$ is bipartite with parts $A$ and $B$ if $V(G)=A\cup B$ and every edge of $G$ contains exactly one vertex from $A$, and that a matching of $G$ is $A$-perfect if every vertex of $A$ is in an edge of the matching. We say a hypergraph $H$ is a \emph{configuration hypergraph} for $G$ if $E(G)\subseteq V(H)$ and every edge $e$ of $H$ has size at least two and $e\cap E(G)$ is a matching of $G$. We say a matching of $G$ is \emph{$H$-avoiding} if it spans no edge of $H$.
 
\begin{definition}[Treasury]\label{def:Treasury}
Let $r,g \ge 2$ be integers. An \emph{$(r,g)$-treasury} is a tuple $T=(G_1,G_2,H)$ where $G_1$ is an $r$-uniform hypergraph and $G_2=(A,B)$ is an $r$-bounded bipartite hypergraph such that $V(G_1)\cap V(G_2)=A$, and $H$ is a $g$-bounded configuration hypergraph of $G_1\cup G_2$.
A \emph{perfect matching} of $T$ is an $H$-avoiding $(V(G_1)\cap V(G_2))$-perfect matching of $G_1\cup G_2$.
\end{definition}

For simplicity, we say that a graph $F$ is in a treasury $T=(G_1,G_2,H)$ if $F$ is an edge of $G_1$.
The main task in this article is to prove the existence of matchings in treasuries, via Forbidden Submatchings with Reserves methodologies~\cite{DP22}. To that end, the treasuries must exhibit regularity properties, whose exact definitions depend on the following concepts. Let $G$ be a hypergraph and let $H$ be a configuration hypergraph of $G$. We define the \emph{$i$-codegree} of $e\in V(G)$ and $F\in E(G)\subseteq V(H)$ with $e\notin F$ as the number of edges of $H$ of size $i$ who contain $F$ and a vertex $F_e$ incident with $e$ in $G$. The \emph{maximum $i$-codegree} of $G$ with $H$ is then the maximum $i$-codegree over all $e\in V(G)$ and $F\in E(G)\subseteq V(H)$ with $e\notin F$. 

We further define the \emph{common $2$-degree} of distinct vertices $F_1, F_2\in V(H)$ as $|\{F\in V(H): FF_1, FF_2\in E(H)\}|$. The \emph{maximum common $2$-degree} of $H$ with respect to $G$ is then the maximum of the common $2$-degree of $F_1$ and $F_2$ over all distinct pairs of vertex-disjoint edges $F_1,F_2$ of $G$.

\begin{definition}[Regular Treasury]\label{def:RegularTreasury}
Let $T=(G_1,G_2,H)$ be a treasury with $G_2=(A,B)$, and $D, \sigma\ge 1$ and $\beta,\alpha \in (0,1)$ be reals. We say that the treasury $T$ is \emph{$(D,\sigma,\beta,\alpha)$-regular} if all of the following hold:
\begin{enumerate}[label=(RT\arabic*),nolistsep,noitemsep, topsep=0pt]
    \item \textbf{Quasi-regularity.} Every vertex of $G_1$ has degree at most $D$ in $G_1$ and every vertex of $A$ has degree at least $D-\sigma$ in $G_1$,\label{item:RT_QuasiReg}
    
    \item \textbf{Reserve degrees.} Every vertex of $B$ has degree at most $D$ in $G_2$, and every vertex of $A$ has degree at least $D^{1-\alpha}$ in $G_2$,\label{item:RT_ReserveDeg}

    \item \textbf{Codegrees.} $G_1\cup G_2$ has codegrees at most $D^{1-\beta}$, and the maximum $2$-codegree of $G_1\cup G_2$ with $H$ and the maximum common $2$-degree of $H$ with respect to $G_1\cup G_2$ are both at most $D^{1-\beta}$,\label{item:RT_CoDeg}

    \item \textbf{Configuration (co)degrees.} For all $2\le s\le g$ we have $\Delta_1\left(H^{(s)}\right) \le \alpha \cdot D^{s-1}\log D$. For all $2\le t<s\le g$, $\Delta_{t}\left(H^{(s)}\right) \le D^{s-t-\beta}$.\label{item:RT_ConfigDeg}
 
\end{enumerate}
\end{definition}

For applications, it might be helpful to the reader to note that if $T$ is $(D,\sigma,\beta,\alpha)$-regular, then $T$ is also $(D,\sigma',\beta',\alpha')$-regular for all $\sigma'\ge \sigma$, $\beta' \le \beta$ and $\alpha'\ge \alpha$. We often cannot guarantee that the main treasuries we are studying are regular, but we prove that they contain regular ``subtreasuries'', using the following natural definition. We say that a treasury $(G_1',G_2',H')$ is a {\em subtreasury} of $(G_1,G_2,H)$ if $G_i'$ is a spanning subgraph of $G_i$ for each $i\in \{1,2\}$ and $H[E(G_1')\cup E(G_2')]$ is a subgraph of $H'$. The following result from~\cite{DPII} is then a natural consequence of this definition.

\begin{proposition}\label{prop:SubTreasury}
Let $T'$ be a subtreasury of a treasury $T$. If $M$ is a perfect matching of $T'$, then $M$ is a perfect matching of $T$.
\end{proposition}

\subsection{Forbidden Submatchings with Reserves}

As a reminder, we use a Nibble-like approach to get a triangle-packing of $G\setminus A$ that covers all edges in $G\setminus (A\cup X)$. The remaining edges will be decomposed using properties of the absorber $A$. The use of the Forbidden Submatchings with Reserves methodology, instead of a simpler Nibble one, ensures that the resulting decomposition spans no edges from the configuration hypergraph (in particular it avoids low girth decompositions). We use the following theorem as established by Delcourt and Postle~\cite{DP22}.

\begin{thm}[Forbidden Submatchings with Reserves~\cite{DP22}]\label{thm:ForbiddenSubmatchingReserves}
For all integers $r,g \ge 2$ and real $\beta \in (0,1)$, there exist an integer $D_0\ge 0$ and real $\alpha_0 > 0$ such that following holds for all $D\ge D_{0}$ and all $\alpha\leq\alpha_0$. If $T$ is an $(r,g)$-treasury that is $(D,D^{1-\beta},\beta,\alpha)$-regular, then there exists a perfect matching of $T$. 
\end{thm}

\subsection{High Girth Omni-Absorber}

Our main treasury of interest, whose matchings encode high girth packings of a given graph, is based on the following auxiliary hypergraphs. 

\begin{definition}[Design Hypergraph]\label{def:DesignHypergraph}
For graphs $F$ and $G$, the \emph{$F$-design hypergraph of $G$}, denoted ${\rm Design}(G,F)$ is the hypergraph $\Gamma$ with $V(\Gamma)=E(G)$ and $E(\Gamma) = \{F'\subseteq E(G): F' \text{ is isomorphic to } F\}$.
\end{definition}

\begin{definition}[Reserve Hypergraph]\label{def:ReserveHypergraph}
Let $F$ be a hypergraph. If $G$ is a hypergraph and $U,W$ are disjoint subsets of $E(G)$, then the \emph{$F$-design reserve hypergraph of $G$ from $U$ to $W$}, denoted ${\rm Reserve}(G,F,U,W)$,  is the bipartite hypergraph $\mathbf{R}=(U,W)$ with $V(\mathbf{R}):=U\cup W$ and $$E(\mathbf{R}) := \{S\subseteq U\cup W: S \text{ is isomorphic to } F,~|S\cap U|=1\}.$$
\end{definition}

Observe that, in particular, each edge of ${\rm Reserve}(G,F,U,W)$ is a copy of $F$ in $G$ with exactly one edge in $U$, and the other edges in $W$. The configuration hypergraph in our treasuries will then be used to avoid low girth packings. Recall that the \emph{girth} of a $K_q$-packing is the smallest integer $g \ge 2$ for which it contains a $(g(q - 2) + 2, g)$-configuration. For $i\geq 3$, an \emph{\Erdos{} $(i(q-2)+2,i)$-configuration} is a set of $i$ copies of $K_q$ on $i(q-2)+2$ vertices with girth exactly $i$, hence containing no $(j+2,j)$-configuration for $3\leq j<i$.

\begin{definition}[Girth Configuration Hypergraph]\label{def:GirthHypergraph}
Let $q,g\geq 3$ be integers. Let $G$ be a graph and let $\mathbf{D}:={\rm Design}(G,K_q)$. The \emph{girth-$g$ $K_q$-configuration hypergraph} of $G$, denoted ${\rm Girth}^g(G,K_q)$ is the configuration hypergraph $H$ of $\mathbf{D}$ with 
\begin{align*}
V(H)&:=E(\mathbf{D}),\\
E(H) &:= \{S\subseteq E(\mathbf{D}): S \text{ is an $(i(q-2)+2,i)$ \Erdos-configuration of } G \text{ for some $3\le i\leq g$}\}.   
\end{align*}
\end{definition}

In particular, ${\rm Girth}^g(G,K_q)$ contains all (minimal) $K_q$-packings of $G$ with girth at most $g$. Our main treasury of interest is then the following. 

\begin{definition}[Design treasury]\label{def:DesignTreasury}
Let $q,g\geq 3$ be integers. Let $G$ be a graph and let $X\subseteq G'\subseteq G$. The \emph{girth-$g$ $K_q$-design treasury of $G$ with reserve $X$}, denoted ${\rm Treasury}^g(G,G',K_q,X)$, is the treasury 
\[T:=\left({\rm Design}(G'\setminus X,K_q),~{\rm Reserve}(G,K_q,G'\setminus X, X),~{\rm Girth}^g(G,K_q)\right).\]
\end{definition}

As explained, we use a Forbidden Submatchings with Reserves methodology to find a matching in a design treasury, corresponding to a high girth packing of $G\setminus A$ whose non-covered edges are in the reserve $X$. We then use the properties of the $K_3$-omni-absorber $A$ to find a high girth packing of the ``left-overs''. However, this is not sufficient to ensure that we obtain a high girth decomposition of $G$. We need to ensure that these two packings together are still high girth, that is that we do not create a small cycle combining cliques in both packings. To that end, we further extend the forbidden configurations in our base treasury, to include any packing of $G\setminus A$ that extend to a low girth decomposition with the absorber $A$, using the following definitions.

\begin{definition}[Common projection]\label{def:CommonProjection}
Let $T=(G_1,G_2,H)$ be a treasury. Let $\mathcal{M}=\{M_1,\ldots,M_k\}$ be such that for all $i\in [k]$, we have $M_i\subseteq V(H)$ and $M_i\cap (G_1\cup G_2)$ is a matching of $G_1\cup G_2$. The \emph{common projection of $T$ by $\mathcal{M}$}, denoted $T\perp \mathcal{M}$ is the treasury  $T'=(G_1',G_2',H')$ where for all $j\in \{1,2\}$, $V(G_j'):=V(G_j)$, and
\[V(H') := V(H)\setminus \{ F\in V(H): \exists Z\in E(H),~M_i\in \mc{M} \text{ such that }~F\not\in M_i \text{ and } F\in Z \subseteq M_i\cup\{F\}\},\]

and for all $j\in \{1,2\}$, $E(G_j') = E(G_j)\cap V(H')$ and
\[E(H'):= \{ P \subseteq V(H'): |P|\geq 2,~\exists Z\in E(H),~M_i\in \mc{M} \text{ such that } P\cap M_i=\emptyset\text{ and }P\subseteq Z\subseteq P\cup M_i\}.\]
\end{definition}

\begin{definition}[Girth-$g$ Projection of Omni-Absorber]\label{def:OmniAbsorberProj}
Let $G$ be a graph, let $X$ be a subgraph of $G$, and let $A\subseteq G$ be a $K_q$-omni-absorber of $X$ with decomposition function $\mathcal{Q}_A$. We define the \emph{girth-$g$ projection treasury} of $A$ on to $G$ and $X$ as:
$${\rm Proj}_g(A,G,X):= {\rm Treasury}^g(G,~G\setminus A,~K_q,~X) \perp \mc{M}(A)$$
where $$\mathcal{M}(A) := \{ \mathcal{Q}_A(L) : L\subseteq X \text{ is $K_q$-divisible}\}.$$
\end{definition}

Observe that this notion of projection is exactly what what we mean in introduction when we state that we need to find an absorber that {\em does not shrink too many configurations}. Indeed, when projecting an absorber, edges of the configuration hypergraph are ``downsized'' to smaller configurations that could be completed back into their original (forbidden) form using a triangle-packing from the absorber. If the projection were to shrink too many configurations, it would then be impossible to avoid all of these small bad configurations when building a high girth decomposition, and our methodology would fail. The following proposition from~\cite{DPII}, naturally builds a high girth decomposition using a perfect matching in the girth projection of an omni-absorber.

\begin{proposition}\label{prop:FindSteiner}
Let $q,g\ge 3$ be integers. Let $G$ be a $K_q$-divisible graph, let $X$ be a subgraph of $G$ and let $A\subseteq G$ be a $K_q$-omni-absorber of $X$ of collective girth at least $g$. If there exists a perfect matching of ${\rm Proj}_g(A,G,X)$, then there exists a $K_q$-decomposition of $G$ of girth at least $g$.
\end{proposition}

At this stage, we usually state an {\em omni-absorber theorem}, explaining how we may find a high girth omni-absorber in the host graph $G$, with the desired regular projection treasury. The standard proof then calls for the concepts of {\em boosters}, {\em omni-boosters}, and {\em quantum boosters}, with respective omni-boosters and quantum-boosters theorem. However this strategy requires the introduction of many technical tools and propositions, following very closely to proof of the high girth conjecture as laid out by Delcourt and Postle~\cite{DPII}. For simplicity and clarity, we decided to first introduce boosters in the next subsection, explain how we can embed most of them in a graph $G$ with linear minimum degree, and finally show how the proof of~\cite{DPII} is marginally impacted by these changes; additional details are provided in~\cref{sec:appBoosters}.

\subsection{Girth Boosters}

A $K_q$-booster is a graph with the desirable property that it has two disjoint $K_q$-decompositions. To obtain omni-absorbers with large collective girth, we view the omni-absorber $A$ from~\cref{thm:NWRefinedEfficientOmniAbsorber} as a template, and we embed boosters for each clique in the decomposition family $\cF$ of $A$, increasing the collective girth of the absorbers. In practice, we use boosters to boost one fixed copy $R$ in $\mc{F}$ at a time. To that extend, we define a rooted booster as follows.

\begin{definition}[Rooted Booster]\label{def:rootedbooster}
A \emph{rooted $K_q$-booster} rooted at $R\cong K_q$ is a graph $B$ that is edge-disjoint from $R$ along with two disjoint $K_q$-packings $\mathcal{B}_{{\rm off}},\mathcal{B}_{{\rm on}}$ where $\mathcal{B}_{{\rm off}}$ is a $K_q$-decomposition of $B$ and $\mathcal{B}_{{\rm on}}$ is a $K_q$-decomposition of $B\cup R$ with $R\not\in \mathcal{B}_{{\rm on}}$.
\end{definition}

It follows from this definition that  $\cBon$ and $\cBoff\cup \{R\}$ are two $K_q$-decompositions of $B\cup R$. If $P$ is a $K_q$-packing of a graph $G$ that uses a copy $R$ of $K_q$, we can replace $R$ by $\cBon$ to obtain a $K_q$-packing of $B\cup P$; alternatively, if $R$ is not in the packing $P$, we simply use $\cBoff$ to decompose the edges of $B$. Our goal is then to build boosters whose decomposition families have high girth. To that end, we define the notion of the rooted girth of a booster as follows.

\begin{definition}[Rooted Girth]\label{def:RootedGirth}
Let $\cB$ be a $K_q$-packing of $G$. For a vertex set $R\subseteq V(G)$, we define the \emph{rooted girth} of $\cB$ at $R$ as the smallest integer $g\geq 1$ such that there exists a subset $\cB'\subseteq \cB$  with $|\cB'|=g$ and $|V(\bigcup \cB')\setminus R| < g$. Similarly, if $B$ is a rooted booster rooted at $R$, then we define the \emph{rooted girth} of $B$ as the minimum of the girth of $\cBon$, the girth of $\cBoff\cup \{R\}$, and the rooted girth of $\cBon$ at $V(R)$.
\end{definition}

In order to embed sufficiently many rooted-boosters in the host graph $G$, we use the standard notion of rooted degeneracy, see~\cref{def:RootedDegeneracy}. We study the specific case of some $K_3$-boosters called $g$-spheres, with rooted degeneracy $4$, introduced by Kwan, Sah, Sawhney, and Simkin~\cite{KSSS2024STS,KSSS2023substructures}. For $q\geq 3$, $K_q$-absorbers with small rooted degeneracy have been studied by Barber, K\"uhn, Lo, and Osthus \cite{BKLO16} when proving that $K_q$-decompositions of $K_q$-divisible graphs of large minimum degree exist; they (\cite[Corollary 8.10, Lemma 8.11, Lemma 12.3]{BKLO16}) proved that $K_q$-absorbers of rooted degeneracy at most $3q-3$ exist, and Delcourt, Kelly, and Postle~\cite{DKPIII} later proved the existence of $K_q$-absorbers of rooted degeneracy at most $2q-2$.

\begin{definition}[$g$-spheres]\label{def:Sphere}
    For $g\geq 2$, and a given copy $R$ of $K_3$, we define a $g$-sphere rooted at $R$ to be the following rooted $K_3$-booster $(B,\cBon,\cBoff,R)$. Denote $V(R)$ by $\{v,b_1,b_{2g}\}$. Add $2g-1$ vertices $u,b_2,b_3,\ldots,b_{2g-1}$, and let $B$ be the graph with $V(B)=\{u,v,b_1,\ldots,b_{2g}\}$ adding the following edges,
    \[\set{vb_j\colon j\in\{2,\ldots,2g-1\}},\quad \set{ub_j\colon j\in\{1,\ldots,2g\}},\quad\set{b_jb_{j+1}\colon j\in\{1,\ldots,2g-1\}},\quad\set{b_{2g}b_ 1}.\]    
    We let $\mathcal{B}_{{\rm off}}$, respectively $\mathcal{B}_{{\rm on}}$, be the $K_3$-decomposition of $B$, respectively $B\cup R$, defined by
    \begin{align*}
        \cBon  &= \{(vb_1b_2),(ub_2b_3),(vb_3b_4),\ldots,(vb_{2g-1}b_{2g}),(ub_{2g}b_{1})\},\\
        \cBoff &= \{(ub_1b_2),(vb_2b_3),(ub_3b_4),\ldots,(ub_{2g-1}b_{2g})\}.
    \end{align*}
\end{definition}

We refer the reader to~\cref{fig:TriangleBooster} for an illustration of a $3$-sphere. 
\begin{figure}
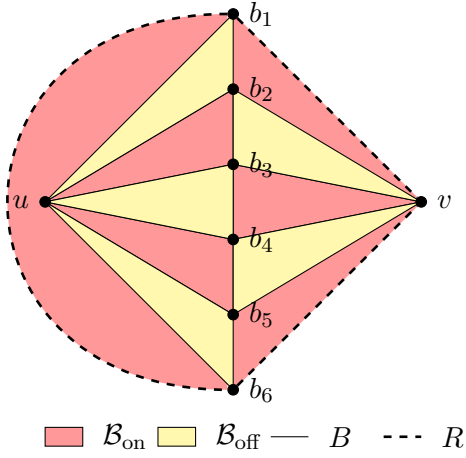

    \centering
    \TriangleBooster{.5}
    \caption{A $3$-sphere rooted at $R=(vb_1b_6)$}
    \label{fig:TriangleBooster}
\end{figure}  
We will use the following observations.

\begin{remark}\label{rem:countingSpheres}
    For every integer $g\geq 2$, there exists $C_g>0$ such that for any fixed copy $R$ of $K_3$, there are at most $C_g\cdot n^{2g-1}$ copies of a $g$-sphere rooted at $R$ in $K_n$. 
\end{remark}

\begin{remark}\label{rem:GirthSpheres}
    For every integer $g\geq 2$, it zfrom~\cite[Lemmas~4.7~and~4.8]{KSSS2024STS} that the rooted-girth of a $g$-sphere is $2g$, and from~\cite[Section~5]{BKLO16} that the rooted degeneracy of a $g$-sphere is $4$.
\end{remark}

The following proposition is then an immediate consequence of~\cref{lemma:DegeneracyEmbedding}.

\begin{proposition}\label{prop:DegeneracySphere}
    For every integer $g\geq 2$ and real $\varepsilon>0$, there exists $c>0$ such that the following holds for every integer $n$ large enough. For every graph $G\subseteq K_n$ with minimum degree at least $(3/4+\varepsilon)n$ and for every root $R\cong K_3$ in $K_n$, there are at least $c n^{2g-1}$ $g$-spheres rooted at~$R$ in~$G$.
\end{proposition}

In order to prove~\cref{thm:GenralisationMain2q}, the generalization of our main results to larger cliques, we require $K_q$-boosters with high rooted girth and low rooted degeneracy, which we extract from~\cite{DPII}. 

\begin{lem}\label{lem:DegeneracyBoosterGeneralQ}
For every integers $q\geq 3$ and $g\geq 3$ be integers, there exists a rooted $K_q$-booster with rooted girth at least $g$ and rooted degeneracy $2q-2$.     
\end{lem}

\begin{proof}
    We briefly recall the construction of girth-boosters from~\cite{DPII} (see Lemma~3.4 and the proof of Theorem 1.16). Let $H$ be a $(q-1)$-regular bipartite graph with parts $(\cB'_1,\cB'_2)$ and girth at least $g+1$ (e.g., see~\cite{wormald1999models} showing that the bipartite regular random graph satisfies this property with positive probability). The graph $H$ induces a $K_{q-1}^1$-booster, that is, a $1$-uniform hypergraph $B'$ with $V(B')=E(H)$, together with two disjoint partitions $(\cB'_1,\cB'_2)$ of $V(B')$ into sets of $q-1$ vertices. Let $B$ be the graph defined by
    \[V(B):= V(B')\cup \{v_1,v_2\},\text{ and }\qquad E(B):= \bigg\{ \{v_i,u\}: u\in V(B'),~i\in [2]\bigg\}~\cup \bigcup_{Q\in \mathcal{B}_1'\cup\mathcal{B}_2'} \binom{Q}{2}.\]

    We refer to $\{v_1,v_2\}$ as the \emph{exceptional vertices of $B$}. Observe that $B$ is a direct generalization of $g$-spheres, obtained by taking $H$ to be a $2$-regular bipartite graph with large girth, hence a large enough cycle. Since $H$ has girth at least $4$, we have that $\bigcup_{Q\in\mathcal{B}_1'} \binom{Q}{2}$ is disjoint from $\bigcup_{Q\in \mathcal{B}_2'} \binom{Q}{2}$. Therefore every vertex $u\in V(B')$ has exactly $2q-2$ neighbors in $B$, including $\{v_1,v_2\}$. Therefore, for any copy $S$ of $K_q$ whose vertices are in $B'$, we have that $B\setminus E(S)$ has rooted degeneracy at $S$ at most $2q-2$.
       
    Delcourt and Postle~\cite[Proof~of~Theorem~1.16]{DPII} proved that there exists a copy $S'$ of $K_q$ whose vertices are in $B'$ such that there exist at most $q-1$ copies $B_i$ of $B$ with $B_1=B$, $V(B_i)\cap V(B_j)=V(S')$ and $E(B_i)\cap E(B_j)=E(S')$ for every $i\neq j$, and such that $(\bigcup_{i}B_i)\setminus E(S')$ is a rooted $K_q$-booster with girth at most $g$. Observe that $V(B_i)\cap V(B_j)=V(S')$ ensures that every vertex in $(\bigcup_{i}V(B_i))\setminus V(S')$ has $2q-2$ neighbors in $(\bigcup_{i}B_i)\setminus S'$, except for the at most $2q-2$ exceptional vertices. It follows that $(\bigcup_{i}B_i)\setminus S'$ has rooted degeneracy at $S'$ at most $2q-2$.    
\end{proof}

\subsection{Absorption Theorem}

We are now ready to prove our absorption theorem,~\cref{thm:HighGirthAbsorber}, that combines with the Forbidden Submatchings with Reserves methodology,~\cref{thm:ForbiddenSubmatchingReserves}, and the Nash-Williams Boosting Lemma,~\cref{lem:NWRegBoostTreasury}, to prove our main result,~\cref{thm:Main_ErdosNashWilliams}.

The proof of~\cref{thm:HighGirthAbsorber} follows very closely the proof of~\cite[Theorem~1.11]{DPII}, with some minor modifications. In particular, the following Lemma is a straightforward application of~\cite[Lemmas~4.12~and~4.13]{DPII}. We include the details in~\cref{sec:appBoosters} for completeness.

\begin{lem}\label{cor:ExistenceOmniAbsorber}
    For all integers $C\ge 1$ and $g\ge 3$ and reals $c\in(0,1)$ and $\alpha\in (0,1/3)$, there exist integers $a, n_0\ge 1$ such that the following holds for all $n\ge n_0$. Fix a family $\cS$ of $g$-spheres in $K_n$ such that for every $R\cong K_3$ in $K_n$, there are at least $c\cdot n^{2g-1}$ distinct $g$-spheres rooted at $R$ in $\cS$. Let $X$ be a spanning subgraph of $K_n$ with $\Delta(X)\le \frac{n}{\log^{ag} n}$, and let $A_0$ be a $C$-refined $K_3$-omni-absorber for $X$ such that $\Delta(A_0)\le C\cdot \max\left\{ \Delta(X),~\sqrt{n}\cdot \log n\right\}$, and ${\rm Treasury}^g(K_n,~K_n\setminus A_0,~K_3,~X)$ is $(n,~\alpha n,~\frac{1}{2g},~\alpha)$-regular.

    Then, there exist a $K_3$-omni-absorber $A$ for $X$ with collective girth at least $g$ and a subtreasury $T=(G_1,G_2,H)$ of ${\rm Proj}_g(A,K_n,X)$ such that 
    \begin{enumerate}
            \item $\Delta(A)\leq a\Delta\log^a\Delta$, , where $\Delta:=\max\left\{ \Delta(X),~\sqrt{n}\cdot \log n\right\}$
            \item $T$ is $(n,~3\alpha n,~\frac{1}{4g},~2\alpha)$-regular, and 
            \item $A$ lives in $A_0\cup \cS$, that is, $A\subseteq A_0\cup\bigcup_{S\in\cS}S$.
    \end{enumerate}
\end{lem}

Armed with~\cref{cor:ExistenceOmniAbsorber}, we now prove our main high absorber theorem.
\begin{proof}[Proof of~\cref{thm:HighGirthAbsorber}]

By~\cref{prop:DegeneracySphere} there exists a constant $c=c(\varepsilon,g)>0$ and a family $\cS$ of $g$-spheres in $G$ such that for every $R\cong K_3$ in $K_n$, there are at least $c n^{2g-1}$ $g$-spheres rooted at~$R$ in~$G$.

By~\cref{thm:NWRefinedEfficientOmniAbsorber}, there exists an integer $C\geq 1$ and a $C$-refined $K_3$-omni-absorber $A_0$ for $X$ in $G$ with $\Delta(A_0) \leq C\cdot\Delta$. Then, with $n$ being a large enough integer, the fact that ${\rm Treasury}^g(K_n,K_n,K_3,X)$ is $(n,\alpha n,\frac{1}{2g},\alpha)$ immediately implies that ${\rm Treasury}^g(K_n,K_n\setminus A_0,K_3,X)$ is $(n,2\alpha n,\frac{1}{2g},\alpha)$-regular and therefore $(n,2\alpha n,\frac{1}{2g},2\alpha)$-regular.

By~\cref{cor:ExistenceOmniAbsorber} there exist a $K_3$-omni-absorber $A$ for $X$ in $G$ with collective girth at least $g$, and a subtreasury $T=(G_1,G_2,H)$ of ${\rm Proj}_g(A,K_n,X)$ with the desired properties.
\end{proof}

\section{Finishing - Proof of Theorem~\ref{thm:Main_ErdosNashWilliams}}\label{sec:Finishing}

We are now ready to prove our main result, stating that every $K_3$-divisible graph with sufficiently large minimum degree contains a high girth triangle-decomposition.

\begin{lateproof}{thm:Main_ErdosNashWilliams}
Fix an integer $g$ and a real $\varepsilon > 0$, let $\beta':=1/8g$ and $\beta:=2\beta'$. Let $c\in(0,1)$ such that~\cref{lem:NWRegBoostTreasury}, the Nash-Williams Restricted Boosting Lemma, holds for $q=3$. Let $\alpha'\in(0,1)$ be arbitrarily small, in particular such that~\cref{thm:ForbiddenSubmatchingReserves} holds for $\beta'$ and $\alpha'$. Let $\alpha\in(0,1)$ be small enough such that $4\alpha<c^{g}\cdot\alpha'$ and such that~\cref{lem:NWRegBoostTreasury} holds with parameters $q=3$, $6\alpha$ and $\varepsilon$. Let $n$ be large enough, and let $G$ be a $K_3$-divisible graph on $n$ vertices with minimum degree at least $(\max\{\delta^*_{K_{3}}, \frac{3}{4}\}+\varepsilon) n$.

Let $\gamma\in(0,1)$ be small enough such that $\gamma<\alpha/2$ and such that~\cref{lem:RandomXHighMinDeg} holds. 
Let $p=n^{-\gamma}$, and $X$ be a spanning subgraph of $G$ as guaranteed by~\cref{lem:RandomXHighMinDeg}, that is with $\Delta(X)\leq 2pn$ and such that every edge $e\in  K_n\setminus X$ is in at least $2^{-6}\cdot p^{2}n$ copies of $K_3$ in $X\cup \{e\}$.

We first show that $(\hat{G}_1,\hat{G}_2,\hat{H}):={\rm Treasury}^g(K_n,K_n,K_3,X)$ is $(n,\alpha n,2\beta,\alpha)$-regular. Indeed, every edge of $K_n$ is in $n-2$ copies of $K_3$, therefore $\Delta(\hat{G}_1),\Delta(\hat{G}_2)\leq n$. Recall that $\Delta(X)\leq 2pn=o(n)$, therefore for $n$ large enough every edge $e\in K_n\setminus X$ is in at least $(1-\alpha)n$ triangles and $\delta(\hat{G}_1)\geq n- \alpha n$, proving~\cref{item:RT_QuasiReg}.
For~\cref{item:RT_ReserveDeg}, note that every $e\in K_n\setminus X$ is in at least $2^{-6}\cdot n^{1-2\gamma}\geq n^{1-\alpha}$ copies of $K_3$ of $X\cup \{e\}$, hence $\delta(\hat{G}_2)\geq n^{1-\alpha}$. As $\hat{H}={\rm Girth}^g(K_n,K_3)$, we know that $\hat{H}$ contains no edges of size at most $2$, therefore it satisfies~\cref{item:RT_CoDeg}. Finally, $\hat{H}$ satisfies~\cref{item:RT_ConfigDeg} as for all $2\le i\le g$ and  all $2\le t < s \leq g$, 
\[\Delta_1\left(\hat{H}^{(i)}\right) \leq 2^{\binom{i-1}{2}} \cdot \binom{n}{i-1} \leq \alpha\cdot n^{i-1}\cdot \log n,\text{ and }\quad\Delta_{t}\left(\hat{H}^{(s)}\right) = O\left(n^{(s+2)-(t+3)}\right),\]
since every hyperedge $e$ of $\hat{H}$ of size $s$ spans at most $s+2$ vertices of $K_n$, while every $t$-subset of $e$ spans at least $t+3$ vertices. Therefore ${\rm Treasury}^g(K_n,K_n,K_3,X)$ is $(n,\alpha n,2\beta,~\alpha)$-regular.\medskip

By~\cref{thm:HighGirthAbsorber}, let $A$ be a $K_3$-omni-absorber for $X$ in $G$ with collective girth at least $g$, with  $T=(G_1,G_2,H)$ being the subtreasury of ${\rm Proj}_g(A, K_n, X)$ that is $(n,6\alpha n,\beta,4\alpha)$-regular. Let $J=G\setminus (X\cup A)$.

Let $T'=(G'_1,G'_2,H[G'_1\cup G'_2])$ be the subtreasury of $T$ where $G'_1$ (respectively $G'_2$) is obtained by removing from $G_1$ (respectively $G_2$) all triangles not in $G\setminus A$. By definition, $T'$ is a subtreasury of ${\rm Proj}_g(A, G, X)$, and we now prove that $T'$ is still regular. Recall that $G_2$ is a bipartite hypergraph with parts $K_n\setminus (X\cup A)$ and $X$. Therefore every hyperedge of $G_2$ represents a triangle of $K_n\setminus A$ containing a unique edge in $K_n\setminus (X\cup A)$. By regularity of $T$, every edge of $K_n\setminus (X\cup A)$ is in at least $n^{1-4\alpha}$ triangles of $G_2$, and therefore every edge of $J$ is in at least $n^{1-4\alpha}$ triangles of $G'_2$. 

By regularity of~$T$, every edge of $K_n\setminus (X\cup A)$ is in at least $(1-6\alpha)n$ triangles of $G_1$. As $G$ has minimum degree at least $(3/4+\varepsilon)n$, we obtain that every edge of $J$ is in at least $(1-6\alpha-1/2+\varepsilon)n$ triangles of $G'_1$. The other properties of regularity are trivially verified (as we only removed hyperedges from $G_1,G_2$) and we obtain that $T'$ is $(n,(1/2+6\alpha-\varepsilon) n,\beta,4\alpha)$-regular. Crucially, as $T$ is $(n,6\alpha n,\beta,4\alpha)$-regular, we know that every edge $e\in J$ is in at most $6\alpha n$ triangles of $J$ not in $G_1$. As we only removed triangles that are not in $J$  from $G_1$, every edge $e\in J$ is in at most $6\alpha n$ triangles in $J$ not in $G'_1$.

Recall that $\Delta= \max\left\{\Delta(X),~\sqrt{n}\cdot \log n\right\}$ and that we have $\Delta(X)\leq 2n^{1-\gamma}$ and $\Delta(A)\leq a \Delta \log^a\Delta$. Therefore, for $n$ large enough, $J$ has minimum degree at least $(\delta_{K_3}^*+\varepsilon/2)n$. Let $\cF_{orb}$ be family of all copies of $K_3$ in $J$ not in $G'_1$. As every edge $e\in J$ is in at most $6\alpha n$ triangles in $\cF_{orb}$, by~\cref{lem:NWRegBoostTreasury}, there exists a family $G''_1\subseteq G'_1$ of triangles in $J$ such that every edge of $J$ is in $(c\pm n^{-1/3})n$ triangles of $G''_1$. 

\begin{claim}
    The subtreasury $T''=(G''_1,G'_2,H[G''_1\cup G'_2])$ of $T$ is $(D,D^{1-\beta'},\beta',\alpha')$-regular with
\[D:=\left(c+n^{-1/3}\right)n.\]
\end{claim}

\begin{proofclaim}
    Let $\sigma := 2c \cdot n^{2/3} $. Observe that for $n$ large enough, with $\beta'=1/8g$ and $g\geq 1$, we have $D^{1-\beta'}\geq \sigma$. Therefore it suffices to prove that $T''$ is $(D,\sigma,\beta',\alpha')$-regular.

    By definition of $G''_1$, every edge of $J$ is in at most $D$ and at least $D-\sigma$ triangles of~$G''_1$, hence verifying~\cref{item:RT_QuasiReg}. For~\cref{item:RT_ReserveDeg}, we obtain from $\Delta(X)\leq 2n^{1-\gamma}$ with $n$ large enough that every vertex of $G'_2$ has degree at most $D'$ in $G_2$. Then, by regularity of $T'$, every vertex of $G'_2$ has degree at least $n^{1-4\alpha}\geq D^{1-4\alpha}\geq D^{1-\alpha'}$ in $G'_2$. 

    To prove the codegree properties~\cref{item:RT_CoDeg,item:RT_ConfigDeg}, let $H'=H[G''_1\cup G'_2]$. For every $2\leq i\leq g$, using $D\geq cn$, we obtain,
    \[\Delta({H'}^{(i)})\leq \Delta(H^{(i)}) 
    \leq 2\alpha n^{i-1}\log n 
    \leq 2\alpha c^{-g} \cdot D^{i-1} \log(D)
    \leq \alpha' D^{i-1}\log D.\]
    By regularity of $T'$ we have that $G'_1\cup G'_2$ has codegrees at most $n^{1-\beta}$, that 
    $\Delta_{t}(H^{(s)}) \le n^{s-t-\beta}$ for all $2\le t< s\le g$,
    and that the maximum $2$-codegree of $G'_1\cup G'_2$ with $H$ and the maximum common $2$-degree of $H$ are both at most $n^{1-\beta}$.
    Since $G''_1\subseteq G'_1$ and $H'\subseteq H$, the same holds for $G''_1$, $G'_2$, and $H'$. Moreover since $cn\leq D$ with $n$ and therefore $D$ large enough, for every integer $m\le g$, we have
    \[n^{m-\beta} \leq (D/c)^{m-\beta} \leq D^{m-\beta'}.\]
    Hence $n^{1-\beta}\le D^{1-\beta'}$ and for all $2\le t< s\le g$, $n^{s-t-\beta}\le D^{s-t-\beta'}$, as desired.
\end{proofclaim}

By~\cref{thm:ForbiddenSubmatchingReserves}, there exists a perfect matching $M$ of $T''$. By~\cref{prop:SubTreasury}, $M$ is also a perfect matching of $T'$ and therefore of ${\rm Proj}_g(A, G, X)$. Thus by Proposition~\ref{prop:FindSteiner}, there exists a $K_3$-decomposition of $G$ of girth at least $g$ as desired.
\end{lateproof}

Regarding~\cref{thm:GenralisationMain2q}, the generalization of our main result to larger cliques, observe that most lemmas and theorems used in the proof above are written in this general setting, namely the Reserve Lemma~\ref{lem:RandomXHighMinDeg}, the Boosting Lemma~\ref{lem:NWRegBoostTreasury} and the Forbidden Submatching with Reserves Theorem~\ref{thm:ForbiddenSubmatchingReserves}.

Therefore, assume that $G$ has minimum degree at least $\big(\max\big\{\delta^*_{q},1-\frac{1}{2q-2}\big\}+\varepsilon\big)n$. By~\cref{thm:NWRefinedEfficientOmniAbsorber}, there exists a $C$-refined $K_q$-omni-absorber $A_0$ for $X$ in $G$. By~\cref{lem:DegeneracyBoosterGeneralQ}, there exists a $K_q$-booster $B_0$ with rooted girth at least $g$ and rooted degeneracy at most $2q-2$. Therefore, by~\cref{lemma:DegeneracyEmbedding}, there exist a family $\cS$ of copies of $B_0$ in $G$ such that for every $F\cong K_q$ in $K_n$, there are at least $c\cdot n^{v(B_0)-q}$ distinct copies of $B_0$ rooted at $F$ in $\cS$. By~\cref{thm:HighGirthOmniAbsorberGeneralQ}, there exists a $K_q$-omni-absorber $A$ with collective girth at least $g$, and with the required properties on its projection treasury. We omit the details as the rest of the proof of~\cref{thm:GenralisationMain2q} is then identical to the one for~\cref{thm:Main_ErdosNashWilliams}.

\section{Concluding Remarks}\label{sec:Concluding}

To summarize, we tied the minimum-degree threshold for the existence of $K_3$-decomposition of arbitrarily high girth, to the threshold for fractional $K_3$-decomposition. A proof that $\delta_{K_3}^*=\frac34$ would now imply both the Nash-Williams' Conjecture and the ``\Erdos{} meets Nash-Williams' Conjecture'' asymptotically. Many more questions remain open. 

A key step toward proving the natural generalization of ``\Erdos{} meets Nash-Williams' Conjecture'' to larger cliques, as stated in~\cref{conj:ENWLargeCliquesNoEpsilon}, would be to extend the results in this article to larger cliques as in the following conjecture.

\begin{conj}\label{conj:GenralisationMain}
    For every integers $g\geq 3$ and $q\geq 3$, and every real $\varepsilon>0$, every sufficiently large $K_q$-divisible graph~$G$ on~$n$ vertices with minimum degree at least
    $(\max\{\delta_{K_q}^*, 1-\frac{1}{q+1}\}+\varepsilon)n$
    admits a $K_q$-decomposition with girth at least $g$.
\end{conj}

As mentioned in~\cref{sec:Generalq}, the results in this article can be generalized from triangles to larger cliques to prove~\cref{thm:GenralisationMain2q}, a version of~\cref{conj:GenralisationMain} with minimum degree at least $(\max\{\delta_{K_q}^*,1-\frac{1}{2q-2}\}+\varepsilon)n$. As mentioned in the proof overview, the obstacles to proving~\cref{conj:GenralisationMain} lie in embedding absorbers and in embedding girth boosters.

Similar questions can be naturally asked in higher uniformity. Keevash~\cite{K14} proved the Existence Conjecture for all integers $q>r\geq 2$, while Delcourt and Postle proved the High Girth Existence Conjecture~\cite{DPII}. Meanwhile, M. K\"uhn~\cite{kuhnPhD2025} proved the existence of high girth approximate $K_q^r$-decompositions when $\delta_{r-1}(G)\ge (\delta_{K_q^r}^*+\varepsilon)n$. Glock, K\"{u}hn and Osthus conjectured that the minimum degree threshold for $K_q^r$-decompositions and fractional ones should be identical and equal to $1-\Theta_r\big(q^{-r+1}\big)$, where we regard the uniformity $r\geq 2$ as fixed and consider asymptotics in the size of the clique $q$. Glock, K\"{u}hn, Lo and Osthus~\cite{GKLO16} proved an upper bound of $1-\Theta_q\big(r^{-2q}\big)$ for $K_q^r$-decompositions, while Barber, K\"{u}hn, Lo, Montgomery, and Osthus~\cite{BKLMO17} showed $\delta^*_{K_q^r}\leq 1-\Theta_q\big(r^{-2q+1}\big)$.
It would be interesting to have a further generalization of~\cref{conj:GenralisationMain} to higher uniformity. It is not clear however if a direct similar statement should hold in this general setting.

Finally, we observe that Kwan, Sah, Sawhney, and Simkin~\cite{KSSS2024STS} included a counting argument in their work on high girth Steiner triple systems. Any progress towards a counting extension for any of the above results and conjectures would be of interest. 


\bibliographystyle{plain}
\bibliography{bibliography}

\begin{thebibliography}{10}

\bibitem{AW23}
Jack Allsop and Ian~M. Wanless.
\newblock Latin squares without proper subsquares.
\newblock {\em arXiv preprint arXiv:2310.01923}, 2023.

\bibitem{alon1994probabilistic}
Noga Alon.
\newblock Probabilistic methods in coloring and decomposition problems.
\newblock {\em Discrete Mathematics}, 127(1-3):31--46, 1994.

\bibitem{AS16}
Noga Alon and Joel~H. Spencer.
\newblock {\em The Probabilistic Method}.
\newblock John Wiley \& Sons, 2016.

\bibitem{AY05}
Noga Alon and Raphael Yuster.
\newblock On a hypergraph matching problem.
\newblock {\em Graphs and Combinatorics}, 21(4):377--384, 2005.

\bibitem{BKLMO17}
Ben Barber, Daniela K{\"u}hn, Allan Lo, Richard Montgomery, and Deryk Osthus.
\newblock Fractional clique decompositions of dense graphs and hypergraphs.
\newblock {\em Journal of Combinatorial Theory, Series B}, 127:148--186, 2017.

\bibitem{BKLO16}
Ben Barber, Daniela K{\"u}hn, Allan Lo, and Deryk Osthus.
\newblock Edge-decompositions of graphs with high minimum degree.
\newblock {\em Advances in Mathematics}, 288:337--385, 2016.

\bibitem{BW19}
Tom Bohman and Lutz Warnke.
\newblock Large girth approximate {S}teiner triple systems.
\newblock {\em Journal of the London Mathematical Society}, 100(3):895--913, 2019.

\bibitem{brouwer1977steiner}
Andries~Evert Brouwer.
\newblock {S}teiner {Triple} {Systems} without forbidden subconfigurations.
\newblock {\em Mathematisch Centrum Amsterdam, ZW 104/77}, 1977.

\bibitem{BES73}
William Brown, Paul Erd{\H{o}}s, and Vera S{\'o}s.
\newblock Some extremal problems on $r$-graphs.
\newblock {\em New directions in the theory of graphs ({P}roceedings of the {T}hird {A}nn {A}rbor {C}onference, {U}niv. {M}ichigan, {A}nn {A}rbor, {M}ich., 1971)}, pages 53--63, 1973.

\bibitem{CH63}
Kereszt{\'e}ly Corradi and Andr{\'a}s Hajnal.
\newblock On the maximal number of independent circuits in a graph.
\newblock {\em Acta Mathematica Hungarica}, 14(3-4):423--439, 1963.

\bibitem{DKPIII}
Michelle Delcourt, Tom Kelly, and Luke Postle.
\newblock Clique {D}ecompositions in {R}andom {G}raphs via {R}efined {A}bsorption.
\newblock {\em arXiv:2402.17857}, 2024.

\bibitem{DKPIV}
Michelle Delcourt, Tom Kelly, and Luke Postle.
\newblock Thresholds for $(n,q,2)$-{S}teiner {S}ystems via {R}efined {A}bsorption.
\newblock {\em arXiv:2402.17858}, 2024.

\bibitem{DP2021progress}
Michelle Delcourt and Luke Postle.
\newblock Progress towards {N}ash-{W}illiams' {C}onjecture on triangle decompositions.
\newblock {\em Journal of Combinatorial Theory, Series B}, 146:382--416, 2021.

\bibitem{DP22}
Michelle Delcourt and Luke Postle.
\newblock Finding an almost perfect matching in a hypergraph avoiding forbidden submatchings.
\newblock {\em arXiv:2204.08981}, 2022.

\bibitem{DP22BES}
Michelle Delcourt and Luke Postle.
\newblock The limit in the $(k+2,k)$-{P}roblem of {B}rown, {E}rd{\H{o}}s and {S}{\'o}s exists for all $k\geq 2$.
\newblock {\em Proceedings of the American Mathematical Society}, 152(05):1881--1891, 2024.

\bibitem{DPII}
Michelle Delcourt and Luke Postle.
\newblock Proof of the {H}igh {G}irth {E}xistence {C}onjecture via {R}efined {A}bsorption.
\newblock {\em arXiv:2402.17856}, 2024.

\bibitem{DPI}
Michelle Delcourt and Luke Postle.
\newblock {R}efined {A}bsorption: {A} {N}ew {P}roof of the {E}xistence {C}onjecture.
\newblock {\em arXiv:2402.17855}, 2024.

\bibitem{DK76}
J{\'o}zsef D{\'e}nes and A.~D. Keedwell.
\newblock Latin squares and their applications.
\newblock {\em {A}cademic {P}ress, {N}ew {Y}ork-{L}ondon}, page 547, 1974.

\bibitem{dirac1952some}
Gabriel~Andrew Dirac.
\newblock Some theorems on abstract graphs.
\newblock {\em Proceedings of the London Mathematical Society}, 3(1):69--81, 1952.

\bibitem{dross2016fractional}
Fran{\c{c}}ois Dross.
\newblock Fractional triangle decompositions in graphs with large minimum degree.
\newblock {\em SIAM Journal on Discrete Mathematics}, 30(1):36--42, 2016.

\bibitem{dukes_minimum_2020}
Peter~J. Dukes and Daniel Horsley.
\newblock On the minimumm degree required for a triangle decomposition.
\newblock {\em SIAM Journal on Discrete Mathematics}, 34(1):597--610, January 2020.

\bibitem{ES66}
Paul Erd\H{o}s and Mikl{\'o}s Simonovits.
\newblock A limit theorem in graph theory.
\newblock {\em Studies Scientiarum Mattiematiearum Hungariea}, 1(51-57):51, 1966.

\bibitem{ES46}
Paul Erd\H{o}s and Arthur~H. Stone.
\newblock On the structure of linear graphs.
\newblock {\em Bulletin of the American Mathematical Society}, 52(12):1087--1091, 1946.

\bibitem{E73}
Paul Erd{\H{o}}s.
\newblock Problems and results in combinatorial analysis.
\newblock In {\em Colloq. Internat. Theor. Combin. Rome}, pages 3--17, 1973.

\bibitem{EH63}
Paul Erd{\H{o}}s and Haim Hanani.
\newblock On a limit theorem in combinatorial analysis.
\newblock {\em Publ. Math. Debrecen}, 10:10--13, 1963.

\bibitem{FB22}
Asaf Ferber and Matthew Kwan.
\newblock Dirac-type theorems in random hypergraphs.
\newblock {\em Journal of Combinatorial Theory, Series B}, 155:318--357, 2022.

\bibitem{FR13}
Zolt{\'a}n F{\"u}redi and Mikl{\'o}s Ruszink{\'o}.
\newblock Uniform hypergraphs containing no grids.
\newblock {\em Advances in Mathematics}, 240:302--324, 2013.

\bibitem{garaschuk2014linear}
Kseniya Garaschuk.
\newblock {\em Linear methods for rational triangle decompositions}.
\newblock PhD thesis, University of {V}ictoria, 2014.

\bibitem{GJKKL24}
Stefan Glock, Felix Joos, Jaehoon Kim, Marcus K{\"u}hn, and Lyuben Lichev.
\newblock Conflict-free hypergraph matchings.
\newblock {\em Journal of the London Mathematical Society}, 109(5):e12899, 2024.

\bibitem{GKLMO19}
Stefan Glock, Daniela K{\"u}hn, Allan Lo, Richard Montgomery, and Deryk Osthus.
\newblock On the decomposition threshold of a given graph.
\newblock {\em Journal of Combinatorial Theory, Series B}, 139:47--127, 2019.

\bibitem{GKLO20}
Stefan Glock, Daniela K{\"u}hn, Allan Lo, and Deryk Osthus.
\newblock On a conjecture of {E}rd{\H{o}}s on locally sparse {S}teiner triple systems.
\newblock {\em Combinatorica}, 40(3):363--403, 2020.

\bibitem{GKLO16}
Stefan Glock, Daniela K{\"u}hn, Allan Lo, and Deryk Osthus.
\newblock The existence of designs via iterative absorption: {H}ypergraph ${F}$-designs for arbitrary ${F}$.
\newblock {\em Memoirs of the American Mathematical Society}, 284(1406), 2023.

\bibitem{GKO20Survey}
Stefan Glock, Daniela Kühn, and Deryk Osthus.
\newblock {\em Extremal aspects of graph and hypergraph decomposition problems}, page 235–266.
\newblock London Mathematical Society Lecture Note Series. Cambridge University Press, 2021.

\bibitem{GGW00AntiPasch}
M.~J. Grannell, T.~S. Griggs, and C.~A. Whitehead.
\newblock The resolution of the anti-{P}asch conjecture.
\newblock {\em Journal of Combinatorial Designs}, 8(4):300--309, 2000.

\bibitem{HS70}
Andr{\'a}s Hajnal and Endre Szemer\'edi.
\newblock Proof of a conjecture of {P}. {E}rd{\H{o}}s.
\newblock {\em Combinatorial Theory and its Applications}, II:601--603, 1970.

\bibitem{H95}
Penny~E. Haxell.
\newblock A condition for matchability in hypergraphs.
\newblock {\em Graphs and Combinatorics}, 11(3):245--248, 1995.

\bibitem{HR01}
Penny~E. Haxell and Vojtech R{\"o}dl.
\newblock Integer and fractional packings in dense graphs.
\newblock {\em Combinatorica}, 21(1):13--38, 2001.

\bibitem{janson1990poisson}
Svante Janson.
\newblock Poisson approximation for large deviations.
\newblock {\em Random Structures \& Algorithms}, 1(2):221--229, 1990.

\bibitem{JMS24}
Felix Joos, Dhruv Mubayi, and Zak Smith.
\newblock Conflict-free hypergraph matchings and coverings.
\newblock {\em arXiv preprint arXiv:2407.18144}, 2024.

\bibitem{K14}
Peter Keevash.
\newblock The existence of designs.
\newblock {\em arXiv preprint arXiv:1401.3665}, 2014.

\bibitem{K18survey}
Peter Keevash.
\newblock Hypergraph matchings and designs.
\newblock {\em Proc. Int. Cong. Math.}, 3:3099--3122, 2018.

\bibitem{K24}
Peter Keevash.
\newblock A short proof of the existence of designs.
\newblock {\em arXiv preprint arXiv:2411.18291}, 2024.

\bibitem{KL20}
Peter Keevash and Jason Long.
\newblock The {B}rown-{E}rd{\H{o}}s-{S}{\'o}s conjecture for hypergraphs of large uniformity.
\newblock {\em arXiv:2007.14824}, 2020.

\bibitem{K47}
Thomas~P. Kirkman.
\newblock On a problem in combinatorics.
\newblock {\em Cambridge Dublin Math. J}, 2:191--204, 1847.

\bibitem{kuhnPhD2025}
Marcus K{\"u}hn.
\newblock {\em The Random Greedy Hypergraph Matching Process}.
\newblock PhD thesis, Heidelberg University, 2025.

\bibitem{KSSS2023substructures}
Matthew Kwan, Ashwin Sah, Mehtaab Sawhney, and Michael Simkin.
\newblock Substructures in {L}atin squares.
\newblock {\em Israel Journal of Mathematics}, 256(2):363--416, 2023.

\bibitem{KSSS2024STS}
Matthew Kwan, Ashwin Sah, Mehtaab Sawhney, and Michael Simkin.
\newblock High-girth {S}teiner triple systems.
\newblock {\em Annals of Mathematics}, 200(3):1059--1156, 2024.

\bibitem{lang2023tiling}
Richard Lang.
\newblock Tiling dense hypergraphs.
\newblock {\em arXiv preprint arXiv:2308.12281}, 2023.

\bibitem{LPR93}
Hanno Lefmann, Kevin~T. Phelps, and Vojt{\v{e}}ch R{\"o}dl.
\newblock Extremal problems for triple systems.
\newblock {\em Journal of Combinatorial Designs}, 1(5):379--394, 1993.

\bibitem{LCGG00construction}
A.~C.~H. Ling, C.~J. Colbourn, M.~J. Grannell, and T.~S. Griggs.
\newblock Construction techniques for anti-{P}asch {S}teiner triple systems.
\newblock {\em Journal of the London Mathematical Society}, 61(3):641--657, 2000.

\bibitem{linial2018challenges}
Nati Linial.
\newblock Challenges of high-dimensional combinatorics.
\newblock In {\em Lov{\'a}sz’s Seventieth Birthday Conference}, volume~2, 2018.

\bibitem{mantel}
Willem Mantel.
\newblock Problem 28 (solution by {Gouwentak, Mantel, Teixeira de Mattes, Schuh and Wythoff}).
\newblock {\em Wiskundige Opgaven}, 10:60--61, 1907.

\bibitem{montgomery_fractional_2019}
Richard Montgomery.
\newblock Fractional clique decompositions of dense graphs.
\newblock {\em Random Structures \& Algorithms}, 54(4):779--796, 2019.

\bibitem{M19b}
Richard Montgomery.
\newblock Spanning trees in random graphs.
\newblock {\em Advances in Mathematics}, 356:106793, 2019.

\bibitem{nash1970unsolved}
Crispin~St~J.~A. Nash-Williams.
\newblock An unsolved problem concerning decomposition of graphs into triangles.
\newblock {\em Combinatorial Theory and its Applications}, 3(1070):1179--1183, 1970.

\bibitem{R85}
Vojt{\v{e}}ch R{\"o}dl.
\newblock On a packing and covering problem.
\newblock {\em European Journal of Combinatorics}, 6(1):69--78, 1985.

\bibitem{RRS06}
Vojt{\v{e}}ch R{\"o}dl, Andrzej Ruci{\'n}ski, and Endre Szemer{\'e}di.
\newblock A {D}irac-type theorem for 3-uniform hypergraphs.
\newblock {\em Combinatorics, Probability and Computing}, 15(1-2):229--251, 2006.

\bibitem{RSST07}
Vojtech R{\"o}dl, Mathias Schacht, Mark~H Siggers, and Norihide Tokushige.
\newblock Integer and fractional packings of hypergraphs.
\newblock {\em Journal of Combinatorial Theory, Series B}, 97(2):245--268, 2007.

\bibitem{turan1941extremal}
P{\'a}l Tur{\'a}n.
\newblock Eine {Extremalaufgabe} aus der {Graphentheorie}.
\newblock {\em Matematikai és. Fizikai Lapok}, 48:436--452, 1941.

\bibitem{W72-EC1}
Richard~M. Wilson.
\newblock {An existence theory for pairwise balanced designs I. Composition theorems and morphisms}.
\newblock {\em Journal of Combinatorial Theory, Series A}, 13(2):220--245, 1972.

\bibitem{W72-EC2}
Richard~M. Wilson.
\newblock {An existence theory for pairwise balanced designs II. The structure of PBD-closed sets and the existence conjectures}.
\newblock {\em Journal of Combinatorial Theory, Series A}, 13(2):246--273, 1972.

\bibitem{W75-EC3}
Richard~M. Wilson.
\newblock {An existence theory for pairwise balanced designs, III: Proof of the existence conjecture}.
\newblock {\em Journal of Combinatorial Theory, Series A}, 18(1):71--79, 1975.

\bibitem{W75}
Richard~M. Wilson.
\newblock Decomposition of complete graphs into subgraphs isomorphic to a given graph.
\newblock {\em Congressus Numerantium XV}, pages 647--659, 1975.

\bibitem{wormald1999models}
Nicholas~C Wormald et~al.
\newblock Models of random regular graphs.
\newblock {\em London mathematical society lecture note series}, pages 239--298, 1999.

\bibitem{Yuster2005asymptotically}
Raphael Yuster.
\newblock Asymptotically optimal ${K}_k$-packings of dense graphs via fractional ${K}_k$-decompositions.
\newblock {\em Journal of Combinatorial Theory, Series B}, 95(1):1--11, 2005.

\bibitem{Yuster05}
Raphael Yuster.
\newblock Integer and fractional packing of families of graphs.
\newblock {\em Random Structures \& Algorithms}, 26(1-2):110--118, 2005.

\bibitem{Yuster12}
Raphael Yuster.
\newblock ${H}$-packing of $k$-chromatic graphs.
\newblock {\em Mosc. J. Comb. Number Theory}, 2(1):73--88, 2012.

\end{thebibliography}

\appendix

\section{Omni-Boosters}\label{sec:appBoosters}

To prove~\cref{thm:HighGirthAbsorber}, we add a rooted $K_3$-booster for each triangle in the decomposition family of a refined omni-absorber; we do this simultaneously and randomly such that with high probability this is done in a high girth manner. To that end, we define the following omni-booster.

\begin{definition}[Omni-Booster]\label{def:OmniBooster}
Let $G$ be a graph, $X$ be a spanning subgraph of $G$, and let $A$ be a $K_q$-omni-absorber for $X$ with decomposition family $\cF$. We say that a graph $B\subseteq G\setminus(A\cup X)$ is a $K_q$-\emph{omni-booster} for $A$ and $X$ in $G$ if $B$ is the edge-disjoint union of a set of graphs $\cB:= (B_F: F\in\cF)$ (called the \emph{booster family} of $B$) where for each $F\in \cF$, $B_F$ is a $K_q$-booster rooted at $F$.
The \emph{matching set} of $\cB$ is
$$\mc{M}(\cB):= \set*{ M \in \prod_{F\in \mathcal{F}} \big\{ (\mathcal{B}_{F})_{{\rm on}},~ (\mathcal{B}_{F})_{{\rm off}} \big\} : M \text{ is a matching of } {\rm Design}(X\cup A\cup B,K_q)}.$$
We say $B$ has \emph{collective girth at least $g$} if every $M\in \mc{M}(\cB)$ has girth at least $g$.
\end{definition}

Naturally, omni-boosters yield the following canonical omni-absorbers, as computed in~\cite{DPII}.

\begin{proposition}\label{prop:CanonicalBoost}
For $q\geq 3$, let $G$ be a graph, $X$ be a spanning subgraph of $G$, and let $A$ be a $K_q$-omni-absorber for $X$ with decomposition family $\cF$ and decomposition function $\mc{Q}_A$. If $B$ is a $K_q$-omni-booster for $A$ and $X$ in $G$ with booster family $\mathcal{B}=(B_F: F\in\mc{F})$, then $A\cup B$ is a $K_q$-omni-absorber for $X$ with decomposition family $\bigcup_{F\in \mc{F}} (\mc{B}_F)_{{\rm on}}\cup (\mc{B}_F)_{{\rm off}}$ and decomposition function
$$\mathcal{Q}_{A\cup B}(L):= \bigcup_{F\in \mathcal{Q}_A(L)} (\mathcal{B}_F)_{{\rm on}}~\cup \bigcup_{F \in \mathcal{F}\setminus \mathcal{Q}_A(L)} (\mathcal{B}_F)_{{\rm off}}.$$    
We refer to $A\cup B$ with decomposition function $\mathcal{Q}_{A\cup B}$ defined above as the \emph{canonical omni-absorber} for $A$ and $B$.
\end{proposition}

For simplicity, instead of showing that $\mathcal{Q}_{A\cup B}$ has high collective girth, we show that all $K_q$-packings that are the union of ``on and off'' decompositions are high girth. To that end, we make the following definition.

\begin{definition}[Girth-$g$ Projection of Omni-Booster]\label{def:OmniBoosterProj}
Let $G$ be a graph, $X$ be a spanning subgraph of $G$, and let $A$ be a $K_q$-omni-absorber for $X$. Let $B$ a $K_q$-omni-booster for $A$ with booster family $\mathcal{B}=(B_F:F\in\mathcal{F})$. We define the \emph{girth-$g$ projection treasury} of $B$ on to $G$ and $X$ as:
$${\rm Proj}_g(B,A,G,X):= {\rm Treasury}^g(G,~G\setminus (A\cup B),~K_q,~X) \perp \mathcal{M}(\mc{B}).$$
\end{definition}

The following proposition, relating the treasury of an omni-booster and its canonical omni-absorber, can be found in~\cite{DPII}.

\begin{proposition}\label{prop:BoosterSubTreasury} For $q\geq 3$, let $G$ be a graph, $X$ be a spanning subgraph of $G$, $A$ be a $K_q$-omni-absorber for $X$ in $G$, and $B$ be a $K_q$-omni-booster for $A$. Then ${\rm Proj}_g(B,A,G,X)$ is a subtreasury of ${\rm Proj}_g(A\cup B,G,X)$.
\end{proposition}

\subsection{Quantum Omni-Boosters}

In order to find omni-boosters with the desired properties, we start with a larger object, containing many (non-edge-disjoint) rooted boosters. We show that, by taking a random sparsification, we obtain with positive probability the desired omni-booster. To that end we define the following quantum omni-boosters (as superposition of multiple omni-boosters), or quantum-boosters for short.

\begin{definition}[Quantum Omni-Booster]\label{def:QuantumOmniBooster}
For $q\geq 3$, let $G$ be a graph, $X$ be a spanning subgraph of $G$, and let $A$ be a $K_q$-omni-absorber for $X$ with decomposition family $\mc{F}$. A \emph{quantum $K_q$-omni-booster} for $A$ and $X$ in $G$ is a graph $B$ together with a family $\mc{B}=(\mathcal{B}_F: F\in \mc{F})$ (called the \emph{quantum booster collection}) where for each $F\in \mc{F}$, $\mc{B}_F$ is a family of (not necessarily-edge-disjoint) $K_q$-boosters rooted at $F$ in $G\setminus (X\cup A)$. 

The \emph{matching set} of $\mc{B}$ is 
$$\mc{M}(\mc{B}):= \bigg\{ M \in \prod_{F\in \mathcal{F}} \bigcup_{B_F\in \mathcal{B}_F} \big\{ (B_{F})_{{\rm on}},~ (B_{F})_{{\rm off}} \big\} : M \text{ is a matching of } {\rm Design}(G,K_q)\bigg\}.$$

For each $F\in \mathcal{F}$, we define 
$${\rm Disjoint}(\mathcal{B}_{F}) := \left\{B_F \in \mathcal{B}_{F}:~ B_F \text{ is (edge-)disjoint from all of }\bigcup_{F' \in \mathcal{F}\setminus \{F\}} \bigcup \mathcal{B}_{F'}\right\}.$$
For an integer $g\ge 3$, we define
\begin{align*}{\rm HighGirth}_g(\mathcal{B}_{F}) := \bigg\{B_F \in \mathcal{B}_{F}: \nexists &R\in {\rm Girth}^g(G,K_q) \text{ where } R\subseteq M \text{ for some } M \in \mc{M}(\mc{B}) \\
&\text{ and } R\cap \big( (B_F)_{\rm on}\cup (B_F)_{\rm off}\big)\ne \emptyset. \bigg\}.
\end{align*}
\end{definition}

Observe that if ${\rm Disjoint}(\mathcal{B}_{F})$ is non-empty for every $F\in\cF$, we can create a omni-booster by taking one element in each $\cB_F$. Furthermore, if these elements are also in ${\rm HighGirth}(\mathcal{B}_{F})$, then we obtain an omni-booster with collective girth at least $g$. We naturally extend the notion of projection to quantum boosters.

\begin{definition}[Girth $g$ Projection of Quantum Omni-Booster]\label{def:QuantumOmniBoosterProj}
For $q\geq 3$, let $G$ be a graph, $X$ be a spanning subgraph of $G$, and let $A$ be a $K_q$-omni-absorber for $X$ with decomposition family $\mc{F}$. Let $B$ a quantum $K_q$-omni-booster for $A$ with quantum booster collection $\mathcal{B} = (\mathcal{B}_F: F\in\mathcal{F})$. We define the \emph{girth-$g$ projection treasury} of $B$ on to $G$ and $X$ as: 
$${\rm Proj}_g(B,A,G,X):= {\rm Treasury}^g(G,~G\setminus A,~K_q,~X) \perp \mathcal{M}(\mc{B}).$$
\end{definition}

The following statement is then our main quantum booster theorem. We show that, given a refined omni-absorber $A_0$, and a family $\mc{S}$ of girth boosters containing, for every root $F$, a constant proportion of all girth boosters rooted at $F$ in $K_n$, then there exists a quantum omni-booster for $A_0$ from which we will be able to sample to yield a high girth omni-booster living inside $\mc{S}$.

\begin{thm}[\Erdos-Nash-Williams Quantum Omni-Booster Theorem]\label{thm:HighGirthQuantumOmniBooster}
For every integers $C\ge 1$ and $g\ge 3$ and reals $c\in(0,1)$ and $\alpha \in (0,1/2)$, there exist integers $a, n_0\ge 1$ such that the following holds for all $n\ge n_0$. Fix a family $\cS$ of rooted $g$-spheres in $K_n$ such that for every $F\cong K_3$ in $K_n$, there are at least $c\cdot n^{2g-1}$ distinct $g$-spheres rooted at $F$ in $\cS$. Let $X$ be a spanning subgraph of $K_n$ with $\Delta(X)\le \frac{n}{\log^{ag} n}$, and let $A_0$ be a $C$-refined $K_3$-omni-absorber for $X$ with decomposition family $\cF_0$ such that $\Delta(A_0)\le C\cdot \max\left\{ \Delta(X),~\sqrt{n}\cdot \log n\right\}$, and ${\rm Treasury}^g(K_n,K_n\setminus A_0,K_3,X)$ is $(n,\alpha n,\frac{1}{2g},\alpha)$-regular.

Then there exists a quantum $K_3$-omni-booster $B$ for $A_0$ and $X$ in $G$ with quantum booster collection $(\mathcal{B}_F: F\in \mathcal{F}_0)$ such that 
\begin{enumerate}
    \item[(1)] $\Delta(B)\le a\Delta\cdot \log^{a} \Delta$, where $\Delta:=\max\left\{ \Delta(X),~\sqrt{n}\cdot \log n\right\}$,
    \item[(2)] for each $F\in \mathcal{F}_0$, ${\rm Disjoint}(\mc{B}_F) \cap {\rm HighGirth}_g(\mc{B}_F) \cap \cS \ne \emptyset$, and
    \item[(3)] ${\rm Proj}_g(B,A_0,K_n,X)$ is $(n,2\alpha n,\frac{1}{4g},2\alpha)$-regular. 
\end{enumerate} 
\end{thm}

We are now prepared to prove~\cref{cor:ExistenceOmniAbsorber} assuming~\cref{thm:HighGirthQuantumOmniBooster}.

\begin{proof}[Proof of~\cref{cor:ExistenceOmniAbsorber}]
By~\cref{thm:HighGirthQuantumOmniBooster}, there exists a quantum $K_3$-omni-booster $B$ for $A_0$ and $X$ with quantum booster collection $(\mathcal{B}_F: F\in \mc{F}_0)$ such that~\cref{thm:HighGirthQuantumOmniBooster}(1)--(3) hold. For each $F\in \mc{F}_0$, choose one $B_F\in {\rm Disjoint}(\mc{B}_F) \cap {\rm HighGirth}(\mc{B}_F)\cap\cS$, let $B':=\bigcup_{F\in \mc{F}} B_F$ and $\mc{B}':=(B_F: F\in \mc{F})$.

Since $B_F\in {\rm Disjoint}(\mc{B}_F)$ for all $F\in \mc{F}_0$, then $B'$ is  a $K_3$-omni-booster for $A_0$ and $X$ with booster family $\mc{B}'$. Furthermore since $B_F\in {\rm HighGirth}(\mc{B}_F)$ for all $F\in \mc{F}_0$, we obtain that $B'$ has collective girth at least $g$.

Let $(G_1,G_2,H):={\rm Proj}_g(B,A_0,K_3,X)$ and $(G_1',G_2',H'):={\rm Proj}_g(B',A_0,K_3,X)$. We let $G_1'':= G_1\cap {\rm Design}(K_n\setminus (X\cup A\cup B),K_3)$ , let $G_2'':= G_2[K_n\setminus (X\cup A\cup B)]$, and let $T:=(G_1'',G_2'',H)$. Note that for $i\in \{1,2\}$, $G_i''$ is a spanning subgraph of $G_i'$. Similarly since $\mc{M}(\mc{B}') \subseteq \mc{M}(\mc{B})$, we find that $H'[E(G_1'')\cup E(G_2'')]\subseteq H$. Hence $T$ is a subtreasury of ${\rm Proj}_g(B',A_0,K_3,X)$. 

We now show that $T$ is $(n,3\alpha n,\frac{1}{4g},2\alpha)$-regular.
By~\cref{thm:HighGirthQuantumOmniBooster}(3), we have that 
${\rm Proj}_g(B,A_0,K_3,X)$ is $(n,2\alpha n,\frac{1}{4g},2\alpha)$-regular. To conclude, it suffices to show that $d_{G_1}(v)-d_{G_1''}(v) \le \alpha n$ for all $v\in V(G_1'')$. However, this is true since 
$$\Delta(B)\le \Delta\cdot \log^a \Delta \le \frac{n}{\log^{ga} n}\cdot \log^a \Delta \le \frac{n}{\log^{(g-1)a}n},$$ 
and hence for all edges $e$ in $K_n\setminus B$, the number of triangles containing $e$ and edge of $B$ is at most $2 \Delta(B) \le \alpha n$, as desired with $a\ge 1$, $g\ge 2$ and $n$ is large enough. 

Let $A:= A_0\cup B'$. By~\cref{cor:ExistenceOmniAbsorber}(iii), we have $B'\subseteq \bigcup_{S\in\cS}S$ and~\cref{thm:HighGirthQuantumOmniBooster}(4) follows immediately. By Proposition~\ref{prop:CanonicalBoost}, we have that $A$ is a $K_3$-omni-absorber for $X$. By Proposition~\ref{prop:BoosterSubTreasury}, ${\rm Proj}_g(B',A_0,K_n,X)$ is a subtreasury of ${\rm Proj}_g(A,K_n,X)$. Thus it follows from the definition of subtreasury that $T$ is also a subtreasury of ${\rm Proj}_g(A,K_n,X)$, as desired.
\end{proof}

\subsection{Proof of Erd\H{o}s-Nash-Williams Booster Theorem}\label{sec:ProofBooster}

The proof of~\cref{thm:HighGirthQuantumOmniBooster} is almost immediate following the work of~\cite{DPII}. The main idea relies on taking a large quantum booster, i.e. containing a booster rooted at $F$ for every $F$ and every possible set of vertices in $K_n\setminus F$, and then taking a random subfamily. With positive probability, this quantum booster has the desired properties. 

For a real $p\in[0,1]$, the \emph{random $p$-sparsification} of a quantum booster $\mathcal{B} = (\mc{B}_F: F\in \mc{F})$ is the quantum booster $\mathcal{B}_p := ( (\mc{B}_F)_p: F\in \mc{F})$ where $(\mc{B}_F)_p$ denotes the subset of $\mc{B}_F$ obtained by choosing each element independently with probability $p$.  Given a absorber $A$ for $X$ with decomposition family $\cF$, we define a {\em full-quantum-$g$-sphere} for $A$ and $X$ to be a quantum $K_3$-omni-booster $B$ for $A$ and $X$ in $K_n$ with quantum booster collection $\mc{B}=(\mathcal{B}_F: F\in \mathcal{F})$ such that for every $F\in\cF$,
\begin{itemize}\itemsep.05in
    \item $\mathcal{B}_F$ is a family of $g$-spheres rooted at $F$,
    \item for every set of $(2g-1)$ vertices $S\subseteq V(K_n)$ disjoint from $V(F)$, there is exactly one  copy of a $g$-sphere rooted at $F$ in $\mc{B}_F$.
\end{itemize}

The following statement summarizes~\cite[Lemmas~4.12~and~4.13]{DPII}, rewriting it for our purpose, that is, for the graph and triangle case ($r=2$, $q=3$), using the specific $g$-sphere booster with a full-quantum-$g$-sphere, instead of a generic booster $B_0$ and full-$B_0$ quantum boosters. We remark that the exact statements of~\cite[Lemmas~4.12~and~4.13]{DPII} are slightly weaker, with fixed absolute constants instead of an arbitrarily small $\varepsilon$, and mentioning only the lower bound in (1), but this version can be trivially extracted from their proof. 

\begin{lem}\label{lem:DPII}
For every integers $C\ge 1$ and $g\ge 3$, reals $\alpha,\varepsilon\in(0,1)$, there exists an integer $a_0\ge 1$ such that for all $a\ge a_0$, there exists an integer $n_0\ge 1$ such that the following holds for all $n\ge n_0$. Let $X$ be a spanning subgraph of $K_n$ with $\Delta(X)\le \frac{n}{\log^{ag} n}$, and let $A$ be a $C$-refined $K_3$-omni-absorber for $X$ such that $\Delta(A)\le C\cdot \max\left\{ \Delta(X),~\sqrt{n}\cdot \log n\right\}$, and ${\rm Treasury}^g(K_n,K_n\setminus A,K_3,X)$ is $(n,\alpha n,\frac{1}{2g},\alpha)$-regular.

Let $B$ be a full-quantum-$g$-sphere for $A$ and $X$ with quantum booster collection $(\mathcal{B}_F: F\in \mathcal{F})$.
If $p :=\frac{\log^{a} n}{\binom{n}{2g-1}}$, then with probability at least $(1-\varepsilon)$, $\mc{B}_p$ satisfies all of the following properties
\begin{enumerate}\itemsep.05in
    \item[(1)] For each $F\in \mathcal{F}$, $|(\mathcal{B}_{F})_p| = (1\pm\varepsilon)\cdot \log^{a} n$. 
    \item[(2)] $\Delta(\bigcup \mc{B}_p)\le a\Delta\cdot \log^{a} \Delta$. 
    \item[(3)] For each $F\in \mathcal{F}$, $|(\mathcal{B}_{F})_p\setminus {\rm Disjoint}((\mathcal{B}_{F})_p)| \leq \varepsilon \cdot \log^{a} n$.
    \item[(4)] For each $F\in \mathcal{F}$, $|(\mathcal{B}_{F})_p\setminus {\rm HighGirth}_g((\mathcal{B}_{F})_p)| \leq \varepsilon \cdot \log^{a} n$.
    \item[(5)]${\rm Proj}_g(B,A,K_n,X)$ is $(n,2\alpha \cdot n,\frac{1}{4g},2\alpha)$-regular.
\end{enumerate}
\end{lem}

\begin{proof}[Proof of~\cref{thm:HighGirthQuantumOmniBooster}]

For each $F\in\cF$, let $\cS_F\subset \cS$, be the subset of $\cS$ containing the $g$-spheres rooted at $F$, and keeping at most one copy on each set of $2g-1$ vertices in $K_n\setminus V(F)$. Let $c'>0$ such that $|\cS_F|\geq c'\cdot\binom{n}{2g-1}$ for each $F\in\cF$, and fix $\varepsilon< c'/(4+c')$. Let $a>1$ such that~\cref{lem:DPII} holds. Let $B$ be a full-quantum-$g$-sphere for $A$ and $X$ with quantum booster collection $(\mathcal{B}_F: F\in \mathcal{F})$, such that $\cS_F\subseteq \cB_F$ for each $F\in\cF$.

Let $\mathcal{B}_p:= ( (\mc{B}_F)_p: F\in \mc{F})$ and $\mathcal{S}_p:= ( (\mc{S}_F)_p: F\in \mc{F})$ be the \emph{random $p$-sparsifications} obtained by choosing each element independently with probability $p:=\log^an/\binom{n}{2g-1}$.

Observe that by the Chernoff's bound, with probability at least $1-\varepsilon$, for every $F\in\cF$ we have $|\cS_F|\geq c'(1-\varepsilon)\cdot \log^an$. By~\cref{lem:DPII}, with probability at least $1-\varepsilon$, we have properties~\cref{lem:DPII}(2) and (5), and for every $F\in\cF$ we have $|(\mathcal{B}_{F})_p| \leq (1+\varepsilon)\cdot \log^{a} n$ and $|{\rm HighGirth}_g((\mathcal{B}_{F})_p)\cap{\rm Disjoint}_g((\mathcal{B}_{F})_p)|\geq (1-3\varepsilon)\log^a n$. Fix such an outcome. As $\varepsilon< c'/(4+c')$, we obtain that $c'(1-\varepsilon)+(1-3\varepsilon)>1+\varepsilon$, therefore ${\rm HighGirth}_g((\mathcal{B}_{F})_p)\cap{\rm Disjoint}_g((\mathcal{B}_{F})_p)\cap \cS_F\neq \emptyset$, hence $\cB_p$ is the desired quantum $K_3$-omni-booster. 
\end{proof}

\subsection{On a Generalization to \texorpdfstring{$K_q$}{Kq}-Omni-Absorbers}\label{app:Generalq}

We now briefly explain the modification required to proved a generalization of ~\cref{cor:ExistenceOmniAbsorber} to larger cliques, in order to prove~\cref{thm:GenralisationMain2q}.

We first extract the following result from~\cite{DPII}, an immediate generalization of~\cref{lem:DPII} to larger cliques. The generalization of full-quantum-$g$-sphere to larger cliques is called a {\em Full $B_0$-Type Omni-Booster}, see~\cite[Definition~4.10]{DPII}; if $B_0$ is a rooted $K_q$-booster, then a quantum $K_q$-omni-booster $B$ for $A$ and $X$ with quantum booster collection $(\mathcal{B}_F: F\in \mathcal{F})$ is \emph{$B_0$-type} if every $\mathcal{B}_F$ is a copy of $B_0$ with root $F$. We say $B$ is \emph{full} if for every set of $(v(B_0)-q)$ vertices $S\subseteq V(K_n)$ disjoint from $V(F)$, there is exactly one  copy of $B_0$ rooted at $F$ in $\mc{B}_F$.
We extract the following from~\cite[Lemmas~4.12~and~4.13]{DPII}, similarly to what was done for~\cref{lem:DPII}.

\begin{lem}\label{lem:DPIIgeneralq}
For every integers $q\geq 3$, $C\ge 1$, and $g\ge 3$, and reals $\alpha,\varepsilon\in(0,1)$, there exists an integer $a_0\ge 1$ such that for all $a\ge a_0$, there exists an integer $n_0\ge 1$ such that the following holds for all $n\ge n_0$. Let $X$ be a spanning subgraph of $K_n$ with $\Delta(X)\le \frac{n}{\log^{ag} n}$, and let $A$ be a $C$-refined $K_q$-omni-absorber for $X$ such that $\Delta(A)\le C\cdot \max\left\{ \Delta(X),~\sqrt{n}\cdot \log n\right\}$, and ${\rm Treasury}^g(K_n,K_n\setminus A,K_q,X)$ is $\big(\binom{n}{q-2},\alpha \binom{n}{q-2},\frac{1}{2g(q-2)},\alpha)$-regular.

Let $B_0$ be a rooted $K_q$-booster with rooted girth at least $g$, and $B$ be a full $B_0$-type $K_q$-omni-booster for $A$ and $X$ with quantum booster collection $(\mathcal{B}_F: F\in \mathcal{F})$.
If $p :=\frac{\log^{a} n}{\binom{n}{v(B_0)-q}}$, then with probability at least $(1-\varepsilon)$, $\mc{B}_p$ satisfies all of the following properties
\begin{enumerate}\itemsep.05in
    \item[(1)] For each $F\in \mathcal{F}$, $|(\mathcal{B}_{F})_p| = (1\pm\varepsilon)\cdot \log^{a} n$. 
    \item[(2)] $\Delta(\bigcup \mc{B}_p)\le a\Delta\cdot \log^{a} \Delta$. 
    \item[(3)] For each $F\in \mathcal{F}$, $|(\mathcal{B}_{F})_p\setminus {\rm Disjoint}((\mathcal{B}_{F})_p)| \leq \varepsilon \cdot \log^{a} n$.
    \item[(4)] For each $F\in \mathcal{F}$, $|(\mathcal{B}_{F})_p\setminus {\rm HighGirth}_g((\mathcal{B}_{F})_p)| \leq \varepsilon \cdot \log^{a} n$.
    \item[(5)]${\rm Proj}_g(B,A,K_n,X)$ is $\big(\binom{n}{q-2},2\alpha \cdot \binom{n}{q-2},\frac{1}{4g(q-2)},2\alpha)$-regular.
\end{enumerate}
\end{lem}

We can then use~\cref{lem:DPIIgeneralq}, to prove the following variant of~\cref{cor:ExistenceOmniAbsorber} for larger cliques, fixing a large family $\cS$ of a $B_0$ rooted-booster with girth $g$.

\begin{thm}\label{thm:HighGirthOmniAbsorberGeneralQ}
For every integers $q,g\ge 3$, and $C\geq 1$, and every reals $c\in(0,1)$ and $\alpha\in\big(0,\frac{1}{6(q-2)}\big)$, there exist integers $a,n_0\ge 1$ such that the following holds for all $n\ge n_0$.

Let $B_0$ be a rooted $K_q$-booster with rooted girth at least $g$. Let $X$ be a spanning subgraph of $K_n$ with $\Delta(X) \leq \frac{n}{\log^{ag}n}$ and such that ${\rm Treasury}^g(K_n,K_n,K_3,X)$ is $\big(\binom{n}{q-2},\alpha \binom{n}{q-2},\frac{1}{2g(q-2)},\alpha)$-regular. Let $\Delta:= \max\left\{ \Delta(X),~\sqrt{n}\cdot \log n\right\}$. 

Assume that there exists a $C$-refined $K_q$-omni-absorber $A_0$ for $X$ in $K_n$ such that $\Delta(A_0)\le C\cdot \Delta$, and a family $\cS$ of copies of $B_0$ in $K_n$ such that for every $F\cong K_q$ in $K_n$, there are at least $c\cdot n^{v(B_0)-q}$ distinct copies of $B_0$ rooted at $F$ in $\cS$.

Then there exists a $K_q$-omni-absorber $A$ for $X$ in $K_n$ with collective girth at least $g$, such that $A\subseteq A_0\cup\bigcup_{S\in\cS}S$, with $\Delta(A)\le a\Delta\cdot \log^{a} \Delta$, together with a subtreasury $T=(G_1,G_2,H)$ of ${\rm Proj}_g(A, K_n, X)$ that is
$\big(\binom{n}{q-2},6\alpha \binom{n}{q-2},\frac{1}{4g(q-2)},4\alpha)$-regular.
\end{thm}

\section{General Embedding Lemma}\label{sec:ProofEmbedding}

For completeness, we include the statement and proof of an extensive version of the General Embedding Lemma,~\cref{lem:EmbedGeneral}. We first require the following definitions for supergraph systems of hypergraphs.

\begin{definition}[Edge-intersecting $C$-bounded Supergraph System]
Let $\mc{H}$ be a family of subgraphs of a hypergraph $J$. A \emph{supergraph system} $\mc{W}$ for $\mc{H}$ is a family $(W_H : H\in \mc{H})$ where for each $H\in \mc{H}$, $W_H$ is a supergraph of $H$ with $V(W_H)\cap V(J)\neq\emptyset$ and for all $H'\ne H\in \mc{H}$, we have $V(W_H)\cap V(W_{H'}) \setminus V(J)= \emptyset$. We let $\bigcup \mc{W}$ denote $\bigcup_{H\in \mc{H}} W_H$ for brevity. For a real $C \geq 1$, we say that $\mc{W}$ is \emph{$C$-bounded} if $\max\{e(W_H),~v(W_H)\}\le C$ for all $H\in \mc{H}$. We say $\mc{W}$ is \emph{edge-intersecting} if for every $e\in W_H\setminus E(J)$ there exists $f\in J$ such that $e\cap V(J)\subseteq f$ for all $H\in \mc{H}$.
\end{definition}

\begin{definition}[Embedding a Supergraph System]
Let $J$ be a hypergraph and let $\mc{H}$ be a family of subgraphs of $J$. Let $\mc{W}$ be a supergraph system for $\mc{H}$. Let $G$ be a supergraph of $J$. An \emph{embedding} of $\mc{W}$ \emph{into} $G$ is a map $\phi : V(\bigcup \mc{W}) \hookrightarrow V(G)$ preserving edges such that $\phi(v)=v$ for all $v\in V(J)$. We let $\phi(\mc{W})$ denote $\bigcup_{e\in \bigcup W} \phi(e)$ (i.e.~the subgraph of $G$ corresponding to $\bigcup \mc{W}$).
\end{definition}


\begin{lem}[General Embedding]
For every integer $r\geq 2$ and reals $C>r$ and $\gamma\in (0,1]$, there exists $C'\ge 1$ such that the following hold. Let $J\subseteq G\subseteq K_n^r$ with $\Delta_{r-1}(J)\le \frac{n}{C'}$. Suppose that $\mc{H}$ is a $C$-refined family of subgraphs of $J$ and $\mc{W}$ is a $C$-bounded edge-intersecting supergraph system of $\mc{H}$. If for each $H\in \mc{H}$ there exist at least $\gamma\cdot n^{|V(W_H)\setminus V(H)|}$ embeddings of $W_H$ into $G\setminus (E(J)\setminus E(H))$,
then there exists an embedding $\phi$ of $\mc{W}$ into $G$ such that $\Delta(\phi(\mc{W})) \le C'\cdot \Delta(J)$.
\end{lem}

\begin{proof}
Let $C'$ be chosen large enough as needed throughout the proof but also such that $\frac{\sqrt{C'}}{C\cdot 2^r}$ is an integer. Note we assume that $\Delta_{r-1}(J)\ge 1$ as otherwise the outcome of the lemma holds trivially. Hence it follows that $n\ge C'$.
Choose $\mc{H}'\subseteq \mc{H}$ with an embedding $\phi' = (\phi'_H: H\in \mc{H}')$ of $\mc{W}':= (W_H:H\in \mc{H}')$ into $G$ as follows: Define $T':= \phi'(\mc{W}')$. For all $S\in \binom{V(G)}{r-1}$ and $R\subsetneq S$ define 
$$d_{T'}(S,R) := |\{ H\in \mc{H}': S\subseteq V(\phi'(W_H)), R=V(\phi'(W_H))\cap V(H),~\exists f\in \phi'(W_H\setminus H) \text{ with } S\subseteq f\}|.$$
Note that since $\mc{W}$ is edge-intersecting the existence of such an $f$ in the definition of $d_{T'}(S,R)$ implies that $\exists~e\in E(H) \text{ with } R\subseteq V(e)$. Furthermore, in the graph case of~\cref{lem:EmbedGeneral}, the existence of $e\in E(H) \text{ with } R\subseteq V(e)$ is trivially satisfied with $R=\emptyset$.
Now choose $\mc{H}'$ and $\phi'$ such that for all such $S$ and $R$,
$$d_{T'}(S,R)\le m:=\frac{\sqrt{C'}}{C\cdot 2^r}\cdot \Delta_{r-1}(J),$$
and subject to those conditions, $|\mc{H}'|$ is maximized. Note such a choice exists since $\mc{H}':=\emptyset$ satisfies the conditions trivially. Also note that $m$ is integral by assumption.

Before proceeding with the proof, we calculate that for $S\in \binom{V(G)}{r-1}$, we have
$$|T'(S)| 
\le C^2\cdot \Delta_{r-1}(J) + \sum_{R\subsetneq S} C\cdot d_{T'}(S,R) 
\le C^2\cdot \Delta_{r-1}(J) + C\cdot (2^r-1) \cdot \left(\frac{\sqrt{C'}}{C\cdot 2^r}\cdot \Delta_{r-1}(J)\right) 
\le \sqrt{C'}\cdot \Delta_{r-1}(J),$$
where we used that $\sqrt{C'}\ge C^2\cdot 2^r$ since $C'$ is large enough. Hence $\Delta(T')\le \sqrt{C'}\cdot \Delta_{r-1}(J) \le \frac{n}{\sqrt{C'}}$.

First suppose that $\mc{H}'=\mc{H}$ but then $\phi'$ is as desired. So we assume that $\mc{H}' \ne \mc{H}$. Hence there exists $H\in \mc{H}\setminus \mc{H}'$. Let $Z:= V(G)\setminus V(H)$. Let $b:= v(W_H)-v(H)$. Let $\Phi_H$ be the promised set of embeddings of $W_H$ into $G\setminus (E(J)\setminus E(H))$. Recall that $|\Phi_H|\ge \gamma\cdot n^b$ by assumption. \medskip

For $\ell \in [r]$, let 
$$A_{\ell} := \left\{U\in \binom{Z}{\ell}: \exists~e\in E(H) \text{ and } R\subseteq V(e) \text{ with } |R|=r-\ell \text{ such that } R\cup U \in T'\right\}.$$
Hence for $e\in E(H)$ and $R\subseteq V(e)$ with $|R|=r-\ell$, we find that 
\begin{align*}
\left|\left\{ U\in \binom{Z}{\ell}: R\cup U \in T'\right\}\right| &\le |T'(R)| \le \Delta(T')\cdot n^{\ell-1} \le \frac{n^{\ell}}{\sqrt{C'}}.
\end{align*}
Note there are at most $C\cdot \binom{r}{r-\ell}$ such subsets $R$ as $\mc{H}$ is $C$-refined.  Hence $|A_{\ell}| \le C\cdot \binom{r}{r-\ell} \cdot \frac{n^{\ell}}{\sqrt{C'}}$. Let $\mc{A}:= \{ \phi \in \Phi_H : \phi(W_H\setminus H)\cap T\ne \emptyset\}$. Hence
$$|\mc{A}| \le \sum_{\ell\in [r]} |A_{\ell}| \cdot n^{b-\ell} \cdot v(W_H)! \le \frac{2^r\cdot C!}{\sqrt{C'}} \cdot n^b,$$
where we used that there are most $v(W_H)!$ embeddings in $\Phi_H$ on the same vertex set (and that $v(W_H)\le C$) but also that $\sum_{\ell\in [r]} \binom{r}{r-\ell} \le 2^r$.\medskip

For $\ell \in [r-1]$, let
$$B_{\ell} := \left\{ U\in \binom{Z}{\ell}: \exists~e\in E(H) \text{ and } R \subseteq V(e) \text{ with } |R|=r-1-\ell \text{ such that } d_{T'}(R\cup U,R) =  m \right\}.$$ 
Note that for $e\in E(H)$ and $R\subseteq V(e)$ with $|R|=r-1-\ell$,
\begin{align*}
\left|\left\{ U\in \binom{Z}{\ell}: d_{T'}(R\cup U,R)= m\right\}\right| &\le \frac{\binom{C}{\ell}}{m} \cdot |\{H' \in \mc{H}': \exists~e'\in E(H') \text{ with } R\subseteq V(e')\}| \le \frac{\binom{C}{\ell}}{m}\cdot C\cdot \Delta_{r-1}(J)\cdot n^{\ell} \\
&\le \frac{C\cdot 2^r\cdot C\cdot \binom{C}{\ell}}{\sqrt{C'}} \cdot n^{\ell},
\end{align*}
where for the first inequality we counted the pairs $(U,\mc{H}')$ in two ways, while for the second inequality we used that there at most $\Delta_{r-1}(J)\cdot n^{\ell}$ choices of $e'$ and then at most $C$ choices of $H'$ as $\mc{H}$ as $C$-refined. Note there are at most $C\cdot \binom{r}{r-1-\ell}$ such subsets $R$ of $H$ as $\mc{H}$ is $C$-refined. Hence $|B_{\ell}| \le C\cdot \binom{r}{r-1-\ell} \cdot \frac{n^{\ell}}{\sqrt{C'}}$. Let $\mc{B}:= \{ \phi \in \Phi_H : \exists~U\in \bigcup_{\ell \in [r-1]} B_{\ell} \text{ with } U\subseteq V(\phi(W_H))\}$. Hence
$$|\mc{B}| \le \sum_{\ell \in [r-1]} |B_{\ell}| \cdot n^{b-\ell}\cdot v(W_H)! \le \frac{C^3\cdot C!\cdot (4C)^r}{\sqrt{C'}} \cdot n^b.$$ 
But then 
$$|\mc{A}|+|\mc{B}| < \gamma\cdot n^b \le |\Phi_H|,$$
since $C'$ is large enough. Hence there exists $\phi(W_H) \in \Phi_H \setminus (\mc{A}\cup \mc{B})$. Let $\mc{H}'':=\mc{H}'\cup \{H\}$ and $\phi'':=\phi'\cup \phi(W_H)$. Since $\phi(W_H)\not\in \mc{A}$, it follows that $\phi(W_H)$ is a subgraph of $G$ that is edge-disjoint from $T'$. Let $T'' := \phi''(\mc{W'}\cup \{W_H\})$.

We note that for all $S\in \binom{V(G)}{r-1}$ and $R\subsetneq S$, we have that $d_{T''}(R,U)=d_{T'}(R,U)$ unless $S\subseteq V(\phi(W_H))$, $R =V(\phi(W_H))\cap V(H)$ and there exists $f\in \phi(W_H)$ with $S\subseteq f$ and in that case we have $d_{T''}(S,R)= d_{T'}(S,R)+1$; in the former case, we have that $d_{T'}(S,R)\le m$ by assumption and in the latter case we find that $d_{T'}(S,R)\le m-1$ since $\phi_H \not\in \mc{B}$. It follows that in either case $d_{T''}(S,R)\le m$, contradicting the choice of $\mc{H}'$ and $\phi'$.
\end{proof}

\end{document}